\documentclass[oneside,english]{amsart}
\usepackage[T1]{fontenc}
\usepackage[utf8]{inputenc}
\usepackage{amsthm}
\usepackage{mathtools}
\usepackage{amssymb}
\usepackage[all]{xy}
\usepackage{lmodern}
\usepackage{mathrsfs}
\usepackage{hyperref}
\usepackage{framed}
\usepackage{color}
\usepackage{stmaryrd} 
\usepackage{tikz-cd}
\usepackage{babel}
\usepackage
{geometry}

\tikzset{%
  symbol/.style={
    draw=none,
    every to/.append style={
      edge node={node [sloped, allow upside down, auto=false]{$#1$}}
    },
  },
}

\DeclareFontFamily{OT1}{pzc}{}
\DeclareFontShape{OT1}{pzc}{m}{it}{<-> s * [1.10] pzcmi7t}{}
\DeclareMathAlphabet{\mathpzc}{OT1}{pzc}{m}{it}
 
\numberwithin{equation}{section}
\numberwithin{figure}{section}
\theoremstyle{plain}
\newtheorem{thm}{\protect\theoremname}[section]
\newtheorem{prop}[thm]{\protect\propositionname}
\newtheorem{lem}[thm]{\protect\lemmaname}
\theoremstyle{definition}
\newtheorem{defn}[thm]{\protect\definitionname}
\newtheorem{example}[thm]{\protect\examplename}
\newtheorem{cor}[thm]{\protect\corollaryname}
\theoremstyle{remark}
\newtheorem{rem}[thm]{\protect\remarkname}

\providecommand{\definitionname}{Definition}
\providecommand{\examplename}{Example}
\providecommand{\lemmaname}{Lemma}
\providecommand{\propositionname}{Proposition}
\providecommand{\remarkname}{Remark}
\providecommand{\corollaryname}{Corollary}
\providecommand{\theoremname}{Theorem}

\newcommand{\id}{\operatorname{id}}
\DeclareMathOperator{\im}{im}
\DeclareMathOperator{\colim}{colim}
\newcommand{\sgn}{\operatorname{sgn}}

\newcommand{\As}{\mathrm{As}}
\newcommand{\Ai}{\mathrm{A}_{\infty}}
\newcommand{\uAs}{\mathrm{uAs}}
\newcommand{\Com}{\mathrm{Com}}

\newcommand{\cAs}{\mathrm{cAs}}
\newcommand{\cuAs}{\mathrm{cuAs}}
\newcommand{\cuAi}{\mathrm{cuA}_\infty}

\newcommand{\B}{\hat{\mathrm{B}}}
\newcommand{\Bm}{\mathrm{B}}
\newcommand{\Om}{\hat{\Omega}}
\newcommand{\coeq}{\textrm{coeq}}

\newcommand{\Cx}{\mathrm{Cx}}
\newcommand{\cCx}{\mathrm{cCx}}
\newcommand{\Cxi}{\mathrm{Cx}_{\infty}}
\newcommand{\cCxi}{\mathrm{cCx}_{\infty}}
\newcommand{\cLie}{\mathrm{cLie}}
\newcommand{\scLie}{\mathpzc{s}\mathrm{cLie}}
\newcommand{\cLi}{\mathrm{cL_\infty}}
\newcommand{\scLi}{\mathpzc{s}\mathrm{cL_\infty}}

\newcommand{\scLa}{\mathpzc{s}\mathrm{cLie}^{\text{!`}}}
\newcommand{\Der}{\mathrm{Der}}
\newcommand{\End}{\mathrm{End}}

\newcommand{\Hom}{\mathrm{Hom}}
\newcommand{\homf}{\hom}

\newcommand{\homa}{\hom^{\alpha}}

\newcommand{\mor}{\mathrm{mor}}

\newcommand{\Li}{\mathrm{L}_{\infty}}
\newcommand{\Lie}{\mathrm{Lie}}
\newcommand{\sLie}{\mathpzc{s}\mathrm{Lie}}
\newcommand{\sAs}{\mathpzc{s}\mathrm{As}}

\newcommand{\sLi}{\mathpzc{s}\mathrm{L_\infty}}
\newcommand{\Sym}{\mathrm{Sym}}

\newcommand{\antishriek}{\text{!`}}
\newcommand{\Cxa}{\mathrm{Cx}^{\text{!`}}}
\newcommand{\cCxa}{\mathrm{cCx}^{\text{!`}}}
\newcommand{\q}{\mathrm{q}}

\newcommand{\Gr}{\mathrm{Gr}}
\newcommand{\Tw}{\mathrm{Tw}}
\newcommand{\Prim}{\mathrm{Prim}}

\newcommand{\Bc}{\mathcal{B}}
\newcommand{\Cc}{\mathcal{C}}

\newcommand{\Dc}{\mathcal{D}}

\renewcommand{\Mc}{\mathcal{M}}

\newcommand{\Pc}{\mathcal{P}}
\newcommand{\Qc}{\mathcal{Q}}

\newcommand{\Rc}{\mathcal{R}}
\newcommand{\Sc}{\mathcal{S}}

\newcommand{\Tc}{\mathcal{T}}
\newcommand{\Wc}{\mathcal{W}}
\newcommand{\Zc}{\mathcal{Z}}
\newcommand{\cfree}{\mathsf{c}\mathcal{T}}

\newcommand{\Rb}{\mathbb{R}}
\newcommand{\Cb}{\mathbb{C}}
\newcommand{\Kb}{\mathbb{K}}
\newcommand{\Lb}{\mathbb{L}}
\newcommand{\Nb}{\mathbb{N}}
\newcommand{\Sb}{\mathbb{S}}
\newcommand{\Qb}{\mathbb{Q}}
\newcommand{\Zb}{\mathbb{Z}}

\newcommand{\Rf}{\mathbf{R}}
\newcommand{\Kf}{\mathbf{K}}

\newcommand{\fk}{\mathfrak{f}}
\newcommand{\Fk}{\mathfrak{F}}

\newcommand{\cOp}{\mathsf{cOp}}
\newcommand{\cOpl}{\mathsf{cOp}^{\mathrm{lax}}}

\newcommand{\coideal}[1]{\widetilde{#1}}

\DeclareMathOperator{\Filt}{Filt}
\DeclareMathOperator{\coim}{coim}
\DeclareMathOperator{\coker}{coker}

\newcommand{\Kmod}{\mathsf{Mod}_{\Kf}}
\newcommand{\Rmod}{\mathsf{Mod}_{\Rf}}
\newcommand{\Rbmod}{\mathsf{Mod}_{\Rb}}
\newcommand{\ModA}{\mathsf{Mod}_\Rf}
\newcommand{\ModAg}{\ModA^{\mathrm{g}}}
\newcommand{\ModAgr}{\ModA^{\mathrm{gr}}}

\newcommand{\ModKgr}{\Kmod^{\mathrm{gr}}}
\newcommand{\ModRgr}{\Rmod^{\mathrm{gr}}}

\newcommand{\ModKpg}{\Kmod^{\mathrm{pg}}}
\newcommand{\ModApg}{\ModA^{\mathrm{pg}}}

\newcommand{\SMod}{\Sb\textrm{-}\mathsf{Mod}}

\newcommand{\complete}{\widehat{\Filt}}

\newcommand{\grModR}{\Rmod^{\mathsf{gr}}}
\newcommand{\grModRb}{\Rbmod^{\mathsf{gr}}}

\newcommand{\ModAP}{\mathsf{Mod}_A^{\Pc}}
\newcommand{\ModAPz}{\mathsf{Mod}_A^{\Pc, 0}}
\newcommand{\ModBP}{\mathsf{Mod}_B^{\Pc}}
\newcommand{\ModBPz}{\mathsf{Mod}_B^{\Pc, 0}}

\newcommand{\AbPA}{\mathsf{Ab(\Pc\text{-}alg}/A)}
\newcommand{\Palg}{\Pc\text{-}\mathsf{alg}}
\newcommand{\PalgA}{\Pc\text{-}\mathsf{alg}/A}
\newcommand{\Calg}{\Cc\text{-}\mathsf{coalg}}
\newcommand{\Ho}{\mathsf{Ho}}

\newcommand{\hotimes}{\hat\otimes}
\newcommand{\hoplus}{\hat\oplus}
\newcommand{\hcirc}{\hat\circ}

\newcommand{\CR}{\mathfrak{cr}}

\newcommand{\curvA}{\vcenter{
\xymatrix@M=0pt@R=3pt@C=3pt{&&\\
&&\\
& *{} \ar@{{*}}[u] \ar@{-}[dd]&\\
&&\\
&&}}}

\newcommand{\idCurvA}{\vcenter{
\xymatrix@M=0pt@R=3pt@C=3pt{&&\\
& *{} \ar@{{*}}[u] \ar@{-}[dddd]&\\
\ar@{.}@<-1.5pt>[rr] &&\\
&&\\ &&\\ &&}}}

\newcommand{\curvAd}{\vcenter{
\xymatrix@M=0pt@R=2.4pt@C=3pt{&&\\
& {\scriptstyle \circ} &\\
& *{} \ar@{-}[u] \ar@{-}[d]&\\
&&}}}

\newcommand{\as}{\vcenter{
\xymatrix@M=0pt@R=4pt@C=4pt{&&\\
& *{} \ar@{-}[lu] \ar@{-}[ur] \ar@{-}[d] &\\
&&}}}

\newcommand{\ass}{\vcenter{ \xymatrix@M=0pt@R=4pt@C=4pt{
      &&&&\\ \ar@{-}[ddrr] && \ar@{-}[dl] && \ar@{-}[ddll]\\
      &&&& \\
      && *{} \ar@{-}[d] &&\\
      &&}}
-
\vcenter{
    \xymatrix@M=0pt@R=4pt@C=4pt{&&&&\\
      \ar@{-}[drdr] && \ar@{-}[dr] && \ar@{-}[dldl]\\
      &&&& \\
      && *{}\ar@{-}[d] && \\
      &&&&}}}

\newcommand{\assu}{\vcenter{ \xymatrix@M=0pt@R=4pt@C=4pt{
      &&&&\\ \ar@{-}[ddrr] && \ar@{-}[dl] && \ar@{-}[ddll]\\
      &&&& \\
      && *{} \ar@{-}[d] &&\\
      &&}}
+
\vcenter{
    \xymatrix@M=0pt@R=4pt@C=4pt{&&&&\\
      \ar@{-}[drdr] && \ar@{-}[dr] && \ar@{-}[dldl]\\
      &&&& \\
      && *{}\ar@{-}[d] && \\
      &&&&}}}

\newcommand{\lunit}{\vcenter{ \xymatrix@M=0pt@R=4pt@C=4pt{
           & & & \\
           & {\scriptscriptstyle \circ} \ar@{-}[d] & &\\
           & \ar@{-}[dr] & & \ar@{-}[dl] \\
           & & *{} \ar@{-}[d] &\\
           & & }}}

\newcommand{\runit}{\vcenter{
         \xymatrix@M=0pt@R=4pt@C=4pt{& & & \\
           & & & {\scriptscriptstyle \circ} \\
           & \ar@{-}[dr] & & \ar@{-}[u] \ar@{-}[dl]\\
           & & *{} \ar@{-}[d] & \\
           & & }}}

\newcommand{\curvAs}{\vcenter{ \xymatrix@M=0pt@R=4pt@C=4pt{
      &&&&\\
      &&&&\\
      & \ar@{{*}}[u] \ar@{-}[dr] && \ar@{-}[dl] & \\
      && *{} \ar@{-}[d] &&\\
      &&}}
-
\vcenter{\xymatrix@M=0pt@R=4pt@C=4pt{
      &&&&\\
      &&&&\\
      & \ar@{-}[dr] && \ar@{{*}}[u] \ar@{-}[dl] &\\
      && *{} \ar@{-}[d] &&\\
      &&}}}

\newcommand{\jn}{{\tiny \vcenter{\xymatrix@M=0pt@R=4pt@C=4pt{& 1 &\\ & \ar@{-}[dd] &\\ & \bullet & \\ &&}}}}

\newcommand{\slie}{{\tiny \vcenter{\xymatrix@M=0pt@R=3pt@C=2pt{1 &&& 2\\ &&& \\ && *{} \ar@{-}[llu] \ar@{-}[d] \ar@{-}[ru] & s^{-1} \\ &&}}}}

\newcommand{\lie}{{\tiny \vcenter{\xymatrix@M=0pt@R=4pt@C=4pt{1 && 2\\ && \\ & *{} \ar@{-}[lu] \ar@{-}[d] \ar@{-}[ru] & \\ &&}}}}

\newcommand{\astwit}{{\tiny \vcenter{\xymatrix@M=0pt@R=4pt@C=4pt{2 && 1\\ && \\ & *{} \ar@{-}[lu] \ar@{-}[d] \ar@{-}[ru] & \\ &&}}}}

\newcommand{\jnr}{{\tiny \vcenter{\xymatrix@M=0pt@R=3pt@C=2pt{& 1 &\\ & \ar@{-}[dddd] &\\ && \\ & \ar@{{*}}[u] \ar@{{*}}[d] & \\ && \\ &&}} + \vcenter{\xymatrix@M=0pt@R=4pt@C=4pt{& 1 &\\ & \ar@{-}[dddd] &\\ && \\ && \\ && \\ &&}}}}

\newcommand{\qjnr}{{\tiny \vcenter{\xymatrix@M=0pt@R=3pt@C=2pt{& 1 &\\ & \ar@{-}[dddd] &\\ && \\ & \ar@{{*}}[u] \ar@{{*}}[d] & \\ && \\ &&}}}}

\newcommand{\unit}{{\tiny \vcenter{\xymatrix@M=0pt@R=3pt@C=2pt{& 1 &\\ & \ar@{-}[dd] &\\ && \\ &&}}}}

\newcommand{\sjlr}{{\tiny \vcenter{\xymatrix@M=0pt@R=3pt@C=2pt{1 &&& 2\\ \ar@{-}[dd] &&&\\ &&&\\ \ar@{{*}}[u] &&& \\ && *{} \ar@{-}[ull] \ar@{-}[d] \ar@{-}[ur] & s^{-1}\\ &&&}}}}

\newcommand{\sljr}{{\tiny \vcenter{\xymatrix@M=0pt@R=3pt@C=2pt{1 &&& 2\\ \ar@{-}[drr] &&& \ar@{-}[ld]\\ && *{} \ar@{-}[dd] & s^{-1} \\ &&&\\ && \ar@{{*}}[u] &}}}}

\newcommand{\jac}{{\tiny \vcenter{\xymatrix@M=0pt@R=3pt@C=3pt{1 && 2 && 3\\ &&&& \\ & \ar@{-}[lu] \ar@{-}[dr] \ar@{-}[ru] &&& \\ && *{} \ar@{-}[d] \ar@{-}[urur] &&\\ &&&&}} + \vcenter{\xymatrix@M=0pt@R=3pt@C=3pt{2 && 3 && 1\\ &&&& \\ & \ar@{-}[lu] \ar@{-}[dr] \ar@{-}[ru] &&& \\ && *{} \ar@{-}[d] \ar@{-}[urur] &&\\ &&&&}} + \vcenter{\xymatrix@M=0pt@R=3pt@C=3pt{3 && 1 && 2\\ &&&& \\ & \ar@{-}[lu] \ar@{-}[dr] \ar@{-}[ru] &&& \\ && *{} \ar@{-}[d] \ar@{-}[urur] &&\\ &&&&}}}}

\newcommand{\curvsLie}{\tiny \vcenter{ \xymatrix@M=0pt@R=3pt@C=2pt{
      &&&&\\
      &&&&\\
      \ar@{{*}}[uu] &&&& \\
      && *{} \ar@{-}[ull] \ar@{-}[d] \ar@{-}[ur] & s^{-1} &\\
      &&}}}

\newcommand{\curvLie}{\tiny \vcenter{ \xymatrix@M=0pt@R=4pt@C=4pt{
      &&&\\
      &&&\\
      \ar@{{*}}[uu] &&& \\
      & *{} \ar@{-}[ul] \ar@{-}[d] \ar@{-}[ur] &&\\
      &&}}}

\newcommand{\curvBra}{\tiny \vcenter{ \xymatrix@M=0pt@R=4pt@C=4pt{
      &&&&\\
      &&&&\\
      \ar@{{*}}[uu] &&& \ar@{{*}}[uu] & \\
      && *{} \ar@{-}[ull] \ar@{-}[d] \ar@{-}[ur] & s^{-1} &\\
      &&}}}

\newcommand{\qJc}{\tiny \vcenter{\xymatrix@M=0pt@R=4pt@C=4pt{& 1\\ & \ar@{-}[d] &\\ & \circ & \\ & \ar@{-}[u] &}}}

\newcommand{\jnd}{{\tiny \vcenter{\xymatrix@M=0pt@R=6pt@C=4pt{& 1 &\\ & \ar@{-}[dd] &\\ && \\ & \ar@{{o}}[u] &}}}}

\newcommand{\scom}{{\tiny \vcenter{\xymatrix@M=0pt@R=4pt@C=4pt{1 && 2\\ && \\ & *{} \ar@{-}[lu] \ar@{-}[d] \ar@{-}[ru] & s\\ &&}}}}
  
\newcommand{\sojn}{{\tiny \vcenter{\xymatrix@M=0pt@R=6pt@C=4pt{& 1 &\\ & \ar@{-}[d] &\\ & \circ & \\ & \ar@{-}[u] &}}}}

\newcommand{\soqjnr}{{\tiny \vcenter{\xymatrix@M=0pt@R=4pt@C=4pt{& 1 &\\ & \ar@{-}[d] &\\ & \circ & \\ & *{} \ar@{-}[u] \ar@{-}[d] & \\ & \circ \ar@{-}[d] & \\ &&}}}}
 
\newcommand{\sojlr}{{\tiny \vcenter{\xymatrix@M=0pt@R=4pt@C=4pt{1 && 2\\ \ar@{-}[d] &&\\ \circ &&\\ \ar@{-}[u] && \\ & *{} \ar@{-}[ul] \ar@{-}[d] \ar@{-}[ur] &\\ &&}}}}

\newcommand{\soljr}{{\tiny \vcenter{\xymatrix@M=0pt@R=4pt@C=4pt{1 && 2\\ \ar@{-}[dr] && \ar@{-}[ld]\\ & *{} \ar@{-}[d] &\\ & \circ &\\ & \ar@{-}[u] &}}}}

\newcommand{\der}{{\tiny \vcenter{\xymatrix@M=0pt@R=4pt@C=4pt{1 && 2\\ \ar@{-}[dr] && \ar@{-}[ld]\\ & *{} \ar@{-}[d] &\\ & \circ &\\ & \ar@{-}[u] &}}} - {\tiny \vcenter{\xymatrix@M=0pt@R=4pt@C=4pt{1 && 2\\ \ar@{-}[d] &&\\ \circ &&\\ \ar@{-}[u] && \\ & *{} \ar@{-}[ul] \ar@{-}[d] \ar@{-}[ur] &\\ &&}}} - {\tiny \vcenter{\xymatrix@M=0pt@R=4pt@C=4pt{1 && 2\\ && \ar@{-}[d] \\ && \circ \\ && \ar@{-}[u] \\ & *{} \ar@{-}[ul] \ar@{-}[d] \ar@{-}[ur] &\\ &&}}}}

\urldef\joanhomepage\url{https://www.math.univ-toulouse.fr/~jmilles}
\urldef\sinanhomepage\url{https://sites.google.com/site/sinanyalinshomepage/home}
\newcommand{\mailurl}[1]{\email{\href{mailto:#1}{#1}}}

\title{André--Quillen cohomology in the context of curved algebras}

\author{Joan Bellier-Mill\`{e}s}
\address{B.-M.: Institut de mathématiques de Toulouse ; UMR5219\\ Université de Toulouse ; CNRS ; UPS, F-31062 Toulouse Cedex 9, France}
\mailurl{joan.milles@math.univ-toulouse.fr}
\urladdr{\joanhomepage} 

\author{Sinan Yalin}
\address{Y.: Laboratoire angevin de recherches en Mathématiques (LAREMA) ; UMR6093\\ Université d'Angers ; CNRS ; 2 boulevard Lavoisier, 49000 Angers, France}
\mailurl{yalinprop@gmail.com}
\urladdr{\sinanhomepage}

\thanks{B.-M. and Y. were supported by ANR-17-CE40-0014 CatAG; 
B.-M. was supported by ANR-20-CE40-0016 HighAGT;
Y. was supported by ANR-16-CE40-0003 ChroK, funded by Agence Nationale pour la Recherche.}

\begin{document}

\begin{abstract}
The André--Quillen cohomology of an algebra with coefficients in a module is defined by deriving a functor based on Kähler differential forms. It can be computed using a cofibrant resolution of the algebra in a model category structure where weak equivalences are quasi-isomorphisms. This construction works for algebras over an operad, providing a cohomology theory tailored for each type of algebra.

For curved algebras however, the notion of quasi-isomorphism is meaningless. The occurrence and importance of curved structures in various research topics (symplectic topology, deformation theory, derived geometry, mathematical physics...) motivate the development of their homotopy theory and André--Quillen cohomology theory. To get a homotopical context with an appropriate notion of weak equivalence, we consider filtered complete modules with a predifferential inducing a differential on the associated graded. Curved algebras in such modules are algebras over a curved operad.

In this article, we consider curved operads which are not necessarily augmented. Bar and cobar constructions adapted to these curved operads are developed, as well as Koszul duality theory. Consequently, we obtain homotopy versions of our curved algebras and make it explicit for interesting cases. Two main examples are the curved operads encoding curved unital associative algebras and curved complex Lie algebras. In particular, homotopy curved unital associative algebras describe the structure of Floer complexes of lagrangian submanifolds and Fukaya categories in symplectic topology. Bar and cobar constructions for curved algebras are also developed, and we obtain resolutions from which we compute their André–Quillen cohomology with module coefficients.

Our computations of André-Quillen cohomology in the case of curved complex Lie algebras reveal an interesting way to implement integrable almost complex structures in derived analytic geometry.
\end{abstract}

\maketitle

\tableofcontents{}

\section*{Introduction}

\subsection*{Motivations and antecedents}

In this article, our goal is to further develop the homotopical and homological theory of curved algebras in continuity with \cite{BMDC20}. 
A first example of algebras is associative algebras. 
They are made of an underlying vector space or module and a multiplication that satisfies the associativity relation. 
The study of extensions and deformations of an associative algebra leads to the definition of the Hochschild cochains and cohomology. 
We can enrich the notion of associative algebra by adding a unit, by assuming that the underlying object is a chain complex or even by relaxing the associativity relation ($\Ai$-algebras). 
The definition of the Hochschild cohomology extends to these settings.

It happens that an associative algebra $(A, \cdot)$ may have a derivation $d$ whose square is not zero but is equal to the bracket with an element $\theta$ in $A$, called \emph{curvature}: $d^2 = [\theta, -] = (\theta \cdot -) - (- \cdot \theta)$. 
Examples can be built by means of a vector bundle endowed with a connection (and the curvature coincides with the curvature of the connection) or can be obtained from a matrix factorisation (nevertheless,  in this example, $\theta$ is central so $d^2=0=[\theta, -]$).
Defining Hochschild cohomology in this setting poses some problems as for example the fact that it vanishes for curved algebras built from matrix factorisations, see Theorem 4.2 in \cite{CT13}. 
Other definitions are proposed there but the relation with Hochschild cohomology is not well-understood. 
A different direction has been followed in \cite{FOOO07} where the authors assume that the algebraic objects are filtered and complete. 
This is also the context used in \cite{BMDC20}.
Let us note also that curved $A_{\infty}$-structures appear naturally as deformations of an $A_{\infty}$-algebra or an $A_{\infty}$-category, in the formal moduli problem controlled by its (full) Hochschild complex (see for example \cite{jvdK23} and \cite{BL15}).

The first cohomology theories of algebras (associative and Lie algebras) were defined by describing explicit cochain complexes. 
Later, in order to study commutative algebras, Quillen \cite{dQ70} and André \cite{mA74} have proposed a cohomology theory described by means of a derived functor. 
These ideas have been applied to many types of algebras using the notion of operads, and widely studied under the name of André--Quillen cohomology (\cite{GH00, jM11, HPN18} for example). 
Following this guideline, the article \cite{BMDC20} presents model category structures on a base category of \emph{gr-dg modules}, on curved operads and on algebras over a curved operad (that we consider as curved algebras in a wide sense). 
In order to compute the cotangent complex which appears in the definition of the André--Quillen cohomology, we need cofibrant resolutions. 

Part of this effort has been made in \cite{BMDC20}. 
However, in the present article, we are concerned in a more general situation than the one presented there. 
While the curved operads studied in the aforementioned article are augmented, we are interested with curved operads that are not. 
Two interesting examples of non-augmented curved operads are the one encoding curved unital associative algebras ($\cuAs$) and the one encoding curved complex Lie algebras ($\cCx$). Beware that in the second example, the operad is defined over $\mathbb{R}$ and not over $\mathbb{C}$ (so it is not the $\mathbb{C}$-linear operad parametrising complex Lie algebras).
The related operad $\Cx$ encoding (uncurved) complex Lie algebras appears in the study of formal complex structures as integrable almost complex structures (see \cite{jM14}). 
In order to define meaningful homotopy versions of these algebras, we  extend the bar and cobar constructions and the Koszul duality theory in this context. 
We then provide cofibrant resolutions of algebras over such curved operads and we use them to compute their André--Quillen cohomology.

One of our motivations for these, sometimes technical, generalisations, comes from symplectic topology. Given an ambient symplectic manifold $M$, the Floer complex of a lagrangian submanifold inherits the structure of a filtered homotopy $\cuAs$-algebra, whose operations encompass the counting of pseudo-holomorphic disks in $M$ bordered by the lagrangian as well as pseudo-holomorphic polygons. There is a relative version defined for a pair of lagrangians, which acquires the structure of a filtered $A_{\infty}$-bimodule. At a categorical level, the Fukaya category of $M$ consists in lagrangian submanifolds with Floer complexes as morphisms and forms a curved unital $A_{\infty}$-category \cite{FOOO07}. Together with Lagrangian Floer cohomology, these are central objects of study in symplectic topology and are closely related to homological mirror symmetry. In particular, their deformation and obstruction theory has still to be well understood.

Another motivation comes from derived complex analytic geometry, which has experienced a striking development during the past few years relying in particular on the works of Lurie, Porta and Porta-Yu \cite{jLDAGIX, mP16, PY20}. Indeed, complex structures on a formal manifold are examples of homotopy $\cCx$-algebras (see \cite{jM14} for its uncurved version $\Cx$). The peculiarity of $\Cx$ is that it is an $\mathbb{R}$-linear (not $\mathbb{C}$-linear) operad parametrising complex structures as integrable almost complex structures, which offers an explicit way to set up a homotopy coherent notion of integrability suitable for \emph{derived} geometric objects.

As detailled below, the current literature does not allow for an operadic treatment for theses two motivations.

In relation with the André--Quillen cohomology, we prove that there is a fully faithful functor from formal derived complex structures to EFC-dgas (dg algebras with an entire functional calculus), defined in \cite{Pri20} as the building blocks of a dg model for derived analytic geometry. In \cite{Pri20}, derived analytic spaces locally structured by these algebras are proved to be equivalent to the derived complex analytic spaces of \cite{jLDAGIX, mP16} whose $0$-truncation is a complex analytic space.

We use this bridge between operadic homotopical algebra and derived complex analytic spaces in an ongoing work devoted to integrability in derived complex geometry. We say a few more words about this perspective at the end of the introduction.

\subsection*{Homotopy algebras in the curved context}

Given a type of algebra described by a curved operad $\Pc$, homotopy $\Pc$-algebras are algebras over an ($\Sb$-)cofibrant resolution $\Pc_\infty$ of $\Pc$. 
When the curved operad $\Pc$ is augmented, two classical resolutions are given by the bar-cobar resolution and the Koszul resolution. 
In the present article, our curved operads are not assumed to be augmented. 
A bar construction for (not necessarily augmented) curved operads and a cobar construction for (not necessarily coaugmented) curved cooperads already exist in \cite{vL14}. 
However they don't form the adjunction that is required to describe the bar-cobar resolution (see Propositions 6.1.2 and 6.1.3 in \cite{vL14}). 
We provide in this article similar constructions adapted to the gr-dg context. Our bar and cobar constructions represent (on the right and on the left respectively) a bifunctor of operadic curved twisting morphisms. These operadic curved twisting morphisms are maps from a curved cooperad $\Cc$ to a curved operad $\Pc$, $\alpha : \Cc \to \Pc$, of homological degree $-1$ solutions to a curved Maurer--Cartan equation $\Theta + \partial \alpha + \frac12 [\alpha, \alpha] = 0$ where $\Theta$ is a curvature term. 
In order to obtain this adjunction, we have to consider (as in \cite{vL14}) lax notions of morphisms between curved operads and between curved cooperads. The counit of the bar-cobar adjunction provides then a functorial cofibrant resolution. To summarize our main results of this part, we have:
\begin{thm}[Theorems \ref{thm: bar-cobar adjunction} and \ref{thm: bar-cobar resolution}]
Let $\Pc$ be a curved semi-augmented operad and $\Cc$ be a curved altipotent cooperad.
\begin{itemize}
\item[(1)] The bar construction $\B$ and cobar construction $\hat\Omega$ define an adjunction between curved sa operads and curved altipotent cooperads with lax morphisms, and the adjunction relation factors naturally through twisting morphisms
\[\Hom_{\mathsf{curv.\, sa\, op.}^{\mathrm{lax}}}(\hat\Omega \Cc,\Pc)
\cong
\Tw(\Cc, \Pc)
\cong
\Hom_{\mathsf{curv.\, alti.\, coop.}^{\mathrm{lax}}}(\Cc,\B\Pc).
\]

\item[(2)] Suppose that $\Gr \Pc$ is a connected bounded below weight filtered operad. Then the counit of the adjunction
\[
\hat\Omega \B \Pc \to \Pc
\]
is a graded quasi-isomorphism and a strict surjection.

\item[(3)] Under good assumptions on the filtration and the curvature of the curved operad $\Pc$, this counit  forms a functorial cofibrant resolution in the model category structure on curved operads described in \cite[Appendix C]{BMDC20}.
\end{itemize}
\end{thm}

To obtain a smaller resolution of our curved operads, we develop a curved Koszul duality adapted to some inhomogeneous quadratic curved operads. 
In good situations, this produces a resolution of the curved operad which is cofibrant in the underlying category of $\Sb$-modules ($\Sb$-cofibrant resolution):
\begin{thm}[Theorem \ref{thm: Koszul resolution} ]
\begin{itemize}
\item[(1)] Let $\Pc$ be a Koszul inhomogeneous quadratic curved operad such that $\Gr\q\Pc$ is connected bounded below weight graded. 
The curved operad morphism
\[
f_\kappa : \hat\Omega \Pc^\antishriek \to \Pc
\]
is a graded quasi-isomorphism and a strict surjection.

\item[(2)] Under good assumptions on the filtration and the curvature of the curved operad $\Pc$, the Koszul resolution $\hat\Omega \Pc^\antishriek$ is  $\Sb$-cofibrant in the model category structure on gr-dg $\Sb$-modules described in \cite[Appendix C]{BMDC20}.
\end{itemize}
\end{thm}
\begin{rem}
By Theorem C.45 in \cite{BMDC20}, the model categories of $\Om \Pc^\antishriek$-algebras and of $\Om \B \Pc$-algebras are both equivalent to the model category of $\Pc$-algebras (because all operads are $\Sb$-split in characteristic 0). 
Consequently, up to equivalence, the $\infty$-category of homotopy $\Pc$-algebras does not depend on a choice of resolution, similarly to what happens for uncurved algebras over operads.
\end{rem}

Our two main examples of interest are the curved operad $\cuAs$ encoding \emph{curved unital associative algebras} and the curved operad $\cCx$ encoding \emph{curved complex Lie algebras} as algebras over an {$\mathbb{R}$-linear} operad.  Both satisfies the conditions of the Theorems above, by Theorems \ref{thm: Koszul dual cuAs} and \ref{thm: cCxa}.
We describe the $\Sb$-cofibrant resolutions obtained by means of the curved Koszul duality and denote them by $\cuAi$ and $\cCxi$ respectively.

\subsection*{Resolutions of curved algebras}

As for curved operads, we construct a cofibrant resolution of algebras by means of the bar-cobar adjunction. 
Given an algebra $A$ over a curved (semi-augmented) operad $\Pc$, a coalgebra $C$ over a curved (altipotent) cooperad $\Cc$ and an operadic curved twisting morphism $\alpha : \Cc \to \Pc$, we define a curved $\Li$-algebra structure on the module of maps $\hom(C, A)$ following ideas given in \cite{fW19}. 
It is defined by means of the following results:
\begin{itemize}
\item
Proposition \ref{prop: algebra over the convolution curved operad} says that $\hom(C,  A)$ is an algebra over the curved operad $\Hom(\Cc,  \Pc)$;
\item
Proposition \ref{prop: twisting morphisms as morphisms} ensures that there is a bijection
\[
\Tw(\Cc, \Pc) \cong \Hom_{\mathsf{curv.\, sa\, op.}^{\mathrm{lax}}}(\scLi, \Hom(\Cc, \Pc))
\]
where $\scLi$ is the curved operad encoding shifted curved Lie algebras;
\item
and Proposition \ref{prop: equivalent definition of homotopie cLie algebras} makes explicit the map $\scLi \to \End_{\hom(C, A)}$ obtained from $\alpha$ as a collection of maps $\{ l_n^\alpha\}_{n \geq 0}$.
\end{itemize}
Algebraic twisting morphisms with respect to $\alpha$ are then degree 0 maps $\varphi : C \to A$ solutions to the Maurer--Cartan equation
\[
\sum_{n \geq 0} \frac1{n!} l_n^\alpha(\varphi, \ldots, \varphi) = 0.
\]
We use these algebraic twisting morphisms as follows:
\begin{thm}
\begin{itemize}
\item[(1)] The bifunctor of algebraic twisting morphisms with respect to $\alpha$ can be represented on the left by the cobar functor $\Om_\alpha$ and on the right by the bar functor $\B_\alpha$.

\item[(2)] Under good assumptions, the counit of this adjunction $\Om_\alpha \B_\alpha A \to A$ provides a cofibrant resolution of $A$. 
\end{itemize}
\end{thm}
The cobar construction $\Om_\alpha$ allows us to define a notion of $\infty$-morphisms for curved algebras in the same way they are defined in the uncurved case. That is, an $\infty$-morphism from $A$ to $B$ is a curved coalgebras morphism $\Om_\alpha A\rightarrow \Om_\alpha B$. By the theorem above, such an $\infty$-morphism corresponds to a zig-zag of curved algebra morphisms $A \xleftarrow{\sim} \Om_\alpha \B_\alpha A \rightarrow B$ of morphisms.

\subsection*{André--Quillen cohomology}

Given an algebra $A$ over a curved operad $\Pc$, we define the notion of $A$-module as a gr-dg module $(M, d_M)$ with a coherent action $\gamma_M$ of $A$ by means of the operations in $\Pc$ on $M$. 
The difference coming from the curved context is that the equality ${d_M}^2 = \gamma_M(\theta_\Pc \otimes -)$ has to be satisfied. 
It follows from this equation that the underlying module of the free $A$-module on a gr-dg module $(N, d_N)$ depends on the gr-differential $d_N$. 
As usual, we can define derivations as a linearized version of the notion of morphisms and we can represent on the left (by means of Kähler differential forms) and on the right the module of derivations.
Deriving the module of Kähler differential forms, we obtain the André--Quillen cohomology:
\begin{thm}[Theorems \ref{thm: adjunction alg-mod} and \ref{thm: cotangent complex}, Definition \ref{def: cotangent complex and AQ cohom}, Proposition \ref{prop: AQ cochain complex}]
\ 

\begin{itemize}
\item[(1)] Let $\Pc$ be a curved operad and $A$ be a $\Pc$-algebra. Let us note $\ModAPz$ the category of operadic $A$-modules with degree zero $A$-module morphisms commuting with the predifferentials.
The following two functors (abelianisation induced by Kähler differential forms on the left, and square zero extension by $A$ on the right)
\[
A\hotimes_{-}^{\Pc}\Omega_{\Pc} -\  :\ \PalgA \rightleftharpoons \ModAPz \cong \AbPA\ :\ A\ltimes -
\]
form a Quillen adjunction, and the \emph{cotangent complex} of a $\Pc$-algebra $A$, denoted by $\Lb_A$, is then the derived abelianisation $\Lb_A \coloneqq \Lb (A\hotimes_{-}^{\Pc}\Omega_{\Pc} -)(A)$ of $A$. 

\item[(2)] Given an $A$-module $M$ in $\AbPA \cong \ModAPz$, the \emph{André--Quillen cohomology of $A$ with coefficients in $M$} is defined by
\[
H_{AQ}^n \left(A, M \right) \coloneqq \Hom_{\Ho(\ModAPz)}\left( \Lb_A, M[n]\right) = H^n \left(\Hom_{\ModAP}\left( \Lb_A, M\right), \partial\right).
\]

\item[(3)] Let $\Pc$ be a curved sa operad, $\Cc$ be a curved altipotent cooperad and $\alpha : \Cc \to \Pc$ be an operadic curved twisting morphism. Let $A$ be a $\Pc$-algebra. Assume that $\Om_\alpha \B_\alpha A$ is a cofibrant resolution of $A$. Then the corresponding cotangent complex is a \emph{quasifree $A$-module over the bar construction $\B_\alpha A$}.
\end{itemize}
\end{thm}

We provide explicit chain complexes to compute it in certain good situations such as for curved unital associative algebras (or $\cuAi$-algebras) and for curved complex Lie algebras (or $\cCxi$-algebras), the two main examples presented at the beginning of the introduction.

\subsection*{$\cCxi$-algebras in derived complex analytic geometry}

We show for example that 0-cycles of a $\cCxi$-algebra $A$ with coefficients in the trivial $A$-module $\Cb$ of complex numbers coincide with $\infty$-morphisms between the $\cCxi$-algebras $A$ and $\Cb$. 
In particular, when $A$ encodes a formal complex structure (see \cite{jM14}), 0-cycles are formal holomorphic functions on a formal complex manifold. 
We also show that a chain complex computing the André--Quillen cohomology of $A$ with coefficients in $\Cb$ is a $\Cb$-cdga (see Proposition \ref{prop: C-algebra structure}) and that it produces a fully faithful functor to EFC-dgas (see Theorem \ref{thm: link EFC}):
\begin{thm}
\label{T:fromCxtoPridham}
The functor
\[ \begin{array}{rccl}
C_{AQ}^\bullet : & (\cCxi\textsf{-alg.}, \infty\text{-}\mor) & \to & \textsf{EFC-dga},\\
& (A, \gamma_A, d_A) & \mapsto & \left(\Hom_{\grModRb}(\B_\iota A, \Cb), \partial_\tau \right)
\end{array} \]
is well-defined and fully faithful.
\end{thm}
Here $(\cCxi\textsf{-alg.}, \infty\text{-}\mor)$ is the category whose objects are $\cCxi$-algebras and morphisms are $\infty$-morphisms. 
We also note $\textsf{EFC}$ the category whose objects are the $\mathbb{C}^n,n\in\mathbb{N}$ and the morphisms are the complex analytic maps. 
It is a small category closed under finite products and generated under products by $\mathbb{C}$, that is, a Lawvere theory. An EFC-algebra is then a product-preserving set-valued functor $\textsf{EFC}\rightarrow \textsf{Set}$. This means that an EFC-algebra is a set $A$ equipped with operations induced by the ring of holomorphic functions $\mathcal{O}(\mathbb{C}^n)$ on $\mathbb{C}^n$ for every $n$ (in particular, for $n=2$ it implies that $A$ is a commutative $\mathbb{C}$-algebra). An important example of EFC-algebra is the ring of global functions of any complex analytic space. An EFC-dga is then a dg algebra whose $0$-cycles carries an EFC-algebra structure. 
Derived analytic spaces à la Pridham \cite{Pri20} are homotopy sheaves over the site of EFC-dgas, so these play a role similar to affine schemes in algebraic geometry and cdgas in derived algebraic geometry.
The comparison with derived analytic spaces in the sense of \cite{jLDAGIX, mP16} follows from \cite[Proposition 4.3, Proposition 4.5]{Pri20}. 
Briefly, derived analytic spaces in the sense of Lurie-Porta whose $0$-truncation is a classical analytic space are equivalent to Pridham analytic spaces.

The theorem above points out a very interesting relationship between $\cCxi$-algebras and such derived complex analytic spaces. Indeed, as mentioned above, André--Quillen complexes of the appropriate $\cCxi$-algebras are dg enhancements of formal holomorphic functions on formal complex manifolds and, (bar) dually, $\cCxi$-algebras encode a \emph{homotopy coherent} notion of integrable almost complex structure on tangent spaces of such formal manifolds. Therefore, we consider such $\cCxi$-algebras as the local picture of yet-to-be-defined derived integrable almost complex analytic spaces. This indicates how to implement a homotopy coherent notion of integrable almost complex structure in derived complex analytic geometry (a derived flavour of the Newlander--Nirenberg integrability theorem), which does not appear in the currently existing models based on the locally ringed space approach. Theorem \ref{T:fromCxtoPridham} provides then a way to compare our model with Pridham's model (and consequently with Lurie-Porta's model), showing that ``going back to the ring of functions'' gives a dg EFC-algebra structure whose $0$-truncation is a standard complex analytic structure. The detailed construction of this new model with various examples and applications in mind is the focus of a subsequent ongoing work.

\subsection*{Model category structures on curved algebras}

There exist other model category structures related to curved algebras. 
In the references \cite{aL16, bLG19, vR23}, the model category structure on curved algebras or curved operads is obtained via transfer through the bar-cobar adjunction. 
This is possible because the category on the other side of the adjunction has a well-defined model category structure where weak equivalences are quasi-isomorphisms. 
In the present article, the same idea cannot be used since our objects are curved on both sides of the bar-cobar adjunction, and therefore quasi-isomorphisms aren't relevant on both sides. 

In \cite[Section 9.3]{lP11}, Positselski describes a model category structure on curved coassociative coalgebras. 
He considers increasing filtrations on the coalgebras with conditions based on the fact that the coalgebra admits a coaugmentation. 
Weak equivalences are then generated by graded quasi-isomorphisms (called filtered quasi-isomorphisms). 
This model category structure is Quillen equivalent through the bar-cobar adjunction to a model category structure on a category of augmented dg associative algebras (see the end of Section 9.3 in \cite{lP11}). In this latter model structure, weak equivalences are quasi-isomorphisms. 
In the present paper, both our algebras and coalgebras have curvature and they're not augmented or coaugmented on the algebraic level. This forces us to use different model category structures.

The reference \cite{lP12} deals with curved homotopy associative algebras which are modules over a local ring and such that the curvature is divisible by the maximal ideal. In this context, it proposes a Koszul duality theory based on the derived functor $\mathrm{Tor}$ and $\mathrm{Ext}$.
This a specific case of a filtration that we can consider in our article but our examples, $\cuAs$ and $\cCx$, don't fit in the context described in \cite{lP12}. It is not only because they are curved operads instead of curved algebras, but also because we are working over a field ($\Rb$) and in this case, the maximal ideal is zero so the curvature would necessarily be zero. This constraint is highlighted in the introduction of \cite{lP12}, paragraph 0.23. 

The notion of weak equivalences defined in \cite{CCN21} is similar to ours but the context is slightly different and the authors starts from operads which are augmented and connected. 
It is not clear whether their results extend to our examples of curved operads.

In their recent paper \cite{BL23}, Booth and Lazarev propose model category structures on the category of curved associative algebras and on the category of curved coassociative coalgebras. 
These model category structures don't rely on any filtration conditions on the curvature. 
Nevertheless, they explain in the end of the introduction that their article do not extend directly to other kind of algebras or to symmetric operads. 
Maybe more problematic, they restrict to curved coassociative coalgebras for which the comultiplication of an element is a finite sum. 
Our Koszul dual cooperad does not satisfy this property.

\subsection*{Acknowledgements}
The first author thanks Joost Nuiten for regular discussions on this work.
We also would like to thank Ricardo Campos and Victor Carmona Sanchez for useful discussions.

\section*{Notations and conventions}

We fix a field $\Kf$ of characteristic 0.
We have to work over a field of characteristic 0 for two reasons: having a well-defined functor $\Gr$ from complete filtered cooperads to complete graded cooperads and having model category structures on the involved categories (on the algebraic side: operads, algebras). 
In the paper, $\Rf$ will denote a unital $\Kf$-cdga (\emph{commutative differential graded algebra}) and $\ModA$ the category of dg-modules over $\Rf$. 
The category $\ModA$ is a Grothendieck category endowed with a closed symmetric monoidal structure given by the tensor product and we can apply the results of \cite{BMDC20}. 
We denote by $\ModApg$ the category of complete filtered predifferential graded (pg for short) objects in $\ModA$ and by $\ModAgr$ the category of complete filtered gr-dg objects in $\ModA$, called $\complete^{\mathsf{gr}}\! \! (\ModA)$ in \cite{BMDC20}. 
The categories of complete filtered gr-dg modules are reflective subcategories of the categories of complete filtered predifferential graded objects.
An object $(M, F, d)$ in $\ModApg$ is an object $M$ in $\ModA$
\begin{itemize}
\item
endowed with an $\Nb$-filtration $M = F_0 M \supset F_1 M \supset F_2 M \supset \cdots$, and
\item
{\em complete} with respect to this filtration, that is $M = \lim_{n} M/F_n M$,
\item
endowed with a $\Zb$-cohomological degree $M^\bullet$ and with an endomorphism $d : M^\bullet \to M^{\bullet + 1}$ of cohomological degree $1$ called a {\em predifferential}.
\end{itemize}
Such an object $(M, F, d)$ is more specifically in $\ModAgr$ when
\begin{itemize}
\item
$(M, F, d)$ is {\em gr-dg}, that is the endomorphism $d$ satisfies
\[ d^2 : F_p M \to F_{p+1} M \]
for all $p \geq 0$, or equivalently $d$ induces a differential on the associated ($\Nb$-)graded $\Gr_\bullet M$ defined by $\Gr_p M \coloneqq F_p M / F_{p+1} M$, for all $p \geq 0$.
\end{itemize}
Morphisms in $\ModApg$ (resp. $\ModAgr$) are morphisms in the category of graded objects in $\ModA$ which respect the filtration and the cohomological degree and which commute with the predifferentials.

By abuse of notation, we often write only $M$ instead of $(M, F, d)$ for a complete dg object (or a complete gr-dg object).

In the following, the algebra $\Rf$ might be dg and complete. We emphasize that in the dg setting, the module structure map $\rho_M$ of an $\Rf$-module $M$ doesn't commute with the differential but satisfies the following compatibility relation
\[ d_M \cdot \rho_M = \rho_M \cdot (d_\Rf \otimes_\Kf \id_M + \id_\Rf \otimes_\Kf d_M), \]
where $\otimes_\Kf$ denotes the complete tensor product over the ring $\Kf$.

In this article, we use the notations $\Tc(M)$ and $\Tc^c(M)$ for the free operad and cofree cooperad given by trees indexed by element in $M$. The notations $\Tc(M)^{(k)}$ and $\Tc^c(M)^{(k)}$ stand for the trees with exactly $k$ vertices.

\subsection{Sign conventions}
\label{section: sign conventions}

We mainly work with homological conventions and we denote the homological degree $n$ elements in a graded object $M$ by $M_n$. However, for geometric examples, we use cohomological convention and we denote the cohomological degree $n$ elements in $M$ by $M^n$ or $M[-n]$. We go from one convention to the other by $M^n = M[-n] = M_{-n}$.

The (homological) suspension of a graded object $M$ is denoted by $sM = s\Kf \otimes M$ where $s$ is a homological degree 1 element. It follows that $(sM)_n = M_{n-1}$. Similarly we denote by $s^{-1}M$ the (homological) desuspension of $M$. From a cohomological point of view, the element $s$ has a cohomological degree equal to $-1$ and the element $s^{-1}$ has cohomological degree $1$.
 
In the whole article, the signs are computed by means of the Koszul sign rule, that is a sign appears when graded elements or maps are switched. For example given two maps $f$ and $g$ and two elements $v$ and $w$ in their respective domains, we have $(f\otimes g)(v \otimes w) = (-1)^{|g| |v|}f(v) \otimes g(w)$ where $|x|$ denote the degree of the element or map $x$.

When working with homotopy algebras as $\Ai$-algebras or $\Li$-algebras, we have to choose a convention on the signs appearing in the relations between the maps defining the $\Ai$ structure (respectively the $\Li$ structure). For example, in the case of $\Ai$-algebras, as it is written in Part 1, Section 4.2 in \cite{tM21}, two possibilities are:
\begin{align}
\sum_{p+q+r=n} (-1)^{pq + r} m_{n+1-q} \circ_{p+1} m_q = 0 \label{eq: conv1}\\
\sum_{p+q+r=n} (-1)^{p + qr} m_{n+1-q} \circ_{p+1} m_q = 0.
\end{align}
The difference is induced by the sign convention of the homotopy for the associativity relation. 
In the first situation we have:
\[
\partial(m_3) = m_{2} \circ_{1} m_2 - m_{2} \circ_{2} m_2
\]
whereas in the second one:
\[
\partial(m_3) = - m_{2} \circ_{1} m_2 + m_{2} \circ_{2} m_2.
\]
In our article, we follow the convention \eqref{eq: conv1}. This choice is different from the one made in some references that we use as \cite{BMDC20, HM12, LV12} to cite a few. 
We can go from one convention to the other by considering $(-1)^{\frac{(n-1)(n+2)}{2}} m_n$ instead of $m_n$. 
The main differences in our article are visible for (homotopy) curved Lie algebras:
\begin{itemize}
\item
the equation appearing in the definition of curved Lie algebras is $d^2 = [\theta, -]$ instead of $d^2 = [-, \theta]$;
\item
there is a similar sign difference in the formulas for homotopy curved Lie algebras and for the definition of curved coalgebras or cooperads;
\item
a Maurer--Cartan element in a curved Lie algebra is a degree $-1$ element satisfying the equation $\Theta + \partial \alpha + \frac{1}{2}[\alpha, \alpha] = 0$ instead of the equation $\partial \alpha + \frac{1}{2}[\alpha, \alpha] = \Theta$;
\item
there is a similar difference for the definition of a Maurer--Cartan element in a homotopy curved Lie algebra.
\end{itemize}
Of course some signs appearing in the construction of this kind of objects likewise the bar construction of a curved operad or the convolution curved Lie algebra are adapted accordingly. 
Our feeling is that these sign conventions are closer to the original definitions in the field.

\section{Operadic twisting morphisms in a curved context}

In this section we are interested in building a cofibrant resolution of a curved operad. 
When the operad $\Pc$ is differentially graded (the curvature is zero), we are able to provide a cofibrant resolution of $\Pc$ in \cite{HM12, HM23}. 
When the curved operad is augmented, we provide an $\Sb$-cofibrant resolution in \cite{BMDC20}. We prove here that the bar-cobar resolution described in \cite{BMDC20} is in fact cofibrant.

The bar-cobar adjunctions exchange the default of (co)augmentation and the curvature. 
In order to define a model category structure on curved operads and curved algebras, we require a filtration condition on the curvature ($\theta \in F_1 \Pc$) in \cite{BMDC20}. 
It isn't possible to require such a condition relative to the default of augmentation of an operad or for the curvature of a curved cooperad. 
The consequence is a small difference on the base categories on which we consider our (co)algebraic structures:
\begin{itemize}
\item
the base category for curved operads (or algebras over a curved operad) is $\ModKgr$, and we have a model category structure on it,
\item
the base  category for curved cooperads is $\ModKpg$, and we don't provide a model category structure on it.
\end{itemize}

\subsection{Complete and curved context}

Similarly as in \cite{BMDC20}, we consider the categories $(\SMod(\ModKpg), \hcirc, I)$ of $\Sb$-modules in $\ModKpg$ and $(\SMod(\ModKgr), \hcirc, I)$ of $\Sb$-modules in $\ModKgr$. 
In particular, objects in this category are filtered and complete with respect to their filtration.

\subsubsection{Curved operad and cooperad}

For operads and cooperads, we define two notions of morphisms. 
The second one is a lax version of the first one and is a useful enlargement for representing the curved twisting morphisms by the bar and the cobar constructions. 
The lax version is very close to the original definition of morphism between curved associative algebras given in \cite{lP93}.

\begin{defn}
\begin{itemize}
\item
A \emph{curved operad} $(\Pc, \gamma, \eta, d, \theta)$ is a gr-dg operad $(\Pc, \gamma, \eta, d)$ in the monoidal category $\SMod(\ModKgr)$ equipped with a map $\theta : I \to (F_1 \Pc(1))_{-2}$ such that
\begin{enumerate}
\item
$d^2 = [\theta, -]$,
\item
$d(\theta) = 0$ \quad ($\theta$ is \emph{closed}).
\end{enumerate}
\item
A {\em morphism $f : (\Pc, d, \theta) \rightarrow (\Qc, d', \theta')$ of curved operads} is a morphism of gr-dg operads which preserves the curvature, that is satisfying the equations 
\begin{align*}
f \cdot d & = d' \cdot f \quad \textrm{ and}\\
f(\theta) & = \theta'.
\end{align*}
\item
A {\em lax morphism} $(\Pc, d, \theta) \rightarrow (\Qc, d', \theta')$ of curved operads is a pair $(f, a)$ where $f : \Pc \to \Qc$ is a morphism of complete operads and $a : I \to F_1\Qc$ is an $\Sb$-module map of degree $-1$ such that
\begin{align}
d' \cdot f + [a, f] & = f \cdot d \quad \textrm{ and} \label{eq: lax morphism diff}\\
f(\theta) & = \theta' + d' \cdot a + \frac{1}{2}[a, a]. \label{eq: lax morphism curvature}
\end{align}
\item
The composition of two lax morphisms $(g, b)$ and $(f, a)$ of curved operads is given by: $(g, b) \cdot (f, a) \coloneqq (g\cdot f, g(a)+b)$ and the unit for the composition is $(\id, 0)$.
\end{itemize}
We denote by $\cOp$ the category of curved operads with morphisms of curved operads and by $\cOpl$ the category of curved operads with lax morphisms of curved operads.
\end{defn}

Lax morphisms admit a description as \emph{strict} morphisms by twisting the predifferential and the curvature on the codomain curved operad.

\begin{lem}
\label{lem: lax morphism unlax}
Let $(\Qc, \gamma', d', \theta')$ be a curved operad and $a : I \to F_1\Qc$.
\begin{enumerate}
\item
The tuple $(\Qc, \gamma', d'_a \coloneqq d' + [a, -], \theta'_a \coloneqq \theta' + d' \cdot a + \frac12 [a, a])$ is a curved operad that we denote $\Qc_a$. 
In particular, $d'_a$ is a derivation with respect to $\gamma'$, $\theta'_a \in F_1\Qc_{-2}$ is closed and it satisfies
\[
{d'_a}^2 = \left[\theta'_a, -\right].
\]
\item
The data of a lax morphism $(f,a) : (\Pc, d, \theta) \rightarrow (\Qc, d', \theta')$ of curved operads is equivalent to the data of a morphism $f : (\Pc, d, \theta) \rightarrow (\Qc, d'_a, \theta'_a)$ of curved operads.
\end{enumerate}
\end{lem}

\begin{proof}
\begin{enumerate}
\item
We first check that $d'_a$ is a derivation with respect to $\gamma'$. 
It is enough to show that $[a, -]$ is a derivation with respect to $\gamma'$. 
We emphasize that this is true only with $a$ in arity 1 ($a \in \Qc(1)$). 
The Lie bracket $[-,-]$ on $\oplus_{n \geq 0} \Qc(n)$ used above is the anti-symmetrisation of a pre-Lie product that we denote $\{-,-\}$. 
With the notations $\epsilon_i = (-1)^{|\mu|+|\nu_1|+\cdots +|\nu_{i-1}|}$, we have
\begin{multline*}
\gamma'\left([a, \mu]; \nu_1, \ldots, \nu_k\right) + \sum_i \epsilon_i\gamma' \left( \mu; \nu_1, \ldots, [a, \nu_i], \ldots \nu_k\right)\\
= \gamma'\left(\{a, \mu\}; \nu_1, \ldots, \nu_k\right) +(-1)^{|\mu|} \gamma'\left(\{\mu, a\}; \nu_1, \ldots, \nu_k\right)\\
 + \sum_i \epsilon_i\left(\gamma' \left( \mu; \nu_1, \ldots, \{a, \nu_i\}, \ldots \nu_k\right) + (-1)^{|\nu_i|} \gamma' \left( \mu; \nu_1, \ldots, \{\nu_i, a\}, \ldots \nu_k\right)\right)\\
= \gamma'\left(\{a, \mu\}; \nu_1, \ldots, \nu_k\right) + \epsilon_{k} \{\gamma' \left( \mu; \nu_1, \ldots, \nu_k\right), a\}\\
= [a, \gamma'\left(\mu; \nu_1, \ldots, \nu_k\right)]
\end{multline*}
by the associativity of $\gamma'$.

We have $\theta'_a = \theta' + d' \cdot a + \frac12 [a, a] : I \to F_1\Qc_{-2}$ since $a : I \to F_1\Qc_{-1}$ and $\theta'_a : I \to F_1\Qc_{-2}$. 
It is closed since
\begin{multline*}
d'_a(\theta' + d' a + \frac{1}{2}[a, a]) =\\
d' \theta' + [a, \theta'] + {d'}^2 a + [a, d' a] + \frac{1}{2} (d' [a, a] + [a, [a, a]]) = 0
\end{multline*}
since $\theta'$ is $d'$-closed, ${d'}^2 a = [\theta', a]$, $d'$ is a derivation with respect to $[-,-]$, and $[-,-]$ satisfies the Jacobi relation.

Then we compute
\begin{align*}
{d'_a}^2& = d'_a(d' + [a, -]) = {d'}^2 + d' [a, -] + [a, d' -] + [a, [a, -]]\\
& = [\theta', -] + [d' (a), -] + \left[ \frac{1}{2}[a, a], -\right] \quad \text{($d'$ is a derivation for $[,]$)}\\
& = \left[ \theta' + d'(a) + \frac{1}{2}[a, a], -\right] = [\theta'_a, -].
\end{align*}
This proves that $\Qc_a$ is a curved operad.
\item
A lax morphism of curved operads $(f,a) : (\Pc, d, \theta) \rightarrow (\Qc, d', \theta')$ is the data of a morphism of the  underlying graded operads $f : \Pc \rightarrow \Qc$ and a map $a : I \to F_1\Qc$ such that $d'_a f = d' \cdot f + [a, f] = f \cdot d$ and $f(\theta) = \theta' + d' \cdot a + \frac{1}{2}[a, a] = \theta'_a$. 
It is therefore equivalently the data of a morphism of curved operads $f : (\Pc, d, \theta) \rightarrow (\Qc, d'_a, \theta'_a)$.
\end{enumerate}
\end{proof}

\begin{example}
A first example of curved operad is given by the \emph{endomorphism operad} of a complete gr-dg $\Rf$-module $(M, d)$
\[
\End_M \coloneqq (\{ \Hom(M^{\hotimes n}, M)\}_{n \geq 0}, \gamma, \partial, \theta)
\]
where $\Hom(M^{\hotimes n}, M)$ is seen as a $\Kf$-module, the composition map $\gamma$ is given by the composition of functions and
\[
\left\{ \begin{array}{lclcr}
\partial(f) & \coloneqq & [d, f] & = & d \cdot f -(-1)^{|f|} \sum_i f\cdot (\id_M, \ldots, d, \ldots, \id_M),\\
&&&& \textrm{for } f\in \Hom (M^{\hotimes n}, M),\\
\theta & \coloneqq & d^2.
\end{array} \right.
\]
\end{example}

\begin{defn}
Let $(\Pc, \gamma, \eta, d, \theta)$ be a curved operad. 
A structure of $(\Pc, \gamma, \eta, d, \theta)$-algebra ($\Pc$-algebra for short when there is no ambiguity) on the gr-dg module $(A, d_A)$ is a curved operad morphism
\[
\gamma_A : (\Pc, \gamma, \eta, d, \theta) \to \End_{(A, d_A)}.
\]
Explicitly $A$ is endowed with a $(\Pc, \gamma, \eta, d)$-algebra structure and the gr-dg endomorphism $d_A : A \to A$ of degree $-1$ is a predifferential (with respect to the $\Pc$-algebra structure) such that
\[
{d_A}^2 = \gamma_A(\theta \circ \id_A).
\]
(The notion of \emph{predifferential of a $\Pc$-algebra} is recalled in Definition \ref{def: der prediff and coder}.)
\end{defn}

\begin{example}
Given a gr-dg module $(A, d_A)$, the identity map $\id_{\End_{(A, d_A)}}$ provides a structure of $\End_{(A, d_A)}$-algebra on $(A, d_A)$.
\end{example}

We define dually the notion of curved cooperad.

\begin{defn}
\begin{itemize}
\item
A \emph{curved cooperad} $(\Cc, \Delta, \varepsilon, d, \theta^c)$ is a pg cooperad $(\Cc, \Delta, \varepsilon, d)$ in the monoidal category $\SMod(\ModKpg)$ endowed with a \emph{curvature} $\theta^c$, that is a map $\theta^c : \Cc \to I$ of degree $-2$ such that
\begin{enumerate}
\item $d^{2} = (\id_{\Cc} \otimes \theta^c - \theta^c \otimes \id_{\Cc}) \circ \Delta_{(1)}$,
\item $\theta^c \cdot d = 0$,
\end{enumerate}
where $\Delta_{(1)}$ is the infinitesimal decomposition map defined in \cite[section 6.1.4]{LV12}. 
The map $\Delta : \Cc \to \Cc \hcirc \Cc$ is called the \emph{decomposition map}, the map $\varepsilon : \Cc \to I$ is the \emph{counit map}. 
\item
A \emph{morphism $(\Cc, d_{\Cc}, \theta^c) \rightarrow (\Cc', d_{\Cc'}, {\theta^c}')$ of curved cooperads} are morphisms of pg cooperads which preserve the curvature, that is satisfying the equations
\begin{align*}
d' \cdot f & = f \cdot d \quad \textrm{ and}\\
{\theta^c}' \cdot f & = \theta^c.
\end{align*}
\item
A \emph{lax morphism} $(\Cc, d_{\Cc}, \theta^c) \rightarrow (\Cc', d_{\Cc'}, {\theta^c}')$ of curved cooperads is a pair $(f, a)$ where $f : \Cc \to \Cc'$ is a morphism of complete cooperads and $a^c : \Cc \to I$ is a degree $-1$ $\Sb$-module map such that
\begin{align}
d_{\Cc'} \cdot f & = f \cdot d_{\Cc} + (f \otimes a^c - a^c \otimes f) \cdot \Delta_{(1)} \quad \textrm{ and} \label{eq: lax morphism diff coop}\\
{\theta^c}' \cdot f & = \theta^c + a^c \cdot d + a^c\star a^c, \label{eq: lax morphism curvature coop}
\end{align}
with $a^c\star a^c := \gamma_I \cdot(a^c \otimes a^c) \cdot \Delta_{(1)}$. 
\item
The composition of two lax morphisms $(g, b^c)$ and $(f, a^c)$ of curved operads is given by: $(g, b^c) \cdot (f, a^c) \coloneqq (g\cdot f, a^c+b^c \cdot f)$ and the unit for the composition is $(\id, 0)$.
\end{itemize}
\end{defn}

\begin{rem}
We emphasize that there is no filtration assumption on the curvature $\theta^c$ of a curved cooperad.
\end{rem}

We recall from \cite{HM12} (up to the sign) the definition of $(\Cc, d_\Cc, \theta_\Cc)$-coalgebra over a curved cooperad.

\begin{defn}
Let $(\Cc, \Delta, \varepsilon, d, \theta^c)$ be a curved cooperad. A \emph{$(\Cc, \Delta, \varepsilon, d, \theta^c)$-coalgebra} ($\Cc$-coalgebra for short) is a triple $(C, \Delta_{C}, d_{C})$ where $(C, \Delta_{C},  d_C)$ is a $(\Cc, \Delta, \varepsilon, d)$-coalgebra in $\ModApg$ and the predifferential $d_{C} : C \to C$ (of degree $-1$) satisfies:
\[
{d_{C}}^{2} = -(\theta^c \circ id_{C}) \cdot \Delta_{C}.
\]
(The notion of \emph{predifferential of a $\Cc$-coalgebra} is recalled in Definition \ref{def: der prediff and coder}.)
\end{defn}

\subsubsection{Semi-augmentation}

Following \cite{HM12}, we define the notion of curved semi-augmented operad.

\begin{defn}
\begin{itemize}
\item
A \emph{curved semi-augmented operad}, or \emph{curved sa operad} for short, $(\Pc, \gamma, \eta, d, \theta, \varepsilon)$ is a curved operad endowed with an augmentation of (complete) $\Sb$-modules $\varepsilon : \Pc \to I$.
\item
A \emph{morphism between two curved sa operads} $(\Pc, \varepsilon) \xrightarrow{f} (\Pc', \varepsilon')$ is a morphism of the underlying curved operads.
\end{itemize}
\end{defn}

\subsubsection{Altipotence}

We now recall the notion of altipotence from \cite{BMDC20}. 
The \emph{gr-flat} condition required is automatically satisfied since our base field is of characteristic 0. 
The definitions, given for dg complete cooperads in the above reference, only depend on the underlying complete cooperad. The definitions apply therefore to curved (complete) cooperads without modifications.

\begin{defn}
Let $(\Cc, \Delta, \varepsilon, d, \theta^c)$ be a curved cooperad.
\begin{itemize}
\item
A \emph{gr-coaugmentation} $\eta : I \to \Cc$ is a map of complete $\Sb$-modules, such that applying the graded functor $\Gr$ to $(\Cc, \Delta, \varepsilon, d, \theta^c, \eta)$ provides a coaugmented curved cooperad.
\item
In the context of complete gr-coaugmented curved cooperads, we call \emph{morphism of complete gr-coaugmented curved cooperads} a morphism of curved cooperads which commutes with the gr-coaugmentations after applying the graded functor $\Gr$.
\end{itemize}

The \emph{infinitesimal coideal} of $\eta$, notated $\coideal{\Cc}$, is the complete $\Sb$-module which is the pushout (in complete objects)
\[\begin{tikzcd}
	I & \Cc \\
	{I[-1]} & {\coideal{\Cc}.}
	\arrow["\id"', from=1-1, to=2-1]
	\arrow["\eta", from=1-1, to=1-2]
	\arrow[dashed, from=2-1, to=2-2]
	\arrow[dashed, from=1-2, to=2-2]
	\arrow["\lrcorner"{anchor=center, pos=0.125, rotate=180}, draw=none, from=2-2, to=1-1]
\end{tikzcd}\]
 We denote by $\coideal{\eta}$ the map $I[-1] \to \coideal{\Cc}$ in the pushout.
\end{defn}

\begin{rem}
Since the filtration on $I$ is the trivial filtration
\[
F_0 I = I \supset F_1 I = 0 \supset \cdots
\]
we obtain that a gr-coaugmentation is in fact a coaugmentation on the level of complete $\Sb$-modules. 
Following the terminology of the previous section, it is a \emph{semi-coaugmentation}.
\end{rem}

\begin{defn}
Let $(\Cc, \Delta, \varepsilon, d, \theta^c, \eta)$ be a gr-coaugmented curved cooperad. 
\begin{itemize}
\item
We define the \emph{reduced decomposition map} $\bar \Delta : \Cc \to \Cc \hcirc \Cc$ by
\[
\bar \Delta \coloneqq \Delta - ((\eta \cdot \varepsilon) \circ \id) \cdot \Delta - (\id \circ (\eta \cdot \varepsilon))\cdot \Delta + ((\eta \cdot \varepsilon) \circ (\eta \cdot \varepsilon)) \cdot \Delta.
\]
\item
We define a bigraded collection $\CR^{p,n}\Cc$ of subcollections of $\Cc$ recursively as follows.
For $p, n \in \Zb$, we fix $\CR^{p,0}\Cc \coloneqq F_p\Cc$ and $\CR^{0,n}\Cc \coloneqq F_0\Cc$.
Next, supposing that $\CR^{p',n'}\Cc$ is defined for $0 \leq n'<n$ and $0 \leq p'\le p$, we first define a subobject $\widetilde{\CR}^{p,n}(\Cc\hcirc \Cc)$ of $\Cc\hcirc \Cc$. We define it as the sub-object in $\Cc\hcirc \Cc$
\[
\hat\bigoplus_{m \geq 0}\hat\sum_{p_0+\cdots+p_m + \delta = p}F_{p_0}\Cc(m) \hotimes_{\Sb_m} \hat\bigotimes_{i=1}^m F_{p_i}\CR^{p_i+\delta,n-1}\Cc,
\]
where the integers $p_i$ are non negative. 
Then define $\CR^{p,n}\Cc$ as the following pullback:
\[
\begin{tikzcd}
\CR^{p,n}\Cc\ar[tail,dashed,r]\ar[dashed,d] & \Cc\ar[d,"\bar\Delta"]
\\
\widetilde{\CR}^{p,n}(\Cc\hcirc\Cc)\rar[tail] & \Cc\hcirc\Cc.
\end{tikzcd}
\]
Recursively, there are natural inclusions $\CR^{p,n}\Cc\to \CR^{p-1,n}\Cc$ and $\CR^{p,n}\Cc\to \CR^{p,n+1}\Cc$.
\end{itemize}
\end{defn}

\begin{defn}
Let $\Cc$ be a gr-coaugmented curved cooperad. 
\begin{itemize}
\item
For $k \in \Zb$, we define the \emph{$k$-primitives} of $\Cc$ to be the $\Sb$-module:
\[ \Prim_k \Cc \coloneqq \bigcap_{p\geq 0} \CR^{p+k,p}\Cc, \]
where we fix $\CR^{p+k,p}\Cc = \CR^{0,p}\Cc = F_0 \Cc$ when $p+k \leq 0$.
\item
We call $\mathcal{C}$ \emph{altipotent} when the following conditions are satisfied
\begin{enumerate}
\item 
\label{item: altipotence condition about conilpotence}
the map \[
\colim_n \CR^{p,n}\Cc\to \Cc
\]
is an isomorphism for all $p$,
\item
\label{item: altipotence condition about coaugmentation being strong} 
the coaugmentation $\eta(I)$ lies in $\Prim_0\Cc$, and 
\item 
\label{item: altipotence condition about primitives being right closed}
for all $k \in \Zb$ we have
\begin{align}
\tag{$\ast$}
\bar \Delta (\Prim_k \Cc) \subset \hat\bigoplus_{m} \hat\sum_{k_0+\cdots +k_m = k+1} F_{k_0}\Cc(m) \hotimes \left( \Prim_{k_1} \Cc \hotimes \cdots \hotimes \Prim_{k_m} \Cc \right)
\end{align}
where we set $F_{k_0}\Cc = F_0 \Cc$ when $k_0 \leq 0$,
\end{enumerate}
\end{itemize}
\end{defn}

We recall that the infinitesimal coideal is left adjoint to the tree cooperad.

\begin{prop}[Corollary 3.34 in \cite{BMDC20}]
\label{prop: cofree cooperad}
The tree cooperad, viewed as a functor from $\Sb$-modules to altipotent cooperads, is right adjoint to the infinitesimal coideal.
That is, the tree cooperad is cofree in the category of altipotent complete cooperads and in particular, there are bijections, natural in $\Cc$ and $M$
\[
\Hom_{\Sb\textsf{-}\mathsf{mod}^{\mathrm{pg}}}(\coideal{\Cc}, M) \cong \Hom_{\mathsf{alt.\, coop.}}(\Cc, \Tc^c(M)).
\]
\end{prop}

\subsection{Bar and cobar constructions in the unital and curved context}

We describe a bar construction in the context of curved semi-augmented operads. 
Similarly we provide a cobar construction associated with curved altipotent cooperads. 
Then we obtain a bar-cobar adjunction in this context. 
These constructions are similar to the ones given by Lyubashenko in \cite[Sections 5.3 and 5.4]{vL14} but slightly different since he doesn't have an adjunction.

We use the free and cofree constructions described in \cite{BMDC20}.

\subsubsection{Bar construction of a curved sa operad}

Let $(\Pc, \gamma, \eta, d, \theta, \varepsilon)$ be a curved sa operad with $\overline{\Pc} \coloneqq \ker \varepsilon$. 
Its \emph{bar construction} $\B \Pc$ is given by the curved altipotent cooperad
\[
\B \Pc \coloneqq \left( \Tc^c(s\overline{\Pc}), \Delta_\beta, \varepsilon_\beta, d_\beta \coloneqq d_0 + d_1 + d_2, \theta^c_\beta, \eta_\beta \right),
\]
where $\Delta_\beta$ is the cofree decomposition map, $\varepsilon_\beta$ is the usual counit, $\eta_\beta$ is the usual injection ($I \rightarrowtail \Tc^c M$, it will be a gr-coaugmentation only), the curvature $\theta^c_\beta$ is the composition
\[
\Tc^c(s\overline{\Pc}) \twoheadrightarrow s\overline{\Pc} \oplus s^{2} (\overline{\Pc} \hcirc_{(1)} \overline{\Pc}) \xrightarrow{s^{-1} \otimes d + s^{-2} \otimes \gamma_{(1)}} \Pc \xrightarrow{\varepsilon} I
\]
and the different coderivations $d_0$, $d_1$ and $d_2$ are respectively the unique coderivations which extend the degree $-1$ maps:
\begin{itemize}
\item
$\Tc^c(s\overline{\Pc}) \twoheadrightarrow I \xrightarrow{-s\theta} s\overline{\Pc}$,
\item
$\Tc^c(s\overline{\Pc}) \twoheadrightarrow s\overline{\Pc} \xrightarrow{\bar d} s\overline{\Pc}$,
\item
$\Tc^c(s\overline{\Pc}) \twoheadrightarrow s\overline{\Pc} \hcirc_{(1)} s\overline{\Pc} \xrightarrow{s^{-1}\bar \gamma_{(1)}} s\overline{\Pc}$.
\end{itemize}

\begin{prop}
\label{prop: bar functor}
The curvature $\theta^c_\beta$ has degree $-2$ and is closed with respect to the coderivation $d_\beta$. The coderivation satisfies
\begin{equation}
\label{eq: cocurved}
{d_\beta}^2 = (\id \otimes \theta^c_\beta - \theta^c_\beta \otimes \id) \cdot {\Delta_\beta}_{(1)}.
\end{equation}
Moreover the bar construction induces a functor
\[
\B : \mathsf{Curved\ sa\ op.}^{\mathrm{lax}} \rightarrow \mathsf{Curved.\ altip.\ coop.}^{\mathrm{lax}}.
\]
\end{prop}

\begin{proof}
To prove Equation \eqref{eq: cocurved}, we first adapt the proof of Lemma 3.3.3 in \cite{HM12} in order to take the curvature into account. 
The map $\theta^c_\beta$ has degree $-2$ by definition. 
Because $\theta^c_\beta$ is zero on $I \oplus \Tc^{c}(s\overline{\Pc})^{(> 2)}$, to get the equality $\theta^c_\beta \cdot d_\beta = 0$, it is enough to notice the following equalities:
\[
\left\{ \begin{array}{lcl}
\theta^c_\beta \cdot {d_{0}}_{|I} & = & \varepsilon \cdot d \cdot \theta = 0 \quad (\text{since $\theta$ is closed}),\\
\theta^c_\beta \cdot (d_0 + d_1)_{|s\overline{\Pc}} & = & \varepsilon \cdot (- [\theta, s^{-1} -] + d^2(s^{-1} -)) = 0,\\
\theta^c_\beta \cdot (d_{1} + d_{2})_{|\Tc^{c}(s\overline{\Pc})^{(2)}} & = & \varepsilon \cdot \gamma_{(1)} \cdot \big((\bar d - d)\otimes \id + \id \otimes (\bar d - d)\big)(s^{-2} -) = 0,\\
\theta^c_\beta \cdot {d_{2}}_{|\Tc^{c}(s\overline{\Pc})^{(3)}} & = & \varepsilon \cdot \big(\gamma_{(1)}(\bar \gamma_{(1)}(-, -), -) - \gamma_{(1)}(-, \bar \gamma_{(1)}(-, -))\big)(s^{-3} -)\\
& = & 0.
\end{array} \right.
\]

The maps ${d_\beta}^2 = \frac{1}{2}[d_\beta, d_\beta]$ and $(\id \otimes \theta^c_\beta - \theta^c_\beta \otimes \id) \cdot {\Delta_\beta}_{(1)}$ are coderivations (see \cite[Lemma 3.2.2]{HM12} for the second statement), so it is enough to prove that the projection to the cogenerators $s\overline{\Pc}$ of Equation \eqref{eq: cocurved} is satisfied. 
We have
\[
(d_0 + d_1 + d_2)^2 = \underbrace{{d_0}^2} + \underbrace{d_0 d_1 + d_1 d_0} + \underbrace{{d_1}^2 + d_0 d_2 + d_2 d_0} + \underbrace{d_1 d_2 + d_2 d_1} + \underbrace{{d_2}^2}.
\]
The term $({d_0}^2)^{|s\overline{\Pc}}$ is zero since $\im ({d_0}^2) \subset \Tc^{(\geq 2)}(s\overline{\Pc})$. 
We have $(d_0 d_1 + d_1 d_0)^{|s\overline{\Pc}} = (d_1 d_0)^{|s\overline{\Pc}} = d \cdot \theta = 0$ since $\theta$ is closed. 
The term $({d_1}^2)^{|s\overline{\Pc}}$ corresponds to ${\overline{d}}^2$ and the term $(d_0 d_2 + d_2 d_0)^{|s\overline{\Pc}} = (d_2 d_0)^{|s\overline{\Pc}}$ is equal to $-\overline{[\theta, -]} = -\overline{d^2}$ on $s\overline{\Pc}$ so $({d_1}^2 + d_0 d_2 + d_2 d_0)^{|s\overline{\Pc}} = 0$. 
Then we have
\begin{align*}
(d_1 d_2 + d_2 d_1)^{|s\overline{\Pc}} = -s(\bar \gamma_{(1)})((d-\bar d)\otimes \id + \id \otimes (d-\bar d))(s^{-2} -)^{|s\overline{\Pc}}_{|\Tc^c(s\overline{\Pc})^{(2)}}\\
= -s(\bar \gamma_{(1)})((\varepsilon \cdot d)\otimes \id + \id \otimes (\varepsilon \cdot d))(s^{-2} -)^{|s\overline{\Pc}}_{|\Tc^c(s\overline{\Pc})^{(2)}}\\
= (\id \otimes \theta^c_\beta - \theta^c_\beta \otimes \id) \cdot {\Delta_\beta}_{(1) |\Tc^c(s\overline{\Pc})^{(2)}}.\\
 \end{align*}
Finally
\[\begin{aligned}
({d_2}^2)^{|s\overline{\Pc}} =\ & s\left( \bar \gamma_{(1)}(\gamma_{(1)}(-, -), -) - \bar \gamma_{(1)}(-, \gamma_{(1)}(-, -)) \right)(s^{-3} -)^{|s\overline{\Pc}}_{|\Tc^c(s\overline{\Pc})^{(3)}}\\
& - (\theta^c_\beta \otimes \id - \id \otimes \theta^c_\beta) \cdot {\Delta_\beta}_{(1) |\Tc^c(s\overline{\Pc})^{(3)}}\\
=\ & \left(\id \otimes \theta^c_\beta - \theta^c_\beta \otimes \id\right) \cdot {\Delta_\beta}_{(1) |\Tc^c(s\overline{\Pc})^{(3)}}
\end{aligned}\]
since $\gamma$ is associative. 
It follows that Equation \eqref{eq: cocurved} is satisfied.

The bar construction is gr-coaugmented since $d_\beta(1) = d_0(1) = -s\theta \in F_1 s\overline{\Pc} \subset F_1 \B \Pc$ and it is altipotent since it is built by means of the cofree altipotent construction. 

It remains to prove that $\B$ sends lax morphisms of curved operads to lax morphisms of curved cooperads and that it preserves composition of lax morphisms. 
Let $(f, a) : (\Pc, d, \varepsilon, \theta) \rightarrow (\Pc', d', \varepsilon', \theta')$ be a lax morphism of sa curved operads. The map of $\Sb$-modules underlying $f$ is characterized by the morphism of $\Sb$-modules $\bar{f} : \overline{\Pc} \rightarrow \overline{\Pc'}$ and the morphism $\varepsilon' \cdot (s^{-1}f) : s\overline{\Pc} \to I$. 
Because the tree cooperad is cofree in the category of altipotent cooperads (Proposition \ref{prop: cofree cooperad}), the map
\[
\coideal{\Tc}^c(s\overline{\Pc}) \twoheadrightarrow \coideal{I} \oplus s\overline{\Pc} \xrightarrow{a + \bar f} s\overline{\Pc'}
\]
induces a cooperad morphism $\B (f, a) : \Tc^c(s\overline{\Pc}) \to \Tc^c(s\overline{\Pc'})$. 
Moreover $\varepsilon' \cdot (s^{-1}f)$ induces a map $a^c(f) : \Tc^c(s\overline{\Pc}) \to I$ by precomposition by the projection to $s\overline{\Pc}$. 
We now show that the couple $(\B(f, a), a^c(f))$ is a lax morphism of curved cooperads. 
We abbreviate the map $a^c(f)$ by $a^c$ when there is no ambiguity. 

First, we prove Equation \eqref{eq: lax morphism diff coop}. 
The cooperad $\Tc^c(s\overline{\Pc})$ can be seen as an infinitesimal $\Tc^c(s\overline{\Pc'})$-comodule by means of the cooperad morphism $\B(f, a)$. 
Noting $d_\beta$, resp. $d'_\beta$, the coderivation on $\Tc^c(s^{-1}\overline{\Pc})$, resp. on $\Tc^c(s^{-1}\overline{\Pc'})$, the map $d_\beta' \cdot \B(f, a) - \B(f, a) \cdot d_\beta$ is a coderivation $\Tc^c(s\overline{\Pc}) \to \Tc^c(s\overline{\Pc'})$. (This is a direct computation for two coderivations and a cooperad morphism.) 
Similarly, the map $(\B(f, a) \otimes a^c - a^c \otimes \B(f, a)) \cdot \Delta_{(1)}$ is a coderivation $\Tc^c(s\overline{\Pc}) \to \Tc^c(s\overline{\Pc'})$ by Lemma 3.2.2 in \cite{HM12}. 
It follows that it is enough to prove the corestriction of Equation \eqref{eq: lax morphism diff coop} to the generators of the cooperad in the codomain (by means of Lemma 15 in \cite{MV09a}). 
The morphism $f$ satisfies Equations \eqref{eq: lax morphism diff} and \eqref{eq: lax morphism curvature} so by considering the projection to $\overline{\Pc}$, we get
\[
\bar d' \bar f + \overline{[a, \bar f]} = \bar f \bar d \quad \mathrm{and} \quad \bar f(\theta) = \theta' + \bar d' a + \frac{1}{2}\overline{[a, a]},
\]
with $\overline{[a, f]} = \overline{[a, \bar f]}$ since $\eta'(1)$ is central.  
We compute the contributory terms as in Lemma 4.4 in \cite{BMDC20}, using Construction 3.10. 
The image in $I \oplus s\overline{\Pc'} \oplus \Tc^c(s\overline{\Pc'})^{(2)} \oplus \dots$ of the morphism $\B(f, a)$ is equal to
\begin{align*}
& \B(f, a)(1) = 1 + sa + sa \otimes sa + \dots,\\
& \B(f, a)(s\mu) = s\bar f(\mu) + \sum_i s\bar f(\mu) \circ_i sa + sa \otimes s\bar f(\mu) + \dots,\\
& \B(f, a)(s\mu \circ_j s\nu) = s\bar f(\mu) \circ_j s\bar f(\nu) + \dots,
\end{align*}
for any $\mu, \nu \in \overline{\Pc}$. 
We therefore obtain
\begin{align*}
& \left(d_\beta' \cdot \B(f, a) - \B(f, a) \cdot d_\beta\right)_{|I}^{|s\overline{\Pc'}} = s\left( \bar f(\theta) - \theta' - \bar d' a - \frac{1}{2}\overline{[a, a]}\right) = 0,\\
& \left(d_\beta' \cdot \B(f, a) - \B(f, a) \cdot d_\beta\right)_{|s\overline{\Pc'}}^{|s\overline{\Pc'}} = -s\left( \bar d' \bar f(-) + \overline{[a, \bar f(-)]} - \bar f \bar d(-)\right) = 0,\\
& \left(d_\beta' \cdot \B(f, a) - \B(f, a) \cdot d_\beta\right)_{|\Tc^c(s\overline{\Pc'})^{(2)}}^{|s\overline{\Pc'}} = \left(\bar f \otimes a^c - a^c\otimes \bar f \right) \cdot \overline{\Delta}_{(1)}.
\end{align*}
We conclude that $\B(f, a)$ satisfies Equation \eqref{eq: lax morphism diff coop}, that is
\begin{equation}
\label{eq: differentials and lax morphisms}
d_\beta' \cdot \B(f, a) - \B(f, a) \cdot d_\beta = (\B(f, a)\otimes a^c - a^c\otimes \B(f, a))\cdot \Delta_{(1)}.
\end{equation}

For the second equation (that is \eqref{eq: lax morphism curvature coop}), it is enough to check it on $I\oplus s\overline{\Pc} \oplus \Tc^c(s\overline{\Pc})^{(2)}$ because of the definition of the maps involved. 
On $I$, we have:
\begin{multline*}
\left({\theta'_\beta}^c \cdot \B(f, a) - \theta_\beta^c - a^c d_\beta - a^c \star a^c \right)(1)\\
= {\theta'_\beta}^c(1+sa + sa \otimes sa) - 0 - a^c(-s\theta) - 0\\
= \varepsilon'\left( -d' a - \frac{1}{2}[a, a] + f(\theta)\right) = \varepsilon' (\theta') = 0.
\end{multline*}
On $s\mu \in s\overline{\Pc}$, we compute:
\begin{multline*}
\left({\theta'_\beta}^c \cdot \B(f, a) - \theta_\beta^c - a^c d_\beta - a^c \star a^c \right)(s\mu)\\
= {\theta'_\beta}^c\left(s\bar f(\mu)+sa \otimes s\bar f(\mu) + \sum_j s\bar f(\mu) \circ_j sa\right) + \varepsilon d\mu - a^c(-s\bar d \mu) - 0\\
= \varepsilon'\left( -d' \bar f(\mu) - [a, \bar f (\mu)] + f(\bar d \mu)\right) + \varepsilon d\mu\\
= -\varepsilon' (f\eta \varepsilon d (\mu)) + \varepsilon d\mu = -\varepsilon' \eta' \varepsilon d (\mu) + \varepsilon d\mu = 0.
\end{multline*}
Finally, on $s\mu \circ_j s\nu \in \Tc^c(s\overline{\Pc})^{(2)}$, we have
\begin{multline*}
\left({\theta'_\beta}^c \cdot \B(f, a) - \theta_\beta^c - a^c d_\beta - a^c \star a^c \right)(s\mu \circ_j s\nu)\\
= {\theta'_\beta}^c\left(s\bar f(\mu) \circ_j s\bar f(\nu)\right) - (-1)^{|\mu|}\left( \varepsilon \gamma (\mu \circ_j \nu) + a^c(s\bar \gamma (\mu\circ_j \nu)) - \varepsilon' f(\mu)  \varepsilon' f(\nu)\right)\\
= (-1)^{|\mu|}\left( \varepsilon' f(\gamma (\mu \circ_j \nu)) - \varepsilon \gamma (\mu \circ_j \nu) - \varepsilon' f \bar \gamma(\mu \circ_j \nu)\right) = 0.
\end{multline*}
It follows that Equation \eqref{eq: lax morphism curvature coop} is satisfied and the bar construction sends lax morphisms to lax morphisms.

We now prove that the bar construction preserves the composition of lax morphisms. 
Let
\[
(f, a) : (\Pc, d, \varepsilon, \theta) \rightarrow (\Pc', d', \varepsilon', \theta') \text{ and } (g, b) : (\Pc', d', \varepsilon', \theta') \rightarrow (\Pc'', d'', \varepsilon'', \theta'')
\]
be two lax morphisms of sa curved operads. 
We denote by $a^c = a^c(f)$ and $b^c = a^c(g)$ the maps appearing in the definition of the bar construction. 
We have
\[
(\B (g, b), b^c) \cdot (\B(f, a), a^c) = (\B (g, b) \cdot \B (f, a), a^c + b^c \cdot \B (f, a))
\]
that we want to compare to
\[
(\B (g\cdot f, b+g(a)), a^c(g \cdot f)).
\]
The morphism $\B (g, b) \cdot \B (f, a)$ is characterised by its projection to the cogenerators $s\Pc''$ which is given by $(b + \bar g(a)) + \bar g \cdot \bar f = b+g(a) + \overline{g \cdot f}$ (since $a \in F_1 \Pc'$). 
On the other side, the morphism $\B (g \cdot f, g(a)+b)$ is characterised by $g(a)+b + \overline{g \cdot f}$. 
Thus they coincide. 
Similarly, $a^c + b^c \cdot \B (f, a)$ is non zero only on $I \oplus s\overline{\Pc}$ where it is equal to $\varepsilon''(g(a) + g(\bar f)) + \varepsilon'(f) = \varepsilon'' \cdot g \cdot \bar f + \varepsilon' \cdot f = \varepsilon'' \cdot g \cdot f$ since $a \in F_1 \Pc$. 
This coincides with $a^c(g\cdot f)$. 
This proves that the bar construction is a functor and this concludes the proof.
\end{proof}

\subsubsection{Cobar construction of a curved altipotent cooperad}

We consider now a curved altipotent cooperad $(\Cc, \Delta, \varepsilon, d, \theta^c, \eta)$ with $\overline{\Cc} \coloneqq \ker \varepsilon$. 
Using the gr-coaugmentation $\eta$, we obtain a projection map $\Cc \twoheadrightarrow \overline{\Cc}$ that we use to define the maps $\bar d : \Cc \xrightarrow{d} \Cc \twoheadrightarrow \overline{\Cc}$ and $\bar \Delta_{(1)}$ (defined similarly).

The \emph{cobar construction} $\hat\Omega \Cc$ is given by the curved semi-augmented operad
\[
\hat\Omega \Cc \coloneqq \left( \Tc (s^{-1}\overline{\Cc}), \gamma_\omega, \eta_\omega, d_\omega \coloneqq D_0 + D_1 + D_2, \theta_\omega, \varepsilon_\omega\right),
\]
where $\gamma_\omega$ is the free composition map, $\eta_\omega$ is the usual unit, $\varepsilon_\omega$ is the usual projection ($\Tc M \twoheadrightarrow I$), the curvature $\theta_\omega$ is the composite
\[
I \xrightarrow{\eta} \Cc \xrightarrow{-s^{-1} \bar d - s^{-2} \bar \Delta_{(1)}} s^{-1} \overline{\Cc} \oplus s^{-2} \overline{\Cc} \hcirc_{(1)} \overline{\Cc} \rightarrowtail \Tc (s^{-1}\overline{\Cc})
\]
and the different derivations $D_0$, $D_1$ and $D_2$ are respectively the unique derivations which extend the degree $-1$ maps:
\begin{itemize}
\item
$s^{-1}\overline{\Cc} \xrightarrow{-s\theta^c} I \rightarrowtail \Tc (s^{-1}\overline{\Cc})$,
\item
$s^{-1}\overline{\Cc} \xrightarrow{\bar d} s^{-1}\overline{\Cc} \rightarrowtail \Tc (s^{-1}\overline{\Cc})$,
\item
$s^{-1} \overline{\Cc} \xrightarrow{\Delta_{s^{-1}} \cdot \bar\Delta_{(1)}} s^{-1}\overline{\Cc} \hcirc_{(1)} s^{-1}\overline{\Cc} \rightarrowtail \Tc(s^{-1} \overline{\Cc})$, where $\Delta_{s^{-1}}(s^{-1}) = -s^{-2}$.
\end{itemize}

\begin{prop}
\label{prop: cobar const is a functor}
The curvature $\theta_\omega$ has degree $-2$ and is closed with respect to the coderivation $d_\omega$. The derivation satisfies
\begin{equation}
\label{eq: curved}
{d_\omega}^2 = [\theta_\omega, -].
\end{equation}
Moreover the cobar construction induces a functor
\[
\hat\Omega : \mathsf{Curved\ altip.\ coop.}^{\mathrm{lax}} \rightarrow \mathsf{Curved.\ sa\ op.}^{\mathrm{lax}}.
\]
\end{prop}

\begin{proof}
The proof that $d_\omega \theta_\omega =0$ and ${d_\omega}^2 = [\theta_\omega, -]$ is dual to the beginning of the proof of Proposition \ref{prop: bar functor}. 
Similarly, we can prove that $\hat\Omega$ sends lax morphisms of curved cooperads to lax morphisms of curved operads. 
Let $(f, a^c) : (\Cc, d, \eta, \theta^c) \to (\Cc', d', \eta',\, {\theta'}^c)$ be a lax morphism of curved cooperads. 
The map of $\Sb$-modules underlying $f$ is characterized by the morphism of $\Sb$-modules $\bar f : \overline{\Cc} \to \overline{\Cc'}$ and the morphism $s^{-1}\overline{f \cdot \eta} : I \to F_1 (s^{-1}\overline{\Cc'})$ ($f$ is gr-coaugmented) which induces a map $a : I \to F_1 \Tc(s^{-1}\overline{\Cc'})$. 
We write $\hat\Omega(f, a^c) : \Tc(s^{-1}\overline{\Cc}) \to \Tc(s^{-1}\overline{\Cc'})$ the operad morphism induced by
\[
s^{-1}\overline{\Cc} \xrightarrow{-sa^c+ \bar f} I \oplus s^{-1}\overline{\Cc'} \rightarrowtail \Tc(s^{-1}\overline{\Cc'}).
\]
The proof that $(\hat\Omega(f, a^c), a)$ is a lax morphism of curved operads is dual to the second part of the proof of Proposition \ref{prop: bar functor}. 
We prove similarly as in Proposition \ref{prop: bar functor} that it preserves the composition of lax morphisms.
\end{proof}

\begin{example}
We denote by $\Lie$ the operad encoding Lie algebras and by $\cLie$ the operad encoding curved Lie algebras. 
It is defined by
\[
\cLie \coloneqq \left(\Tc\left(\curvA, \lie \right)/\left(\jac \right), 0, \theta \coloneqq \curvLie\right).
\]
where $\curvA$ is of degree $-2$, $\lie$ is of degree 0 and the predifferential is zero.
We filter $\cLie$ by the number of $\curvA$, say
\[
F_p \cLie \coloneqq \{ \mu \in \cLie \textrm{ s.t. the number of } \curvA \textrm{ in } \mu \textrm{ is greater than or equal to } p\}.
\]
(This filtration is in fact a graduation.) 
It is a curved operad whose curvature belongs to $F_1 \cLie$. 
Indeed, $0 = d^2 = [\theta, -]$ because it is a suboperad of the curved operad $\cAs$ for which it is  proved in \cite[Proposition 6.2]{BMDC20}.

We also fix $\sLie \coloneqq \End_{s \Rb} \otimes_H \Lie$ the operad encoding \emph{shifted Lie algebras} and $\scLie \coloneqq \End_{s \Rb} \otimes_H \cLie$ the operad encoding \emph{shifted curved Lie algebras}. 

The operads encoding homotopy version of these algebras are built by means of the cobar construction and the Koszul duality theory of curved operad \cite{BMDC20} and are denoted by $\Li \coloneqq \hat\Omega \Lie^\antishriek$, $\cLi \coloneqq \hat\Omega \cLie^\antishriek$, $\sLi \coloneqq \hat\Omega \sLie^\antishriek$ and $\scLi \coloneqq \hat\Omega \scLie^\antishriek$.

The coproduct $\Delta$ and the partial coproduct $\Delta_{(1)}$ of the cooperad $\sLie^\antishriek$ are computed in \cite[Appendix A]{jM14}. 
The case of the cooperad $\scLie^\antishriek$, Koszul dual cooperad of the curved operad $\scLie$, can be done similarly as Theorem 6.9 in \cite{BMDC20} which concerns the operad encoding curved associative algebras. 

We obtain that the Koszul dual cooperad $\scLie^\antishriek$ is 1-dimensional in each arity, $\scLie^\antishriek(n) \cong \Kb \cdot {l}_{n}^{c}$, where ${l}_{n}^{c}$ is an element of degree $0$ on which $\Sb_n$ acts trivially. The infinitesimal decomposition map on $\scLie^\antishriek$ is given by
\begin{align}
\label{eq: coproduct scLiea}
\Delta_{(1)}({l}_{n}^{c}) = 
\sum_{\substack{p+q=n+1\\ p \geq 1, q \geq 0}} \sum_{\sigma \in Sh_{q,\, p-1}^{-1}} ({l}_{p}^{c} \circ_{1} {l}_{q}^{c})^{\sigma},
\end{align}
where $Sh_{q,\, p-1}^{-1}$ is the set of $(q,\, p-1)$-unshuffles, that is, inverses of $(q,\, p-1)$-shuffles. The formula for the full decomposition map is given by
\[
\Delta_{\scLie^\antishriek}({l}_{n}^{c}) = \sum_{q_{1}+\cdots +q_{p}= n} \frac{1}{p!} \sum_{\sigma \in Sh_{q_{1},\, \ldots ,\, q_{p}}^{-1}} \left({l}_{p}^{c}\, ;\, {l}_{q_{1}}^{c},\, \ldots ,\, {l}_{q_{p}}^{c}\right)^{\sigma},
\]
where $Sh_{q_{1},\, \ldots ,\, q_{p}}^{-1}$ is the set of $(q_1,\, \ldots,\, q_p)$-unshuffles $\sigma$.
\end{example}

\subsubsection{Convolution curved operad}

We aim to extend the $\Li$-algebra structures on
the convolution algebras $\hom(C, A)$ of $\Rf$-linear maps from a $\Cc$-coalgebra $C$ to a $\Pc$-algebra $A$ given in \cite{fW19} to the curved setting.
We first describe the convolution curved operad $\Hom(\Cc, \Pc)$ and see how it is related to the bar and cobar constructions.

Let $(\Cc, \Delta, \varepsilon_\Cc, d_\Cc, \theta^c, \eta_\Cc)$ be a curved altipotent cooperad and $(\Pc, \gamma, \eta_\Pc, d_\Pc, \theta, \varepsilon_\Pc)$ be a curved sa operad.

We consider the graded collection of modules
\[
\Hom(\Cc, \Pc) \coloneqq \{ \Hom_{\textrm{filt.}\ \Kf\text{-mod}}(\Cc(n), \Pc(n))\}_{n \geq 0}.
\]
In order to get an $\Sb$-module, we endow each module with the following action of the symmetric group: for $f \in \Cc(n) \to \Pc(n)$, $\sigma \in \Sb_n$ and $\nu \in \Cc(n)$, we have
\[
(f^\sigma)(\nu) \coloneqq (f(\nu^{\sigma^{-1}}))^\sigma.
\]

Let $f : \Cc(k) \to \Pc(k)$ and let, for $1\leq j \leq k$, $g_j : \Cc(i_j) \to \Pc(i_j)$ be a collection of filtered
maps. 
We define a composition map $\gamma_H(f ; g_1, \ldots, g_k)$ by the formula
\begin{multline*}
\Cc(n) \xrightarrow{\Delta} (\Cc \hcirc \Cc)(n) \twoheadrightarrow\\
\Cc(k) \hotimes_{\Kf[\Sb_k]} \left(\left(\Cc(i_{\sigma(1)}) \hotimes \cdots \hotimes \Cc(i_{\sigma(k)}) \right) \hotimes_{\Kf[\Sb_{i_{\sigma(1)}} \times \cdots \times \Sb_{i_{\sigma(k)}}]} \Kf[\Sb_n] \right)\\
\xrightarrow{\sum_{\sigma \in \Sb_k}f^\sigma\otimes g_{\sigma(1)}\otimes \cdots \otimes g_{\sigma(k)} \otimes \id}\\
\Pc(k)\hotimes_{\Kf[\Sb_k]} \left(\left( \Pc(i_{\sigma(1)})\hotimes \cdots \hotimes \Pc(i_{\sigma(k)}) \right) \hotimes_{\Kf[\Sb_{i_{\sigma(1)}} \times \cdots \times \Sb_{i_{\sigma(k)}}]} \Kf[\Sb_n]\right)\\
\hookrightarrow (\Pc \hcirc \Pc)(n) \xrightarrow{\gamma} \Pc(n).
\end{multline*}
It is well defined and it is filtered since the maps $f$, $g_1$, \ldots, $g_k$ are, and $\Delta$ and $\gamma$ are compatible with the filtrations. 

Let $\eta_H : I \to \Hom(\Cc, \Pc)$ the composite $\eta_\Pc \cdot \varepsilon_\Cc$ and $\varepsilon_H : \Hom(\Cc, \Pc) \to \End(I) \cong I$ be given by
\[
(f : \Cc \to \Pc) \mapsto (\varepsilon_\Pc \cdot f \cdot \eta_\Cc : I \to I),
\]
so that $\varepsilon_H \cdot \eta_H = \id_I$. 
We note also $\partial(f) \coloneqq d_\Pc \cdot f - (-1)^{|f|} f \cdot d_\Cc$.\\

We consider on $\Hom(\Cc, \Pc)$ several filtrations. First we denote by $F_\bullet$ the filtration induced by the filtrations on $\Cc$ and $\Pc$, that is for all $p\geq 0$
\[
F_p \Hom(\Cc, \Pc) \coloneqq \{f \in \Hom(\Cc, \Pc)\ |\ \forall q \geq 0,\ f(F_q\Cc) \subset F_{q+p}\Pc\}.
\]
The maps $\gamma_H$ and $\partial$ are compatible with this filtration since the maps involved in their definitions are filtered. The filtration is complete since the filtration on $\Pc$ is. 

Secondly, we consider on $\Hom(\Cc, \Pc)$ a coarser filtration $\Gamma_\bullet$ that is still complete and such that our intended curvature for $\Hom(\Cc, \Pc)$ is in filtration degree 1. 
We consider all the filtrations $G_\bullet$ compatible with $\gamma_H$ and $\partial$ such that
\begin{equation}
\label{cond: filtration op}
	\left\lbrace
\begin{aligned}
G_0 \Hom(\Cc, \Pc) & = F_0 \Hom(\Cc, \Pc) = \Hom(\Cc, \Pc)\\
G_1 \Hom(\Cc, \Pc) & \supset F_1 \Hom(\Cc, \Pc) +\Hom(\Cc, \Pc)(0) + \Hom(\coideal{\Cc}, \Pc)(1)\\
G_p \Hom(\Cc, \Pc) & \supset F_p \Hom(\Cc, \Pc) \text{ for all } p \geq 2.
\end{aligned}
	\right.
\end{equation}

The constant filtration given by $G_p \Hom(\Cc, \Pc) = \Hom(\Cc, \Pc)$, for $p\geq 0$, is a filtration satisfying the above conditions \eqref{cond: filtration op}. 
Therefore there exists a smallest such filtration, namely the intersection of all filtrations satisfying \eqref{cond: filtration op}. We denote it by
\[
\Gamma_0 \Hom(\Cc, \Pc) = \Hom(\Cc, \Pc) \supseteq \Gamma_1 \Hom(\Cc, \Pc) \supseteq \cdots
\]
and we call it the \emph{lower central filtration} $\Gamma_\bullet$ of $\Hom(\Cc, \Pc)$. 

\begin{rem}
We prove in Lemma \ref{lem: the lower central filtration is complete} that $\Gamma_\bullet$ is complete.
\end{rem}

Before describing slightly the lower central filtration, we prove several properties on a (curved) altipotent cooperad. First of all, by the freeness property of the tree cooperad (see Corollary 3.35 in \cite{BMDC20}), there exists a cooperad morphism $\coideal{\Delta} : \Cc \to \Tc^c(\coideal{\Cc})$ corresponding to the identity map $\id : \coideal{\Cc} \to \coideal{\Cc}$. 
Because $\Cc$ is gr-coaugmented, we have on the level of $\Sb$-modules the isomorphism $\Cc \cong I \oplus \overline{\Cc}$ and we deduce a map
\[
\Delta^\pi : \Cc \to \prod_{l \geq 0} \Tc^c(\overline{\Cc}, \coideal{I})_l [l] \cong \prod_{l \geq 0} \Tc^c(\overline{\Cc}, I)_l,
\]
where $\Tc^c(\overline{\Cc}, I)_l$ stands for the $\Sb$-module of trees labelled by elements in $\overline{\Cc}$ or in $I$ with exactly $l$ elements in $I$.

\begin{lem}
\label{lem: coproduct of altipotent cooperad is bounder mod p}
Let $\Cc$ be a curved altipotent cooperad.
Let $\Cc^\fk \subset \Cc$ be a finitely generated sub-$\Sb$-module. 
Fix an integer $p$. 
Then there exists an integer $s(\fk, p) \geq 0$ such that for each $l_0 \geq 0$, the image $I_\Delta^{l_0, p}$ of the composite
\[
\Cc^f \xrightarrow{\Delta^\pi} \prod_{l \geq 0} \Tc^c(\overline{\Cc}, I)_l \twoheadrightarrow [\Tc^c(\overline{\Cc}, I)_{l_0}]_p \coloneqq \Tc^c(\overline{\Cc}, I)_{l_0}/F_p \Tc^c(\overline{\Cc}, I)_{l_0}
\]
factorises through
\[
\Tc^c(\overline{\Cc}, I)^{(\leq s(\fk, p)+l_0)}_{l_0}/F_p \Tc^c(\overline{\Cc}, I)^{(\leq s(\fk, p)+l_0)}_{l_0},
\]
where $\Tc^c(\overline{\Cc}, I)^{(\leq s(\fk, p)+l_0)}_{l_0} \subset \Tc^c(I \oplus \overline{\Cc})^{(\leq s(\fk, p)+l_0)}$ stands for the trees with at most $s(\fk, p)+l_0$ vertices and exactly $l_0$ vertices labelled by an element in $I$.
\end{lem}

\begin{proof}
The size of a tree appearing in the image under consideration is bounded in width by the fact that $\Delta : \Cc \to \Cc \hcirc \Cc$ is well-defined and in depth (or height) by the fact that $\Cc$ is altipotent.
\end{proof}

\begin{lem}
\label{lem: the lower central filtration is complete}
The collection $\Hom(\Cc, \Pc)$ is complete for the lower central filtration $\Gamma_\bullet$.
\end{lem}

\begin{proof}
The $\Sb$-module $\Gamma_{v} \Hom(\Cc, \Pc)$ is the $\Sb$-module of multiple compositions of $\gamma_H$ applied to elements $f_1$, \ldots, $f_m$ such that there exists a partition $\{1, \ldots, m\} = E_1 \sqcup E_2$ with
\begin{enumerate}
\item
$f_j \in F_{p_j}\Hom(\Cc, \Pc)$ for all $j \in E_1$,
\item
$f_j \in \Hom(\Cc, \Pc)(0) \times \Hom(\coideal{\Cc}, \Pc)(1)$ for all $j \in E_2$,
\end{enumerate}
and $(\sum_{j \in E_1} p_j) + |E_2| \geq v$. 

For each $n\in \Nb$, we fix a direct sum decomposition by finite-dimensional $\Sb_n$-submodules $\Cc(n) = \hoplus_{\fk \in \Fk_n} \Cc^\fk$ of the $\Sb_n$-module $\Cc(n)$ (which exists since $\Kf$ is field and therefore $\Kf[\Sb_n]$ is semi-simple) so that
\begin{equation}
\label{eq: decomposition of the convolution operad}
\begin{aligned}
\Hom(\Cc, \Pc) & = \{ \Hom(\Cc(n), \Pc(n)) \}_{n \in \Nb} = \{ \Hom\left(\hoplus_{\fk \in \Fk_n} \Cc^\fk, \Pc(n)\right) \}_{n \in \Nb}\\
& \cong \left\{ \prod_{\fk \in \Fk_n} \Hom(\Cc^\fk, \Pc(n)) \right\}_{n \in \Nb}.
\end{aligned}
\end{equation}
The filtration $F_p$ on $\Hom(\Cc, \Pc)$ restricts to $\Hom(\Cc^{\fk}, \Pc(n))$ and we keep the same notation for it.

For all $\fk \in \Fk_n$, because $\Cc^\fk$ is finite-dimensional and the filtration $F_\bullet$ is complete, there exists a smallest integer, denoted $q(\fk)$, such that $\Cc^\fk \cap F_{q(\fk)+1} \Cc = \{ 0\}$.  
Fix an integer $p$. 
Because $\Cc$ is altipotent and by Lemma \ref{lem: coproduct of altipotent cooperad is bounder mod p}, we get that
for all $l_0 \geq 0$ the image $I_\Delta^{l_0, q(\fk)+p}$
lands in $\Tc^c(\overline{\Cc}, I)^{(\leq s(\fk, p)+l_0)}_{l_0}/F_{q(\fk)+p} \Tc^c(\overline{\Cc}, I)^{(\leq s(\fk, p)+l_0)}_{l_0}$, where $s(\fk, p)$ depends only on $\Cc^\fk$ and $p$.

We now prove that for any
\begin{equation} \label{eq: upper-bound}
v \geq 2p+s(\fk, p)
\end{equation}
we have
\[
\forall \varphi \in \Gamma_v \Hom(\Cc, \Pc),\ \varphi_{|\Cc^{\fk}} \in F_p \Hom(\Cc^{\fk}, \Pc).
\]
An element $\varphi \in \Gamma_v \Hom(\Cc, \Pc)$ is a sum $\sum_{t_v} \tilde \gamma_H \cdot (f_1 \otimes \cdots \otimes f_{m(t_v)}) \cdot t_v$, where $t_v \in \Tc (\Cc)$  is a tree with $m(t_v)$ vertices, $\tilde \gamma_H$ is induced by $\gamma_H$ and the $f_j$'s are elements in $\Hom(\Cc, \Pc)$ such that there exists a partition $\{1, \ldots, m(t_v)\} = E_1 \sqcup E_2$ with $f_j \in F_{p_j} \Hom(\Cc, \Pc)$ for all $j \in E_1$ and $f_j \in \Hom(\Cc, \Pc)(0) \times \Hom(\coideal{\Cc}, \Pc)(1)$ for all $j \in E_2$, and $(\sum_{j \in E_1}p_j) + |E_2| \geq v$. 
Due to the isomorphism $\Cc \cong I \oplus \overline{\Cc}$, we can assume without loss of generality that the vertices of $t_v$ are either in $I$ or in $\overline{\Cc}$. 
When $\sum_{j \in E_1}p_j \geq p$, we automatically have the desired conclusion.  
Otherwise by Condition \eqref{eq: upper-bound}, we have $|E_2| > p+s(\fk, p)$. 
This implies that there are at least $p$ maps $f_j$, for $j \in E_2$, that are applied to $I$ (thus each of them increases the filtration degree by at least 1) and therefore the term $\tilde \gamma_H \cdot (f_1 \otimes \cdots \otimes f_{m(t_v)}) \cdot t_v$ increases the filtration degree by at least $p$. 
It follows that $\varphi_{|\Cc^{\fk}} \in F_p \Hom(\Cc^{\fk}, \Pc)$ as expected.

Let us now prove that the filtration is complete. 
For $\fk \in \Fk$, we note $H_\fk \coloneqq \Hom(\Cc^\fk, \Pc)$. 
We write $i_\fk : \Cc^\fk \hookrightarrow \Cc$ for the inclusion and $i_\fk^* : \Hom(\Cc, \Pc) \to H^\fk$ for the restriction. 
We have just seen the existence for all $p\in \Nb$ of some integer $v_p$ such that $i_\fk^* \Gamma_{v_p}H \subset F_p i_\fk^* H = i_\fk^* F_p H$, where the last equality follows from the decomposition \eqref{eq: decomposition of the convolution operad}. 
Using moreover that $H \coloneqq \Hom(\Cc, \Pc)$ is complete for the filtration $F_\bullet$ and the interchange property of limits, we obtain
\begin{multline*}
H_\fk = i_\fk^* H \cong \lim_{p \in \Nb} H_\fk /(i_\fk^* F_p H) = \lim_{p \in \Nb} H_\fk /(F_p H_\fk)\\
\cong \lim_{p \in \Nb} \lim_{q \in \Nb} H_\fk /(F_p H_\fk + i_\fk^*\Gamma_q H) \cong \lim_{q \in \Nb} \lim_{p \in \Nb} i_\fk^*H /(i_\fk^* F_p H + i_\fk^* \Gamma_q H)\\
\cong i_\fk^*\lim_{q \in \Nb} H/ \Gamma_q H.
\end{multline*}
This isomorphism being valid for all $\fk \in \Fk$, we get by the decomposition \eqref{eq: decomposition of the convolution operad} the identification $H \cong \lim_q H/\Gamma_q H$.
This concludes the proof.
\end{proof}

We define the degree $-2$ element $\Theta_H \coloneqq \theta \cdot \varepsilon_\Cc + \eta_\Pc \cdot \theta^c$. Because $\eta_\Pc \cdot \theta^c \cdot \eta_\Cc = 0$ and by definition of the central lower filtration $\Gamma_\bullet$, we have
\[
\Theta_H \in \Hom(\coideal{\Cc}, \Pc)(1) \subset \Gamma_1\Hom(\Cc, \Pc)(1).
\]

\begin{rem}
We use the filtration $\Gamma_\bullet$ instead of the filtration $F_\bullet$ on $\Hom(\Cc, \Pc)$ so that the element $\Theta_H$ belong to $\Gamma_1 \Hom_(\Cc, \Pc)(1)$. 
Since $\theta^c$ cannot have filtration degree $\geq 1$ (for the filtration $F_\bullet$) when it is non zero, the element $\Theta_H$ doesn't belong to $F_1 \Hom(\Cc, \Pc)$ in general.
\end{rem}

\begin{prop}
\label{prop: convolution curved operad}
The collection $(\Hom(\Cc, \Pc), \Gamma_\bullet, \gamma_H, \eta_H, \partial, \Theta_H, \varepsilon_H)$ is a curved sa operad.
It is called the \emph{convolution curved operad}.
\end{prop}

\begin{proof} 
The operad structure $\gamma_H$ is associative since $\Delta$ is coassociative and $\gamma$ is associative. The map $\eta_H$ is a unit for the composition map since $\varepsilon_\Cc$ is a counit for $\Cc$ and $\eta_\Pc$ is a unit for $\Pc$. 
The map $\varepsilon_H$ is a semi-augmentation since $\varepsilon_H \cdot \eta_H = \id_I$. 

Finally, the element $\Theta_H \in \Gamma_1 \Hom(\Cc, \Pc)$ is closed since $\theta^c$ and $\theta$ are, and a direct computation shows that for a linear map $f : \Cc(k) \to \Pc(k)$, we have
\begin{align*}
\partial^2(f) & = {d_\Pc}^2 \cdot f - f \cdot {d_\Cc}^2\\
& = \gamma(\theta \cdot \varepsilon_\Cc; f) - \sum_i \gamma(f ; \id^{\otimes (i-1)}\otimes \theta \cdot \varepsilon_\Cc \otimes \id^{\otimes (k-i-1)})\\
&\hspace{3cm} - \sum_i \gamma(f ; \id^{\otimes (i-1)} \otimes \eta_\Pc \cdot \theta^c \otimes \id^{\otimes (k-i-1)}) + \gamma(\eta_\Pc \cdot \theta^c ; f)\\
& = \gamma(\Theta_H; f) - \sum_i \gamma(f; \id^{\otimes (i-1)} \otimes \Theta_H \otimes \id^{\otimes (k-i-1)}) = [\Theta_H, f].
\end{align*}
\end{proof}

\begin{rem}
In the sequel, we are mainly interested in the central lower filtration on the convolution curved operad (rather than the filtration $F_\bullet$).
\end{rem}

To the symmetric operad $\Hom(\Cc, \Pc)$, we can associate a pre-Lie product $\star$ on $\hom(\Cc, \Pc) \coloneqq \prod_{n \geq 0} \Hom(\Cc, \Pc)(n)$ as in Section 5.4.3 in \cite{LV12}. 
The product $f \star g$ is given by a symmetrical version of the formula
\[
\Cc \xrightarrow{\Delta_{(1)}} \Cc \hcirc_{(1)} \Cc \xrightarrow{f \circ_{(1)} g} \Pc \hcirc_{(1)} \Pc \xrightarrow{\gamma_{(1)}} \Pc.
\]
This pre-Lie product restricts to invariant elements under the conjugation action to endow the gr-dg module $\hom_{\Sb}(\Cc, \Pc) \coloneqq \left(\prod_{n \geq 0} \Hom_{\Sb}(\Cc, \Pc)(n), \Gamma_\bullet \right)$ with the pre-Lie product $\star$. 
The anti-symmetrisation of $\star$ provides a Lie bracket $[f, g] = f \star g -(-1)^{|f||g|} g \star f$. The computation seen at the end of the proof of Proposition \ref{prop: convolution curved operad} shows that $(\hom_{\Sb}(\Cc, \Pc), [-, -], \partial, \Theta_H)$ is a curved Lie algebra.

\begin{defn}
An \emph{operadic curved twisting morphism} is a degree $-1$ element $\alpha \in \hom_{\Sb}(\coideal{\Cc}, \Pc)$ (that is $\alpha \cdot \eta_\Cc : I \to F_1 \Pc$) solution to the curved Maurer--Cartan equation
\[
\Theta_H + \partial(\alpha) + \frac{1}{2}[\alpha, \alpha] = 0.
\]
We denote by $\Tw(\Cc, \Pc)$ the set of operadic curved twisting morphisms.
\end{defn}

\begin{lem}
Operadic curved twisting morphisms form a bifunctor
\[
\Tw(-,-) : \left(\mathsf{curved\ alt.\ coop.}^{\mathrm{lax}}\right)^{\mathrm{op}} \times \mathsf{curved\ sa\ op.}^{\mathrm{lax}} \to \mathsf{Set}.
\]
\end{lem}

\begin{proof}
We first prove the functoriality in the left variable. 
Let $(f, a^c) : \Cc \to \Cc'$ be a lax morphism between curved altipotent cooperads. To a curved twisting morphism $\alpha : \Cc' \to \Pc$, we associate the map $F(f, a^c)(\alpha) \coloneqq \alpha \cdot f + \eta_\Pc \cdot a^c : \Cc \to \Pc$. 
We first compute
\[
\partial (\alpha \cdot f + \eta_\Pc \cdot a^c) = (\partial \alpha)\cdot f - \alpha \cdot (f \otimes a^c - a^c \otimes f) \cdot \Delta_{(1)} + \eta_\Pc \cdot a^c \cdot d_\Cc.
\]
Then
\begin{multline*}
(\alpha \cdot f + \eta_\Pc \cdot a^c) \star (\alpha \cdot f + \eta_\Pc \cdot a^c) =\\
(\alpha \star \alpha)\cdot f + \alpha \cdot (f \otimes a^c - a^c \otimes f) \cdot \Delta_{(1)} + \eta_\Pc \cdot (a^c \otimes a^c)\cdot \Delta_{(1)}
\end{multline*}
It follows that
\begin{align*}
\partial (F(f, a^c)) + F(f, a^c) \star F(f, a^c) & = -\eta_\Pc \cdot \theta \cdot f + \eta_\Pc \cdot \left( a^c \cdot d_\Cc + (a^c \otimes a^c)\cdot \Delta_{(1)}\right)\\
& = -\Theta_H' \cdot f + \eta_\Pc \cdot ({\theta^c}' \cdot f - \theta^c) = -\Theta_H.
\end{align*}
Therefore $F(f, a^c)$ is a curved twisting morphism (we have $F(f, a^c) \in \hom_\Sb(\coideal{\Cc}, \Pc)$ since $a^c$ has degree $-1$). 
It is direct to check that $F(\id, 0) = \id$ and $F$ preserves the composition. 

Similarly we prove the functoriality in the right variable. Let $(f, a) : \Pc \to \Pc'$ be a lax morphism between curved operads. To a curved twisting morphism $\alpha : \Cc \to \Pc'$, we associate the map $G(f, a)(\alpha) \coloneqq f \cdot \alpha + a \cdot \varepsilon_\Cc : \Cc \to \Pc'$. 
We first compute
\[
\partial (f \cdot \alpha + a \cdot \varepsilon_\Cc) = - [a, f] \cdot \alpha + f \cdot (\partial \alpha) + d_{\Pc'} \cdot a \cdot \varepsilon_\Cc.
\]
Then
\begin{align*}
(f \cdot \alpha + a \cdot \varepsilon_\Cc) \star (f \cdot \alpha + a \cdot \varepsilon_\Cc) = f \cdot (\alpha \star \alpha) + [a, f] \cdot \alpha + a^2 \cdot \varepsilon_\Cc
\end{align*}
It follows that
\begin{align*}
\partial (G(f, a)) + G(f, a) \star G(f, a) & = f\cdot (\partial \alpha + \alpha \star \alpha) + (d_{\Pc'} a + a^2) \cdot \varepsilon_\Cc\\
& = -f \cdot \Theta_H + \left( f(\theta_\Pc) - \theta_{\Pc'} \right) \cdot \varepsilon_\Cc = -\Theta_{H}'.
\end{align*}
Therefore $G(f, a) : \Cc \to \Pc'$ is a curved twisting morphism (we have $G(f, a) \in \hom_\Sb(\coideal{\Cc}, \Pc)$ since $a : I \to F_1 \Pc'$). 
It is direct to check that $G(\id, 0) = \id$ and $G$ preserves the composition. 
This concludes the proof.
\end{proof}

\begin{thm}
\label{thm: bar-cobar adjunction}
For any curved altipotent cooperad $\Cc$ and any curved sa operad $\Pc$, there are natural bijections
\[\Hom_{\mathsf{curv.\, sa\, op.}^{\mathrm{lax}}}(\hat\Omega \Cc,\Pc)
\cong
\Tw(\Cc, \Pc)
\cong
\Hom_{\mathsf{curv.\, alti.\, coop.}^{\mathrm{lax}}}(\Cc,\B\Pc).
\]
\end{thm}

\begin{proof}
We make the first bijection explicit.
A lax morphism of curved sa operads $(f_\alpha, a_\alpha) : \Tc (s^{-1}\overline{\Cc}) \to \Pc$ is uniquely determined by a map $-s \bar\alpha : s^{-1}\overline{\Cc} \to \Pc$ of degree $0$ and an $\Sb$-module map $a_\alpha : I \to F_1\Pc$ of degree $-1$. 
This is equivalent to the data of a map $\alpha \coloneqq a_\alpha \cdot \varepsilon_\Cc + \bar \alpha : \Cc \to \Pc$ of degree $-1$ such that $\alpha \cdot \eta_\Cc : I \to F_1\Pc$, or equivalently to $\alpha \in \hom_\Sb(\coideal{\Cc}, \Pc)$. 
Moreover, $f_\alpha$ commutes with the predifferential up to the term $[a_\alpha, f_\alpha]$ if and only if the following diagram commutes
\[
\xymatrix@C=50pt{s^{-1}\overline{\Cc} \ar[r]^{-s\bar\alpha} \ar[d]_{D_0+D_{1} + D_{2}} & \Pc \ar[r]^{d_{\Pc}} & \Pc\\
I \oplus s^{-1}\overline{\Cc} \oplus (s^{-1}\overline{\Cc} \hcirc_{(1)} s^{-1}\overline{\Cc}) \ar[rr]_{\hspace{.75cm} \eta_\Pc + (-s\bar\alpha) + (-s\bar\alpha) \hcirc_{(1)} (-s\bar\alpha)} && \Pc \oplus \Pc \circ_{(1)} \Pc \ar[u]_{\id_\Pc + \gamma_{(1)}}}
\]
up to the term $[a_\alpha, -s\bar\alpha]$. 
This corresponds to the equality $\eta_\Pc \cdot \theta^c + \partial(\bar\alpha) + [a_\alpha, \bar\alpha] + \bar\alpha \star \bar\alpha = 0$ on $\overline{\Cc}$, or equivalently to the restriction of the equality $\Theta_H + \partial(\alpha) + \alpha \star \alpha = 0$ on $\overline{\Cc}$. 
The fact that the morphism $f_\alpha$ sends the curvature $\theta_\omega$ to $\theta_\Pc + d_\Pc a_\alpha + \frac{1}{2}[a_\alpha, a_\alpha]$ coincides with the restriction of the equation $\Theta_H + \partial(\alpha) + \alpha \star \alpha = 0$ to $\im \eta_\Cc$ since $\alpha \cdot \eta_\Cc = a_\alpha$.

We now make the second bijection explicit. 
A lax morphism of curved altipotent cooperads $(g_{\alpha}, a_\alpha^c) : \Cc \rightarrow {\Tc^{c}}(s\overline{\Pc})$ is uniquely determined by a map $s\bar\alpha : \coideal{\Cc} \rightarrow s\overline{\Pc}$ (by Proposition \ref{prop: cofree cooperad}) and an $\Sb$-module map $a_\alpha^c : \Cc \to I$ of degree $-1$. 
This is equivalent to a map $\alpha \coloneqq \bar\alpha + \eta_\Pc \cdot a_\alpha^c : \Cc \rightarrow \Pc$ of degree $-1$ satisfying $\alpha \cdot \eta_\Pc : I \to F_1 \Pc$, or equivalently to $\alpha \in \hom_\Sb(\coideal{\Cc}, \Pc)$.
Moreover, because $g_\alpha(\mu) = \varepsilon_\Cc(\mu) + s\bar \alpha(\mu) + (s\bar\alpha \otimes s\bar\alpha)\Delta_{(1)} + \dots \in I\oplus s\overline{\Pc} \oplus s\overline{\Pc} \hcirc_{(1)} s\overline{\Pc} \oplus \dots$ (by Construction 3.10 in \cite{BMDC20}), the morphism $g_{\alpha}$ satisfies Equation \eqref{eq: lax morphism curvature coop} if and only if $\varepsilon_\Pc \cdot (-d_\Pc \cdot \bar\alpha - \gamma_{(1)}(\bar\alpha \otimes \bar\alpha) \cdot \Delta_{(1)}) = \theta^c + a_\alpha^c d_\Cc + a_\alpha^c \star a_\alpha^c$, that is to say $\varepsilon_\Pc (\Theta_H + \partial \alpha + \alpha \star \alpha) = 0$. 
Also $g_\alpha$ commutes with the predifferentials up to $(g_\alpha \otimes a_\alpha^c - a_\alpha^c \otimes g_\alpha) \cdot \Delta_{(1)}$ (see Equation \eqref{eq: lax morphism diff coop}) if and only if the following diagram commutes
\[
\xymatrix@C=120pt@R=10pt{\Cc \ar[r]^{\varepsilon_\Cc + s\bar\alpha + ((s\bar\alpha) \otimes (s\bar\alpha)) \cdot \Delta_{(1)} \hspace{1cm}} \ar[dd]_{d_{\Cc}} & I \oplus s\overline{\Pc} \oplus s\overline{\Pc} \hcirc_{(1)} s\overline{\Pc} \ar[dd]^{d_{\beta}^{|s\overline{\Pc}} = (d_0 + d_{1} + d_{2})^{|s\overline{\Pc}}}\\
&\\
\Cc \ar[r]_{s\bar\alpha} & s\overline{\Pc}}
\]
up to the term $(s\bar\alpha \otimes a_\alpha^c - a_\alpha^c \otimes s\bar\alpha) \cdot \Delta_{(1)}$. 
Similarly as before, this is equivalent to the projection of the equality $\Theta_H + \partial(\alpha) + \alpha \star \alpha = 0$ on to $s\overline{\Pc}$.
This concludes the proof.
\end{proof}

\begin{example}\
\label{ex: twisting morph}
\begin{itemize}
\item
Through the adjunction the identity morphism $\id : \B \Pc \to \B\Pc$ corresponds to the operadic curved twisting morphism $\pi : \B \Pc \twoheadrightarrow s\overline{\Pc} \cong \overline{\Pc} \rightarrowtail \Pc$.
\item
Through the adjunction the identity morphism $\id : \hat\Omega \Cc \to \hat\Omega\Cc$ corresponds to the operadic curved twisting morphism $\iota : \Cc \twoheadrightarrow \overline{\Cc} \cong s^{-1}\overline{\Cc} \rightarrowtail \hat\Omega \Cc$.
\end{itemize}
\end{example}

In Section \ref{sec: bar-cobar resolution}, we study the counit of the bar-cobar adjunction above and we prove that it provides a cofibrant resolution in the model category of curved operads. \\

Before that we use the bar-cobar adjunction in order to describe operadic curved twisting morphisms as lax morphisms of curved sa operads from $\scLi$ to the convolution curved operad $\Hom_\Sb(\Cc, \Pc)$. 
This extends a result given in \cite[Section 7]{fW19} which applies to dg operads and cooperads which are connected (that is satisfy $\Pc(0) = 0 = \Cc(0)$ and $\Pc(1) = \Kf = \Cc(1)$). 
Then, in Section \ref{sec: convolution algebra}, we use this lax morphism in order to describe an $\scLi$ structure on the convolution algebra $\homf(C, A)$ for a $\Cc$-coalgebra $C$ and a $\Pc$-algebra $A$ (on which $\Hom(\Cc, \Pc)$ acts).

Let $a \in \Gamma_1 \Hom(\Cc, \Pc)(1) = \hom_\Sb(\coideal{\Cc}, \Pc)(1)$ be an operadic curved twisting morphism. 
For $f \in \Hom(\Cc, \Pc)$, we fix $\partial_a(f) \coloneqq \partial(f) + [a, f]$ and $\Theta_a \coloneqq \Theta_H + \partial a + \frac{1}{2}[a, a]$. 
By Lemma \ref{lem: lax morphism unlax}, the collection $(\Hom(\Cc, \Pc), \gamma_H, \eta_H, \partial_a, \Theta_a, \varepsilon_H)$ is curved sa operad, that we call a \emph{twisted convolution curved operad} and denote $\Hom(\Cc, \Pc)_a$.

\begin{prop}
\label{prop: twisting morphisms as morphisms}
For any curved altipotent cooperad $\Cc$ and any curved sa operad $\Pc$, there is a bijection
\[
\Tw(\Cc, \Pc) \cong \Hom_{\mathsf{curv.\, sa\, op.}^{\mathrm{lax}}}(\scLi, \Hom(\Cc, \Pc)).
\]
Moreover, given an operadic curved twisting morphism $\alpha \in \Tw(\Cc, \Pc)$, we get a morphism of curved operads $f_\alpha : \scLi \to \Hom(\Cc, \Pc)_{a_\alpha} = (\Hom(\Cc, \Pc), \partial_{a_\alpha}, \Theta_{a_\alpha})$, where $a_\alpha :  I \to \Hom(\Cc, \Pc)$ is defined by $a_\alpha(1) = \alpha(1) : \Cc(1) \to \Pc(1)$.
\end{prop}

\begin{proof}
By Theorem \ref{thm: bar-cobar adjunction}, we have
\[
\Hom_{\mathsf{curv.\, sa\, op.}^{\mathrm{lax}}}(\scLi, \Hom(\Cc, \Pc)) \cong \Tw (\scLie^\antishriek, \Hom(\Cc, \Pc)).
\]
Given a curved sa operad $(\Qc, F_\bullet)$, there is an isomorphism
\[
\hom_\Sb\left(\coideal{\scLie^\antishriek}, \Qc\right) \cong F_1\Qc(0)^{\Sb_0} \times F_1 \Qc(1)^{\Sb_1} \times \prod_{n \geq 2} \Qc(n)^{\Sb_n}
\]
since the $\Sb_n$ action on $\scLie^{\antishriek}$ is trivial. When applied to $(\Qc, F_\bullet) = (\Hom(\Cc, \Pc), \Gamma_\bullet)$ (on which the $\Sb$-action is given by conjugation), this implies the isomorphisms
\[
\hom_\Sb\left(\coideal{\scLie^\antishriek}, \Hom(\Cc, \Pc)\right) \cong \hom_\Sb(\coideal{\Cc}, \Pc).
\]
Then, denoting by $\Upsilon_\alpha \in \hom_\Sb\left(\coideal{\scLie^\antishriek}, \Hom(\Cc, \Pc)\right)$ the element associated with $\alpha \in \hom_\Sb(\coideal{\Cc}, \Pc)$ through the bijection, a direct computation using that:
\begin{itemize}
\item
the curvature on $\scLie^\antishriek$ is zero, and that
\item
the infinitesimal decomposition map of $\scLie^\antishriek$ is given in Equation \eqref{eq: coproduct scLiea},
\end{itemize}
shows that $\Upsilon_\alpha$ is a solution of the curved Maurer--Cartan equation if and only if $\alpha$ is a solution of the curved Maurer--Cartan equation. 

Then, given an operadic curved twisting morphism $\alpha \in \Tw(\Cc, \Pc)$, the previous bijection provides a corresponding lax morphism $(f_\alpha, a_\alpha) : \scLi \to \Hom(\Cc, \Pc)$, where $a_\alpha(1) = \alpha(1) : \Cc(1) \to \Pc(1)$. 
By Lemma \ref{lem: lax morphism unlax}, this is equivalent to a morphism of curved operads $f_\alpha : \scLi \to \Hom(\Cc, \Pc)_{a_\alpha}$. 
This finishes the proof.
\end{proof}

\subsubsection{Bar-cobar resolution}
\label{sec: bar-cobar resolution}

In order to make use of the known bar-cobar resolution in the dg context, we put some weight condition on our curved operad. 
The weight filtration should not be confused with the complete filtration $F_\bullet$ of the objects of the base category.

We call an $\Sb$-module $M$ \emph{weight filtered} if it is endowed with an $\Nb$-filtration
\[
M = M^{(\leq 0)} \subset M^{(\leq 1)} \subset \cdots \subset M^{(\leq w)} \subset \cdots.
\]
We assume moreover that the structural maps of $M$ (the maps defining the actions of the symmetric groups and the predifferential) and the maps of weight filtered $\Sb$-modules are compatible with the weight filtration. 
Because we are working over a field of characteristic 0, we can find (by Maschke's Theorem) a splitting on the level of $\Sb$-modules of the filtration. 
We denote by $M^{(w)}$, for $w \geq 0$, the elements of weight $w$ when such a splitting is fixed. 

We call a dg operad $\Qc$ \emph{weight filtered} if its underlying $\Sb$-module is weight filtered and if its structural maps (the composition map, the differential, \ldots) are weight filtered. 
We say that the weight filtration is \emph{connected} if $\Qc^{(0)} = \Qc^{(\leq 0)} = I$. 
Moreover, we call $\Qc$ \emph{bounded below} when for each fixed $w \geq 0$, the chain complex underlying $\Qc^{(\leq w)}$ is bounded below.

\begin{rem}
In the next proposition, we consider a curved operad such that $\Gr \Pc$ is weight filtered. Be aware of the fact that the filtration $F_p \Pc$ differs in general from the weight filtration $(\Gr \Pc)^{(\leq w)}$.
\end{rem}

We work here the model category structure on curved operads described in \cite[Appendix C]{BMDC20} and recalled in Appendix \ref{appendix: model cat struct}. 
In that article, the notation $\Tc (\vartheta I)$ stands for the initial object of the category of curved operads with non-zero curvature (that is $\vartheta$ is a degree $-2$ element in filtration degree 1). 
We say that a morphism of curved operads $f : \Rc \to \Pc$ is an \emph{$\Sb$-cofibrant resolution} when it is a graded quasi-isomorphism and a strict surjection (that is $f$ a surjection such that for all $p \geq 0$, $f(F_p\Rc) = f(\Rc) \cap F_p\Pc$) and when moreover the gr-dg $\Sb$-module map $\Tc (\vartheta I) \to \Rc$ is cofibrant. 
Furthermore, we say that the filtration $F_\bullet$ on an object $M$ comes from the graduation $\Gr_q M$ when $F_p M = \oplus_{q \geq p} \Gr_q M$. 

\begin{thm}
\label{thm: bar-cobar resolution}
Let $\Pc$ be a curved operad such that $\Gr \Pc$ is a connected bounded below weight filtered operad. The counit of the adjunction
\[
\hat\Omega \B \Pc \to \Pc
\]
obtained from Theorem \ref{thm: bar-cobar adjunction} is a graded quasi-isomorphism and a strict surjection.

Assuming moreover that the filtration $F_\bullet$ on the curved operad $\Pc$ comes from a graduation $\Gr_\bullet$ (on the level of $\Sb$-modules) and that the curvature of $\Pc$ is concentrated in $\Gr_1 \Pc$, then the map $\cfree(0) = \Tc(\vartheta I) \to \hat\Omega \B \Pc$ is cofibrant in the model category structure on curved operad described in \cite[Appendix C]{BMDC20}.
\end{thm}

\begin{proof}
We denote by $\Omega$ the cobar construction of a coaugmented curved cooperad and by $\Bm$ the bar construction of an sa dg operad. These constructions appear \cite{HM12} (see also the corrigendum \cite{HM23}). 
The functor $\Gr$ is strong monoidal since $\Kf$ is a field so, by a careful inspection of the differentials, $\Gr \hat\Omega \B \Pc \cong \Omega \Bm \Gr \Pc$. 
By Theorem 3.4.4 in \cite{HM12}, we obtain that $\hat\Omega \B \Pc \to \Pc$ is a graded quasi-isomorphism. 
It is a strict surjection since it is the composite of the two strict surjections $\hat\Omega \B \Pc \twoheadrightarrow \Tc (\overline{\Pc}) \xrightarrow{\bar\gamma} \Pc$.

When the curvature $\theta_\Pc$ is non zero, the curvature $\theta_\omega = \theta_\Pc$ of $\hat\Omega \B \Pc$ is also non zero and we have $\ker ([\theta_\Pc, -]) = \Tc(\theta_\Pc) \cong \Tc (\vartheta I)$ by the fact that the graded operad underlying the cobar construction is free. 
Assume now that the filtration $F_\bullet$ on $\Pc$ comes from a graduation $\Gr_q \Pc$. 
We now show that the map $\Tc(\vartheta I) \to \Om \B \Pc$, sending $\vartheta$ to $\theta_\omega = \theta_\Pc$, is a cofibration. 
We recall the following notations from the definitions of the bar and cobar constructions. The gr-dg differential on the cobar construction is $d_\omega = D_0 + D_1 + D_2$, where $D_1$ is induced by the pre-differential on the bar construction $d_0 + d_1 + d_2$.

We first define the map $h : \B\Pc \to \B \Pc$ by
\begin{align*}
h(t) & = -t' \text{ when } t = \vcenter{\tiny{
\xymatrix@M=2pt@R=4pt@C=4pt{& \cdots &\\
& t' \ar@{-}[lu] \ar@{-}[ur] \ar@{-}[d] &\\ & s\theta \ar@{-}[d] &\\ &&}}} \qquad \text{ and }\\
h(t) & = 0 \text{ when the bottom tree of $t$ is in $V$ satisfying $\Kf \theta_\Pc \oplus V = s\overline{\Pc}$.}
\end{align*}
This defines a contracting homotopy for $d_0$ on $\B \Pc$ as  we  can see by a direct computation:
\[
(d_0 h+hd_0)\left( \vcenter{\tiny{
\xymatrix@M=2pt@R=4pt@C=4pt{& \cdots &\\
& t' \ar@{-}[lu] \ar@{-}[ur] \ar@{-}[d] &\\ & s\theta \ar@{-}[d] &\\ &&}}}\right) = d_0(-t') + h\left( -\vcenter{\tiny{
\xymatrix@M=2pt@R=4pt@C=4pt{& \cdots &\\
& d_0(t') \ar@{-}[lu] \ar@{-}[ur] \ar@{-}[d] &\\ & s\theta \ar@{-}[d] &\\ &&}}} - \vcenter{\tiny{
\xymatrix@M=2pt@R=4pt@C=4pt{& \cdots &\\
& t' \ar@{-}[lu] \ar@{-}[ur] \ar@{-}[d] &\\ & s\theta \ar@{-}[d] &\\ & s\theta \ar@{-}[d] &\\ &&}}}\right) = \vcenter{\tiny{
\xymatrix@M=2pt@R=4pt@C=4pt{& \cdots &\\
& t' \ar@{-}[lu] \ar@{-}[ur] \ar@{-}[d] &\\ & s\theta \ar@{-}[d] &\\ &&}}}
\]
and when the bottom vertex $t_0$ of $t = \vcenter{\tiny{
\xymatrix@M=2pt@R=4pt@C=4pt{t_1 && t_k\\
& t_0 \ar@{-}[lu] \ar@{-}[u] \ar@{-}[ur] \ar@{-}[d] &\\ &&}}}$ is in $V$
\[
(d_0 h+hd_0)(t) = h\left( \vcenter{\tiny{
\xymatrix@M=2pt@R=4pt@C=4pt{& \cdots &\\
& t \ar@{-}[lu] \ar@{-}[ur] \ar@{-}[d] &\\ & -s\theta \ar@{-}[d] &\\ &&}}} + \sum_i (-1)^{|t_0|}\vcenter{\tiny{
\xymatrix@M=2pt@R=4pt@C=4pt{& d_0(t_i)&\\
& t_0 \ar@{-}[lu] \ar@{-}[u] \ar@{-}[ur] \ar@{-}[d] &\\ &&}}}\right) = t
\]
because $d_0$ is a coderivation. 
As a consequence, we have $\ker d_0 = \im d_0$ on $\B \Pc$. 
We decompose $\overline{\B \Pc}$ by the number of vertices $\overline{\B \Pc} = \oplus_{l \in \Nb^*} \Tc(s\overline{\Pc})^{(l)}$. 
We denote by $V_1 = V$ the direct complement of $\Kf \theta_\Pc$ in $s^{-1}(s\overline{\Pc}) = \overline{\Pc}$ as above. 
We fix
\[
\left\{\begin{array}{lcll}
S_0 & \coloneqq & \{0\},\\
S_1 & \coloneqq & \ker({\Gr d_\omega}_{|V_1}) \oplus D_1\ker({\Gr d_\omega}_{|V_1}) & \subset s^{-1}(s\overline{\Pc} \oplus \Tc(s\overline{\Pc})^{(2)}),\\
S_2 & \coloneqq & S_1 \oplus W_1 \oplus D_1 W_1 & \subset s^{-1}(s\overline{\Pc} \oplus \Tc(s\overline{\Pc})^{(2)}),
\end{array}\right.
\]
where $W_1$ is a direct complement of $\ker({\Gr d_\omega}_{|V_1})$ into $V_1$ (such a direct complement exists since $\Kf$ is a field of characteristic 0 and $\Kf[\Sb_n]$ is semi-simple for every $n \in \Nb$) and the inclusions hold  because the filtration $F_\bullet$ comes from a graduation. 
A direct sum appears in the definition of $S_1$ and $S_2$ since $V_1 \cap \ker d_0 = V_1 \cap \im {d_0}_{|I} = V_1 \cap \Kf \theta_\Pc = \{ 0\}$.
For $l \geq 2$, we fix by induction a direct complement $V_l$ of $s^{-1}\overline{\Tc}(s\overline{\Pc})^{(\leq l-1)} \oplus D_1 V_{l-1}$ in $s^{-1}\overline{\Tc}(s\overline{\Pc})^{(\leq l)}$, that is
\[
s^{-1}\overline{\Tc}(s\overline{\Pc})^{(\leq l)} = s^{-1}\overline{\Tc}(s\overline{\Pc})^{(\leq l-1)} \oplus D_1 V_{l-1} \oplus V_l.
\]
Moreover let $W_l$ be a direct complement of $\ker {\Gr d_\omega}_{|V_l}$ into $V_l$. 
Then we define for each $l \geq 2$
\[
\left\{\begin{array}{lcl}
S_{2l-1} & \coloneqq & S_{2(l-1)} \oplus \ker({\Gr d_\omega}_{|V_l}) \oplus D_1\ker({\Gr d_\omega}_{|V_l}),\\
S_{2l} & \coloneqq & S_{2l-1} \oplus W_l \oplus D_1W_l.
\end{array}\right.
\]
As above, the second direct sum which appears in the definition of $S_{2l-1}$ and $S_{2l}$ follows from the following argument. For $v \in V_l$, $D_1v \in s^{-1}\overline{\Tc}(s\overline{\Pc})^{(\leq l)}$ implies $d_0v = 0$ so $v \in \ker {d_0}_{s^{-1}\overline{\Tc}(s\overline{\Pc})^{(l)}} = \im {d_0}_{s^{-1}\overline{\Tc}(s\overline{\Pc})^{(l-1)}}$ where
\[
s^{-1}\overline{\Tc}(s\overline{\Pc})^{(l-1)} \subset s^{-1}\overline{\Tc}(s\overline{\Pc})^{(\leq l-1)} = s^{-1}\overline{\Tc}(s\overline{\Pc})^{(\leq l-2)} \oplus V_{l-1} \oplus D_1 V_{l-2}
\]
so there exists $v' \in s^{-1}\overline{\Tc}(s\overline{\Pc})^{(\leq l-2)} \oplus V_{l-1} \oplus D_1 V_{l-2}$ such that $v = D_1v' -d_1v'-d_2v'$. 
Because $D_1(s^{-1}\overline{\Tc}(s\overline{\Pc})^{(\leq l-2)} \oplus D_1 V_{l-2}) \subset s^{-1}\overline{\Tc}(s\overline{\Pc})^{(\leq l-1)}$, we get
\[
v \in (s^{-1}\overline{\Tc}(s\overline{\Pc})^{(\leq l-1)} \oplus D_1 V_{l-1})\cap V_l = \{0\}.
\]

By definition, $\left(S_l \right)_l$ is an exhaustive increasing filtration of $s^{-1}\overline{\Tc}(s\overline{\Pc})$ and it satisfies $S_{l-1} \rightarrowtail S_l$ are split monomorphisms  of $\Sb$-modules with cokernels isomorphic to
\[ S_l /S_{l-1} \cong \coprod_\alpha \left(\xi^\alpha \cdot \Kf[\Sb_{m_\alpha}] \oplus \zeta^\alpha \cdot \Kf[\Sb_{m_\alpha}]\right) \]
where $\xi^\alpha$ is in homological degree $n_\alpha +1$ and filtration degree $q_\alpha$ and $\zeta^\alpha$ is in homological degree $n_\alpha $ and filtration degree $q_\alpha+1$. The predifferential $d_\omega$ is such that $d_\omega(\xi^\alpha) + \zeta^\alpha \in \left(\Tc_+ (S_{l-1}),\, d_\omega\right)$  since $d_\omega = D_0+D_1+D_2$ with $(D_0+D_2)(v) \in \Tc_+(S_{l-1})$ for $v \in S_l$ and because $d_0$ increases exactly the filtration degree $F_\bullet$ by 1 since $\theta_\Pc \in \Gr_1 \Pc$. Moreover $d_\omega(\zeta^\alpha)$ is obtained by the fact that $d_\omega^2(\xi^\alpha) = [\theta_\Pc,\, \xi^\alpha]$. 
This is precisely the description of cofibrant curved operads  described in \cite[Proposition C.33]{BMDC20} hence $\hat\Omega \B \Pc$ is cofibrant.
\end{proof}

\subsection{Curved Koszul duality}
\label{sec: CKD}

In this section, we merge Section 4 in \cite{HM12} and Section 5 in \cite{BMDC20} in order to obtain a generalisation of Koszul duality to the context of curved inhomogeneous quadratic operads.

\begin{defn}
\label{def: inhomogeneous quadratic}
A curved operad $(\Pc, d_\Pc, \theta_\Pc)$ is called \emph{inhomogeneous quadratic} if the following requirements are satisfied.
\begin{enumerate}
\item
The curved operad $\Pc$ admits a presentation of the form
\[
\Pc = (\Tc(E)/(R), d_\Pc = 0, \theta_{\Pc}),
\]
where $E$ is a $\Sb$-module and $(R)$ is the ideal generated by a strict sub-$\Sb$-module $R \subset I \oplus E \oplus \Tc(E)^{(2)}$ and the curvature $\theta_\Pc$ is induced by a map $\tilde{\theta}_{\Pc} : I \to F_1 \Tc(E)^{(2)}$. 
The superscript $(2)$ indicates the number of vertices and defines the weight degree. 
\item
The sub-$\Sb$-module $R$ is a direct sum of (homological) degree homogeneous subspaces. 
\end{enumerate}
A first consequence is that the curved operad $\Pc$ is degree graded and has a weight filtration $\Pc^{(\leq w)}$ induced by the $\Sb$-module of generators $E$. 
We denote the associated graded (curved) operad by $\Pc^{(\bullet)}$. 
Let $\q : \Tc(E) \twoheadrightarrow \Tc(E)^{(2)}$ be the canonical projection and let $\q R \subset \Tc(E)^{(2)}$ be the image under $\q$ of $R$. 
We assume further that the following conditions hold:
\begin{itemize}
\item[(I)] The space of generators is minimal, that is $R \cap \{I \oplus E\} = \{ 0\}$.
\item[(II)] The space of relations is maximal, that is $(R)\cap \{I \oplus E \oplus \Tc(E)^{(2)}\} = R$.
\item[(III)] The quadratic part of the relations satisfies $\q R \cap F_1\Tc(E) = \{ 0\}$.
\end{itemize}
We assume moreover that the curved operad $\q\Pc := \left(\Tc(E)/(\q R), 0, \theta_{\q\Pc} \coloneqq [\tilde{\theta}_\Pc]\right)$ is homogeneous quadratic (in the sense of Definition 5.7 of \cite{BMDC20}), that is
\begin{enumerate}
\item[(3)]
the counit $\Tc^c(sE) \to I$ does not factor through the coideal quotient $(S)$, where we fix
\[
\Cc \coloneqq \Tc^c(sE) \text{ and } S \coloneqq \left( I \oplus \Tc^c(sE)^{(2)}\right)/\left( s^2 \q R \oplus (\id+s^2 \tilde{\theta}_\Pc)(I)\right).
\]
\end{enumerate}
\end{defn}

Since $R \cap \{I \oplus E\} = \{0\}$ and $\q R \cap F_1\Tc(E) = \{ 0\}$, there exists an $\Kf$-linear map $\varphi : \q R \rightarrow I \oplus E$ that we assume to be filtered such that $R$ is the graph of $\varphi$:
\begin{eqnarray*}
R & = & \{X + \varphi(X),\, X \in \q R\}\\
& = & \{X + \varphi_{1}(X) + \varphi_{0}(X),\, X\in \q R,\, \varphi_{1}(X)\in E,\, \varphi_{0}(X)\in I\}.
\end{eqnarray*}
The relations $\q R$ hold in $\Pc^{(\bullet)}$. 
Therefore, there exists an epimorphism of graded (curved) operads
\[
p : \q\Pc \twoheadrightarrow \Pc^{(\bullet)}.
\] 
We recall from \cite[Section 5]{BMDC20} that the \emph{Koszul dual cooperad} of the homogeneous quadratic curved operad $\q\Pc$ is the sub-cooperad of $\Tc^c(sE)$ generated by the strict epimorphism $\Tc^c(sE) \twoheadrightarrow S$ and we denote it by:
\[
\q\Pc^{\antishriek} := \Cc\left(sE,\, s^{2}\q R \oplus (\id+s^2 \tilde{\theta}_\Pc)(I)\right).
\]
It is a sub-cooperad of the cofree cooperad $\Tc^{c}(sE)$ on $sE$ and its counit is given by the composite $\varepsilon_{\q\Pc^{\antishriek}} : \q\Pc^\antishriek \rightarrowtail \Tc^c(sE) \twoheadrightarrow I$.  
By the definition of $\q\Pc^\antishriek$, since $\q\Pc^\antishriek \to \Tc^c(sE) \twoheadrightarrow S$ is zero, we obtain that the projection $\q\Pc^\antishriek \twoheadrightarrow I \oplus \Tc^c(sE)^{(2)}$ factors through $s^2 \q R \oplus (\id+s^2 \tilde{\theta}_\Pc)(I)$ and we can define the projection map $p^{(2)} : \q\Pc^\antishriek \twoheadrightarrow s^2\q R$. 
There exists a unique coderivation $d_{\Pc^\antishriek} : \q\Pc^{\antishriek} \rightarrow \Tc^{c}(sE)$ of degree $-1$ which extends the map
\[
\q\Pc^\antishriek \twoheadrightarrow s^{2}\q R \xrightarrow{-s^{-1} \otimes \varphi_{1}} sE.
\]
Moreover, we denote by $\theta_{\Pc^\antishriek}^c : \q\Pc^{\antishriek} \rightarrow I$ the map of degree $-2$
\[
\theta_{\Pc^\antishriek}^c :\q\Pc^{\antishriek} \twoheadrightarrow s^{2}\q R \xrightarrow{-s^{-2} \otimes \varphi_{0}} I.
\]

\begin{lem}
Let $\Pc = \Tc(E)/(R)$ be an inhomogeneous quadratic curved operad. Condition (II) implies that:
\begin{itemize}
\item The image of $d_{\Pc^\antishriek}$ lies in $\q\Pc^\antishriek = \Cc\left(sE,\, s^{2}\q R \oplus (\id+s^2 \tilde{\theta}_\Pc)(I)\right)$;
\item The coderivation $d_{\Pc^\antishriek}$ satisfies ${d_{\Pc^\antishriek}}^{2} = (\id_{\q\Pc^\antishriek} \otimes \theta_{\Pc^\antishriek}^c - \theta_{\Pc^\antishriek}^c \otimes \id_{\q\Pc^\antishriek}) \cdot \Delta_{(1)}$;
\item The coderivation $d_{\Pc^\antishriek}$ satisfies $\theta_{\Pc^\antishriek}^c \cdot d_{\Pc^\antishriek} = 0$.
\end{itemize}
\end{lem}

\begin{proof}
Proposition 5.5 in \cite{BMDC20} says that
\[
\q\Pc^\antishriek = \ker \left( \Cc \xrightarrow{(\id \otimes \Delta) \cdot \Delta_{(1)}} \Cc \hcirc_{(1)} \Cc^{\hcirc 2} \xrightarrow{\id_\Cc \circ_{(1)} \xi \circ \id_\Cc} \Cc \hcirc_{(1)} (S \hcirc \Cc) \right),
\]
where $\Cc = \Tc^c(sE)$ and $S = \left( I \oplus \Tc^c(sE)^{(2)}\right)/\left( s^2 \q R \oplus (\id+s^2 \tilde{\theta}_\Pc)(I)\right)$. 
Applying $d_{\Pc^\antishriek}$ to an element in $\q\Pc^\antishriek$, we obtain that the result lands in $\q\Pc^\antishriek$ whenever $\xi \cdot d_{\Pc^\antishriek} =0$ (since $d_{\Pc^\antishriek}$ is a coderivation and by the characterisation of $\Pc^\antishriek$ just recalled). 
The cooperad $\q\Pc^\antishriek$ is a sub-cooperad of $\Tc^c(sE)$ so an element $c$ in $\q\Pc^\antishriek$ can be written as a sum $c = \sum_{k\geq 0} c^{(k)}$ of trees $c^{(k)}$ in $\Tc^c(sE)^{(k)}$. 
By definition, the coderivation $d_{\Pc^\antishriek}$ decreases on each component the number of vertices by one. 
Since $\varepsilon \cdot d_{\Pc^\antishriek} =0$, we have $\xi \cdot d_{\Pc^\antishriek}(c) = \xi \cdot d_{\Pc^\antishriek}(c^{(3)})$. 

Then the proof works similarly as \cite[Lemma 4.1.1]{HM12}. 
Any kind of tree in $\Tc^c(sE)^{(3)}$ has one of the two forms
\[
\vcenter{\xymatrix@M=2pt@R=4pt@C=4pt{&&&&\\ & se_2 \ar@{-}[dr] \ar@{-}[ul] \ar@{-}[u] & \ar@{-}[d] & se_3 \ar@{-}[dl] \ar@{-}[ul] \ar@{-}[u] \ar@{-}[ur] &\\
      && se_1 \ar@{-}[d] &&\\ &&&&}}\ ,\quad \vcenter{\xymatrix@M=2pt@R=2pt@C=3pt{&&&\\ & se_3 \ar@{-}[ul] \ar@{-}[u] \ar@{-}[ur] &&\\
     && se_2 \ar@{-}[ul] \ar@{-}[ur] &\\
     & se_1 \ar@{-}[ur] \ar@{-}[u] \ar@{-}[ul] \ar@{-}[d] &&\\ &&&}}\ ,
\]
where the elements $e_i$ are in $sE$. 
Applying $\Delta_{(1)}$ to the trees above
, we obtain
\[
\left\{ \begin{array}{l}
\Delta_{(1)}\left( \vcenter{\xymatrix@M=2pt@R=4pt@C=2pt{&&&&\\ & se_2 \ar@{-}[dr] \ar@{-}[ul] \ar@{-}[u] & \ar@{-}[d] & se_3 \ar@{-}[dl] \ar@{-}[ul] \ar@{-}[u] \ar@{-}[ur] &\\
      && se_1 \ar@{-}[d] &&\\ &&&&}} \right) = 
(-1)^{|se_2||se_3|}\vcenter{\xymatrix@M=2pt@R=4pt@C=2pt{&&&&\\ & se_2 \ar@{-}[ddr] \ar@{-}[ul] \ar@{-}[u] &&&\\  \ar@{.}@<5.5pt>[rrrr] && \ar@{-}[d] & se_3 \ar@{-}[dl] \ar@{-}[ul] \ar@{-}[u] \ar@{-}[ur] &\\
      && se_1 \ar@{-}[d] &&\\ &&&&}}
+ \vcenter{\xymatrix@M=2pt@R=4pt@C=2pt{&&&&\\ &&& se_3 \ar@{-}[ddl] \ar@{-}[ul] \ar@{-}[u] \ar@{-}[ur] & \\ \ar@{.}@<5.5pt>[rrrr] & se_2 \ar@{-}[dr] \ar@{-}[ul] \ar@{-}[u] & \ar@{-}[d] &&\\
      && se_1 \ar@{-}[d] &&\\ &&&&}},\\
\Delta_{(1)} \left( \vcenter{\xymatrix@M=2pt@R=2pt@C=3pt{&&&\\ & se_3 \ar@{-}[ul] \ar@{-}[u] \ar@{-}[ur] &&\\
     && se_2 \ar@{-}[ul] \ar@{-}[ur] &\\
     & se_1 \ar@{-}[ur] \ar@{-}[u] \ar@{-}[ul] \ar@{-}[d] &&\\ &&&}} \right) = \quad \vcenter{\xymatrix@M=2pt@R=2pt@C=3pt{&&&\\ & se_3 \ar@{-}[ul] \ar@{-}[u] \ar@{-}[ur] &&\\ &&&\\
    \ar@{.}@<5.5pt>[rrr] && se_2 \ar@{-}[uul] \ar@{-}[ur] &\\
     & se_1 \ar@{-}[ur] \ar@{-}[u] \ar@{-}[ul] \ar@{-}[d] &&\\ &&&}}
\quad + \quad \vcenter{\xymatrix@M=2pt@R=2pt@C=3pt{&&&\\ & se_3 \ar@{-}[ul] \ar@{-}[u] \ar@{-}[ur] &&\\
   && se_2 \ar@{-}[ul] \ar@{-}[ur] &\\ &&&\\
   \ar@{.}@<5.5pt>[rrr] & se_1 \ar@{-}[uur] \ar@{-}[u] \ar@{-}[ul] \ar@{-}[d] &&\\ &&&}}.
\end{array} \right.
\]
It follows by a careful computation of the signs that applying
\[
\left(s^{-1} p^{(2)} \otimes p^{(1)} + p^{(1)} \otimes s^{-1} p^{(2)}\right) \cdot \Delta_{(1)}
\]
(where $p^{(1)}$ is the projection $\q \Pc^\antishriek \twoheadrightarrow sE$) to an element in $\q \Pc^{\antishriek}$ gives 0. 
Therefore, we obtain
\begin{multline*}
\Big(\big(s^{-1} \otimes (\id_{\q R} + \varphi_{1} + \varphi_{0}) \cdot p^{(2)}\big) \otimes p^{(1)} +\\ 
p^{(1)} \otimes \big(s^{-1} \otimes (\id_{\q R} + \varphi_{1} + \varphi_{0}) \cdot p^{(2)}\big)\Big) \cdot \Delta_{(1)}({\q\Pc^\antishriek})\\
\subset \{sR \otimes_{(1)} sE + sE \otimes_{(1)} sR\} \cap s^{2}\{E \oplus \Tc(E)^{(2)}\}
\subset s^2 R
\end{multline*}
by Condition \textrm{(II)} for the last inclusion. 
Denoting $s^{-1} \otimes \varphi_{1} \cdot p^{(2)}$ by $\tilde{\varphi}_{1}$ and $s^{-1} \otimes \varphi_{0} \cdot p^{(2)}$ by $\tilde{\varphi}_{0}$, we get the following system of equations
\begin{enumerate}
\item (in $\Tc(E)^{(2)}$)\hspace{.2em} $(\tilde{\varphi}_{1} \otimes p^{(1)} + p^{(1)} \otimes \tilde{\varphi}_{1}) \cdot \Delta_{(1)}({\q \Pc^\antishriek}) \subset s^2\q R$;
\item (in $sE$)\hspace{1.5em} $\big( \tilde{\varphi}_{1} \cdot ( \tilde{\varphi}_{1} \otimes p^{(1)} + p^{(1)} \otimes \tilde{\varphi}_{1}) - s^{-1}(\tilde{\varphi}_{0} \otimes p^{(1)} + p^{(1)} \otimes \tilde{\varphi}_{0})\big) \cdot {\Delta_{(1)}}_{|{\q\Pc^\antishriek}} = 0$;
\item (in $I$)\hspace{2.3em} $\tilde{\varphi}_{0} \cdot (\tilde{\varphi}_{1} \otimes p^{(1)} + p^{(1)} \otimes \tilde{\varphi}_{1}) \cdot {\Delta_{(1)}}_{|{\q\Pc^\antishriek}} = 0$.
\end{enumerate}
Equation (1) is equivalent (up to sign) to $\xi \cdot d_{\Pc^\antishriek} = 0$. 
This ensures that the image of $d_{\Pc^\antishriek}$ lies in $\q\Pc^\antishriek$. 
Equation $(2)$ corresponds to the projection of the second point of the lemma to $sE$. 
The equality extends to $\q\Pc^\antishriek$ since ${d_{\Pc^\antishriek}}^{2} = \frac{1}{2} [d_{\Pc^\antishriek},\, d_{\Pc^\antishriek}]$ and $(\id_{\q\Pc^\antishriek} \otimes \theta_{\Pc^\antishriek}^c - \theta_{\Pc^\antishriek}^c \otimes \id_{\q\Pc^\antishriek}) \cdot \Delta_{(1)}$ are coderivations (see Lemma 3.2.2 in \cite{HM12}). 
Equation $(3)$ corresponds to the third point of the lemma.
\end{proof}

\begin{defn}
\label{def: Koszul}
Let $(\Pc, \theta_\Pc)$ be an inhomogeneous quadratic curved operad with presentation $\Pc = \Tc(E)/(R)$ satisfying the conditions of Definition \ref{def: inhomogeneous quadratic}. 
\begin{itemize}
\item
The \emph{Koszul dual curved cooperad} of $(\Pc, \theta_\Pc)$ is the curved cooperad
\[
\Pc^\antishriek \coloneqq (\q\Pc^\antishriek, d_{\Pc^{\antishriek}}, \theta^c_{\Pc^\antishriek}).
\]
It is altipotent since it is a (strict) sub-cooperad of an altipotent cooperad (the cofree one). 
\item
We define the application $\kappa :\Pc^\antishriek \to \Pc$ of degree $-1$ by
\[
\kappa : \Pc^\antishriek \rightarrowtail \Tc^c(sE) \twoheadrightarrow sE \cong E \to \Pc.
\]
\item
The curved operad $(\Pc, \theta_\Pc)$ is called \emph{Koszul} when
\begin{enumerate}
\item
$(\Pc, \theta_\Pc) \cong (\widehat{\Gr \Pc}, \bar\theta_\Pc)$ with $\Gr \Pc \cong \mathrm{Free}(E)/(R)$ in \emph{filtered} operads and $\Pc = \Tc(E)/(R)$ in \emph{complete} operads is the presentation which makes the curved operad $(\Pc, \theta_\Pc)$ inhomogeneous quadratic, and
\item
$(\q\Pc, \theta_{\q\Pc}) \cong (\widehat{\Gr \q\Pc}, \bar\theta_\Pc)$ and $\Gr \q\Pc \cong \q\Gr \Pc$ as (curved) quadratic operads, and
\item
the inhomogeneous quadratic operad $\Gr \Pc$ is Koszul in the sense of \cite[4.3]{HM12} or equivalently the curved quadratic operad $\q\Pc$ is Koszul in the sense of \cite[Definition 5.20]{BMDC20}, that is $\q\Gr \Pc \cong \Gr\q\Pc$ is a Koszul (quadratic) operad. 
\end{enumerate}
\end{itemize}
\end{defn}

\begin{lem}
Let $(\Pc, \theta_\Pc)$ be an inhomogeneous quadratic curved operad as in the previous definition. 
The degree $-1$ map $\kappa$ is an operadic curved twisting morphism, that is
\[
\Theta_H + \partial (\kappa) + \frac{1}{2} [\kappa, \kappa] = 0,
\]
where $\Theta_H = \theta_\Pc \cdot \varepsilon_{\Pc^\antishriek} + \eta_\Pc \cdot \theta_{\Pc^\antishriek}^c$ and $\partial (\kappa) = d_\Pc \cdot \kappa + \kappa \cdot d_{\Pc^\antishriek} = \kappa \cdot d_{\Pc^\antishriek}$ since $d_\Pc = 0$.
\end{lem}

\begin{proof}
As in \cite[Section 5.17]{BMDC20}, we have
\[
\eta_{\Pc^\antishriek}(1) = 1 + \left( s^2\tilde\theta_\Pc + \sum s^2 qr\right) + \cdots \in I \oplus \Tc^c(sE)^{(2)} \oplus \cdots
\]
with $\sum s^2 qr \in s^2 \q R$. 
On $\eta_{\Pc^\antishriek}(I)$, we therefore obtain
\begin{multline*}
\left(\Theta_H \cdot \eta_{\Pc^\antishriek} + \partial (\kappa) \cdot \eta_{\Pc^\antishriek} + \frac{1}{2} [\kappa, \kappa] \cdot \eta_{\Pc^\antishriek}\right)(1) =\\
\theta_\Pc - \eta_\Pc \cdot \sum \varphi_0(qr) - \sum\varphi_1(qr) - \theta_\Pc - \sum qr = 0 \in \Pc.
\end{multline*} 
Then, on $\overline{\Pc^\antishriek} \subset sE \oplus \Tc^c(sE)^{(2)} \oplus \cdots$, because the decomposition product appearing in the definition of $[-, -]$ is the cofree one and by the definition of $\kappa$ and $d_{\Pc^\antishriek}$, it is enough to consider elements in $\overline{\Pc^\antishriek}$ whose projection on $\Tc^c(sE)^{(2)}$ is non zero. 
Again by \cite[Section 5.17]{BMDC20}, these elements write $\sum s^2 qr + \cdots \in \Tc^c(sE)^{(2)} \oplus \Tc^c(sE)^{(\geq 4)}$, with $\sum s^2 qr \in s^2 \q R$. 
On such an element, we get
\begin{multline*}
\left(\Theta_H + \partial (\kappa) + \frac{1}{2} [\kappa, \kappa]\right)\left(\sum s^2 qr + \cdots\right) =\\
\left(\eta_\Pc \cdot \theta^c_{\Pc^\antishriek} + \kappa \cdot (-s^{-1} \otimes \varphi_1) + \frac{1}{2} [\kappa, \kappa]\right)\left(\sum s^2 qr\right) =\\
-\sum \left(\varphi_0 + \varphi_1 + \id_{\q R}\right)\left(\sum qr\right) = 0 \in \Pc.
\end{multline*}
\end{proof}

As a consequence of this lemma and of Theorem \ref{thm: bar-cobar adjunction}, we can associate to $\kappa$ two maps:
\begin{itemize}
\item
$f_\kappa : \hat\Omega \Pc^\antishriek \to \Pc$ is a morphism of curved sa operads,
\item
$g_\kappa : \Pc^\antishriek \to \B\Pc$ is a morphism of curved altipotent cooperads.
\end{itemize}

\begin{thm}
\label{thm: Koszul resolution}
Let $\left(\Pc = \Tc(E)/(R), \theta_\Pc\right)$ be a (presented) Koszul inhomogeneous quadratic curved operad such that $\Gr\q\Pc$ is connected bounded below weight graded. 
The curved operad morphism
\[
f_\kappa : \hat\Omega \Pc^\antishriek \to \Pc
\]
is a graded quasi-isomorphism and a strict surjection.

Assuming moreover that the filtration $F_\bullet$ on $\Pc$ comes from a graduation and that the curvature $\theta_\Pc$ is non zero and is concentrated in $\Gr_1 \Pc$, then the map $\cfree(0) = \Tc(\vartheta I) \to \hat\Omega \Pc^\antishriek$ is an $\Sb$-cofibration, that is, a cofibration in the underlying category of $\Sb$-modules, in the model category structure on gr-dg $\Sb$-modules described in \cite[Appendix C]{BMDC20}.
\end{thm}

\begin{proof}
The proof follows the same lines as the proof of Theorem \ref{thm: bar-cobar resolution}. 
The functor $\Gr$ is strong monoidal since $\Kf$ is a field of characteristic 0 so $\Gr \hat\Omega \Pc^\antishriek \cong \Omega \Gr (\Pc^\antishriek)$ as dg operads (by a careful inspection of the differentials). 
Moreover, on the level of $\Sb$-modules, we have the following isomorphisms
\[
\Gr (\Pc^\antishriek) = \Gr (\q\Pc^\antishriek) \cong (\Gr \q\Pc)^\antishriek \cong (\q\Gr \Pc)^\antishriek
\]
where the first isomorphism follows from the Poincaré--Birkhoff--Witt type isomorphism in \cite[Theorem 5.21]{BMDC20} (since $\q\Pc$ is a curved quadratic Koszul operad) and the second isomorphism is true since $\Pc$ is Koszul. 
Since $\Pc$ is Koszul, the data $(E, R)$ of generators and relations of $\Pc$ and $\Gr \Pc$ are the same. 
It follows that $\Gr (d_{\Pc^\antishriek}) = d_{(\Gr \Pc)^\antishriek}$ and $\Gr (\theta^c_{\Pc^\antishriek}) = \theta^c_{(\Gr\Pc)^\antishriek}$. 
Therefore, we obtain by \cite[Theorem 4.3.1]{HM12} the quasi-isomorphism
\[
\Gr \hat\Omega \Pc^\antishriek \cong \Omega ((\Gr \Pc)^\antishriek) \xrightarrow{\sim} \Gr \Pc.
\]
This says precisely that $f_\kappa$ is a graded quasi-isomorphism. 
The map $f_\kappa$ is a strict surjection since it is the composite of the two strict surjections $\hat\Omega \Pc^\antishriek \twoheadrightarrow \Tc(E) \twoheadrightarrow \Pc$. 

When $\theta_\Pc \neq 0$, the curvature $\theta_\omega$ of $\hat\Omega \Pc^\antishriek$ is also non zero and we have $\ker ([\theta_\omega, -]) = \Tc(\theta_\omega) \cong \Tc (\vartheta I)$ by the fact that the graded operad underlying the cobar construction is free. 
Assume now that the filtration $F_\bullet$ on $\Pc$ comes from a graduation $\Gr_p \Pc$ and that $\theta_\Pc$, hence $\theta_\omega$ is concentrated in $\Gr_1 \Pc$. 
Then, as in \cite[Theorem 5.24]{BMDC20}, we get that the underlying gr-dg $\Sb$-module map $\Tc (\vartheta) \to \hat\Omega \Pc^\antishriek$ is cofibrant, so $\hat\Omega \Pc^\antishriek \to \Pc$ is an $\Sb$-cofibration.
\end{proof}

\begin{rem}
\begin{enumerate}
\item
A priori, we cannot hope in general for $\Om \Pc^\antishriek$ to be cofibrant in the model category structure on curved operad (obtained by transfer) since its curvature is a linear combination of products.
\item
Nevertheless, by Theorem C.45 in \cite{BMDC20}, the model categories of $\Om \Pc^\antishriek$-algebras and of $\Om \B \Pc$-algebras are both equivalent to the model category of $\Pc$-algebras (because all operads are $\Sb$-split in characteristic 0). 
In particular, up to equivalence, the $\infty$-category of homotopy $\Pc$-algebras doesn't depend of a choice of resolution.
\end{enumerate}
\end{rem}

\section{Examples of inhomogeneous quadratic curved operads}

In this section, we make the case of the $\Kf$-linear operad encoding curved unital associative algebras and the $\Rb$-linear operad encoding curved Lie $\Cb$-algebras explicit. 
We show that they form Koszul inhomogeneous quadratic curved operads.

To shorten the notations, we use in this section the convention $[n] \coloneqq \{ 1, \ldots, n\}$ for an integer $n$.

\subsection{Curved unital associative algebras}

We consider the curved operad presented as follows
\[
\cuAs \coloneqq \left(\Tc\left(\curvAd, \curvA, \as \right)/\left(\ass, \lunit - |, \runit - | \right), 0, \theta \coloneqq \curvAs\right).
\]
where $\curvA$ is of degree $-2$, $\curvAd$ and $\as$ are of degree 0 and the predifferential is zero.
We filter $\cuAs$ by the number of $\curvA$, say
\[
F_p \cuAs \coloneqq \{ \mu \in \cuAs \textrm{ s.t. the number of } \curvA \textrm{ in } \mu \textrm{ is greater than or equal to } p\}.
\]
(This filtration is in fact a graduation.) 
It is a curved operad whose curvature belongs to $F_1 \cuAs$. 
Indeed, we have already seen in \cite[Proposition 6.2]{BMDC20} that the bracket of the curvature with the elements in $\cAs \subset \cuAs$ is 0. 
It is therefore enough to remark that $[\theta, \curvAd] = 0$, which is true because $\curvAd$ is a unit for $\as$.

\begin{lem}
A $\cuAs$-algebra on a gr-dg module $A$ is the same data as a curved unital associative algebra $(A, \mu, \eta_A, d_A, \theta_A)$ with curvature $\theta_A \in F_1 A$.
\end{lem}

\begin{proof}
A map of curved operad $\cuAs \to \End_A$ is characterized by the image of the generators $\curvAd$, $\curvA$ and $\as$ which give respectively three maps: $\eta_A : \Kf \to A$ of degree 0, $\theta_A : \Kf \to F_1 A$ of degree $-2$ and $\mu : A^{\otimes 2} \to A$ of degree 0. 
The relation defining $\cuAs$ ensures that $\mu$ is associative and that $\eta_A$ is a unit for $\mu$. 
The fact that the curvature is sent to the curvature says finally ${d_A}^2 = [\theta_A,\, -]$.
\end{proof}

We now describe our sign convention for a basis of the cooperad $\As^\antishriek$ and the consequences for the formula describing its decomposition map. 
Following \cite[Section 9.1.5]{LV12} we note $\mu_n^c = -\sum_{t \in PBT_n} \mathrm{sgn}(\tilde t)t$ for $n \geq 3$ (where the sum runs over planar binary trees with $n$ leaves identified with $\Tc(\mu^c_2)(n)$) and we fix $\tilde\mu_n^c \coloneqq (-1)^{\frac{(n-1)(n+2)}{2}}\mu_n^c$ for $n \geq 0$. 
This change of sign gives for example $\tilde\mu_0^c = -\mu_0^c$, $\tilde\mu_1^c = \mu_1^c$ and $\tilde\mu_2^c = \mu_2^c$. 
With this basis, we obtain the following formula for the decomposition map.

\begin{lem}
\label{lem: change of convention}
The decomposition map $\Delta$ on $\As^\antishriek = \{ \tilde\mu_n^c\}_{n \geq 1}$ is given by
\[
\Delta(\tilde\mu_n^c) = \sum_{i_1 + \cdots + i_m = n} (-1)^{\sum_j (i_j-1)\sum_{l<j}i_l} (\tilde\mu_m^c ;\, \tilde\mu_{i_1}^c ,\, \cdots ,\, \tilde\mu_{i_m}^c).
\]
\end{lem}

\begin{proof}
In \cite[Lemma 9.1.2]{LV12}, the formula given for the decomposition map of $\As^\antishriek = \{ \mu_n^c\}_{n \geq 1}$ is
\begin{align*}
\Delta(\mu_n^c) & = \sum_{i_1 + \cdots + i_m = n} (-1)^{\sum (i_j-1)(m-j)} (\mu_m^c ;\, \mu_{i_1}^c ,\, \cdots ,\, \mu_{i_m}^c).
\end{align*}
We consider the isomorphism of $\As^\antishriek$ given by $\mu_n^c \mapsto \tilde\mu_n^c \coloneqq (-1)^{\frac{(n-1)(n+2)}{2}}\mu_n^c$. 
We compute
\begin{align*}
\frac{(n-1)(n+2)}{2} &+\sum_j (i_j-1)(m-j) - \frac{(m-1)(m+2)}{2} - \sum_j \frac{(i_j-1)(i_j+2)}{2} =\\
& = \frac{1}{2}\left(\sum_j (i_j-1) + (m-1)\right)(n+2) +\sum_j (i_j-1)(m-j)\\
& \qquad - \frac{(m-1)(m+2)}{2} - \sum_j \frac{(i_j-1)(i_j+2)}{2}\\
& = \frac{1}{2}\left(\sum_j (i_j-1)\left( i_j + 2 + \sum_{l \neq j}i_l\right) + (m-1)(m+2 + n-m)\right)\\
& \qquad +\sum_j (i_j-1)(m-j)- \frac{(m-1)(m+2)}{2} - \sum_j \frac{(i_j-1)(i_j+2)}{2} \\
& = \frac{1}{2} \left( \sum_j ( i_j - 1)\left(\sum_{l \neq j} i_l + 2(m-j)\right) + (m-1)(m-n) \right)\\
& = \frac{1}{2} \sum_j \sum_{l \neq j} i_j i_l + \frac{n(m-1)}{2} + m(n-m) - \sum_j i_j j+ \frac{m(m+1)}{2}\\
& \qquad + \frac{(m-1)(m-n)}2\\
& = \sum_j i_j \left(\sum_{l > j} i_l - j\right) + nm = \sum_j i_j\left(\sum_{l > j} (i_l-1) +m- 2j\right) + nm \\
& \equiv \sum_j \sum_{l > j} i_j (i_l-1)  \mod 2 \equiv \sum_j \sum_{l < j} (i_j-1)i_l  \mod 2.
\end{align*}
This proves the formula.
\end{proof}

We fix some notations that we use in the next theorem. 
Let $n\geq 0$ and $k\geq 0$. For $S \subseteq [n+k]$ such that $|S| =k$, we note $\mu_{n+k}^{c, S}$ the element $\tilde\mu_{n+k}^c$ on which we have grafted the element $s\curvA$ in the positions given by the set $S$. 
Moreover, for $T \subseteq [n]$, we note $\left( \mu_{n+k}^{c, S}\right)^T$ the element $\mu_{n+k}^{c, S}$ on which we have grafted the element $s\curvAd$ in the positions given by the set $T$. 
As an example, we obtain pictorially:
\[
\left( \mu_{2+2}^{c, \{1, 3\}}\right)^{\{ 1\}} = \vcenter{
  \xymatrix@M=0pt@R=6pt@C=6pt{
    & \scriptstyle\circ & & & \\
    & *{} \ar@{-}^s[u] & & &\\
    \ar@{{*}}^s[u] & *{} \ar@{-}[u] & & \ar@{{*}}^s[u] & \\
    & & \tilde\mu_4^c \ar@{-}[d] \ar@{-}[ull] \ar@{-}[ul] \ar@{-}[ur] \ar@{-}[urr] & &\\
    & & & &}}\quad \in \Tc^c({\scriptstyle s}\curvAd, {\scriptstyle s}\curvA, {\scriptstyle s}\as).
\]

\begin{thm}
\label{thm: Koszul dual cuAs}
The inhomogeneous quadratic curved operad $\cuAs$ is Koszul. 
Its Koszul dual curved cooperad $\cuAs^{\antishriek}$ is isomorphic to
\[ \cuAs^{\antishriek} \cong \left(s\curvA \oplus \q\uAs^\antishriek = \left\{ \nu_n^T \right\}_{n \geq 0, T \subseteq [n]}, \Delta^\theta, 0, \theta^c\right), \]
where the terms $\nu_{n}^{T}$ have degree $n-1+|T|$ and for any $n \geq 0$ and $T \subseteq [n]$, we have
\begin{align}
\label{eq: decomposition map cAs}
\Delta^{\theta}( \nu_n^{T} ) \coloneqq \sum_{\substack{i_1+ \dots + i_{m'} = n-|T_0|\\ \{T_0, \dots, T_{m'}\}}} (-1)^{\epsilon} (\nu_{m}^{T_0}; \nu_{i_{1}}^{T_{1}}, \ldots, \nu_{i_{m'}}^{T_{m'}})
\end{align}
where $m' \coloneqq m-|T_0|$ and $T_0 = T_0^0 \sqcup \dots \sqcup T_0^{m'}$ with for all $j$ between 1 and $m'$, $\max T_0^{j-1} < \min T_0^j$ and
\begin{align*}
T & = T_0^0 \sqcup (T_1+|T_0^0|) \sqcup (T_0^1+i_1 -1) \sqcup \dots \sqcup\\
& \qquad \left(T_{m'} + |T_0^1|+ \sum_{j= 1}^{m'-1} (|T_0^j| + i_j)\right) \sqcup \left(T_0^{m'} + \sum_{j=1}^{m'} (i_j-1)\right)\\
& = \tilde{T}_0 \sqcup \dots \sqcup \tilde{T}_{m'}
\end{align*}
(where $\tilde{T}_j$ corresponds to the elements of $T_j$ reindexed) and where
\begin{multline*}
\epsilon \coloneqq \sum_{j=1}^{m'} (i_{j}-1)\left(\sum_{l<j}i_l+|T_0^0|+|T_1|+|T_0^1|+\cdots + |T_{j-1}|+|T_0^{j-1}|\right) +\\
|T_0|\left(n-m\right)+\sum_{j=1}^{m'} |T_0^{j}|(|T_{1}|+\cdots +|T_{j}|).
\end{multline*}
The curvature $\theta^c : \cuAs^\antishriek \to I$ is equal to
\begin{equation*}
  \theta^c (\nu_n^T) = \left\{
    \begin{array}{rl}
      | &  \text{if } n = 2 \text{ and, } T = \{1 \} \text{ or } T = \{2\}\\
      0 & \text{otherwise.}
    \end{array}
  \right.
\end{equation*}
Moreover, the natural map
\[ f_\kappa : \cuAi \coloneqq \hat\Omega \cuAs^\antishriek \to \cuAs \]
is an $\Sb$-cofibrant resolution.
\end{thm}

\begin{rem}
Even if it is only implicitly visible, the difference with the decomposition map formula given in \cite{HM12} is (up to signs) that in the above theorem the sum runs over integers $m$, $i_1$, \dots, $i_{m'}$ which can be 0.
\end{rem}

\begin{proof}
The filtration $F_\bullet$ on the curved operad $\cuAs$ is a graduation and $\cuAs$ is the completion of its associated graded. 
The presentation of $\cuAs$ given above is easily seen to be inhomogeneous quadratic, except for Condition (3) in Definition \ref{def: inhomogeneous quadratic} that will be proven during the proof. 
The presentation satisfies the fact that $\q\cuAs$ is the completion of its associated graded. 
It satisfies also that $\Gr \q\cuAs \cong \q\Gr\cuAs$ as quadratic operad. 
In order to prove the Koszulness of $\cuAs$, it remains to show that $\q\Gr\cuAs$ is a quadratic Koszul operad (see Definition \ref{def: Koszul}). 
Noting $\Pc \ast \Qc$ for the coproduct of two operads $\Pc$ and $\Qc$, we have
\begin{align*}
\q\Gr\cuAs & \cong \mathrm{Free}\left(\curvAd, \curvA, \as \right)/\left(\ass, \lunit, \runit \right) \cong (\curvAd \oplus \As) \ast \mathrm{Free}(\curvA)\\
& = \q\uAs \ast \mathrm{Free}(\curvA) \cong \curvAd \oplus \Gr\cAs.
\end{align*}
We also use the notation $\ast$ for the coproduct of two cooperads. 
By Proposition 6.8 in \cite{BMDC20}, we have
\[
(\q\Gr\cuAs)^\antishriek \cong \mathrm{coFree}(s\curvA) \ast (\q\uAs)^\antishriek \cong s\curvA \oplus \uAs^\antishriek \cong (\Gr\cAs)^\antishriek \circ \mathrm{coFree}(s\curvAd),
\]
where $\uAs^\antishriek$ has been computed in \cite[Section 6]{HM12}. 
We proceed as in the proof of Proposition 6.1.7 in \cite{HM12}. 
Denoting generically $\kappa$ the usual twisting morphism associated with a quadratic operad and making use of the notation $\circ_\kappa$ recalled in \cite{HM12}, we have
\[
\q\Gr\cuAs \circ_\kappa (\q\Gr\cuAs)^\antishriek \cong \curvAd \oplus \left\{\bigoplus_{T \subseteq [n]} (\Gr\cAs \circ_{\kappa} (\Gr\cAs)^{\antishriek}) (n) \right\}_{n \geq 1},
\]
with $\Gr \cAs \cong \As \ast \mathrm{Free}(\curvA)$ and $(\Gr\cAs)^\antishriek \cong \curvA \oplus \As^\antishriek$. 
Since $\vcenter{
  \xymatrix@M=0pt@R=2pt@C=2pt{
    & & & & \\
    & \scriptstyle\circ & & &\\
    & \ar@{-}[u] \ar@{-}[dr] & & \ar@{-}[dl] & \\
    & & *{} \ar@{-}[d] & &\\
    & & }} = 0 = \vcenter{
  \xymatrix@M=0pt@R=2pt@C=2pt{& & & & \\
    & & & \scriptstyle\circ &\\
    & \ar@{-}[dr] & & \ar@{-}[u] \ar@{-}[dl] & \\
    & & *{} \ar@{-}[d] & & \\
    & & }}$ in $\q\cuAs$, the differential on the copies $(\Gr\cAs \circ_{\kappa} \Gr\cAs^{\antishriek})
(n)$ is given by the usual differential on $\Gr\cAs \circ_{\kappa} \Gr\cAs^{\antishriek}$
except for the copy $T = \{ 1 \} \subseteq \{1\}$ where $d\left(\vcenter{
  \xymatrix@M=0pt@R=3pt@C=3pt{& \scriptstyle\circ &\\
   \ar@{.}[rr] & *{} \ar@{-}[u] \ar@{-}[d] &\\
    &&}}\right) = \curvAd$. 
Appart from this particular copy (which is clearly acyclic), each copy $\Gr\cAs \circ_{\kappa} (\Gr\cAs)^{\antishriek}(n)$ is acyclic (or quasi-isomorphic to $I$ in the case $n=1$) because $\Gr\cAs$ is Koszul (see \cite[Theorem 6.9]{BMDC20}) and we get that $\q\Gr\cuAs \circ_\kappa (\q\Gr\cuAs)^\antishriek \xrightarrow{\sim} I$. 
It follows that $\q\Gr\cuAs$ is Koszul as a quadratic operad. 
Thus $\cuAs$ is a Koszul inhomogeneous quadratic curved operad (up to Condition (3) in Definition \ref{def: inhomogeneous quadratic}).

To compute $\cuAs^\antishriek$, we can proceed as in the proof of Theorem 6.9 in \cite{BMDC20}. 
We first prove that $\Cc = \left(\left\{ \nu_n^T \right\}_{n \geq 0, T \subseteq [n]}, \Delta^\theta, 0, \theta^c\right)$ is a curved sub-cooperad of $\Tc^c\left( {\scriptstyle s}\curvAd, {\scriptstyle s}\curvA, {\scriptstyle s}\as\right)$ satisfying the property for which $(\q\cuAs)^\antishriek$ is universal. 
As a consequence $\Cc$ injects into $(\q\cuAs)^\antishriek$. 
The filtration that we consider on $\Cc$ are $F_0\Cc = \Cc$, $F_1\Cc = \Kf\cdot \nu_0^{\emptyset}$ and $F_2\Cc = 0$
. 
Then we prove that this inclusion is an isomorphism. 
First we show that the map defined for $n\geq 0$ and $T \subseteq [n]$ by
\begin{equation}
\label{eq: injection of the Koszul dual coop - cuAs}
\upsilon : \nu_n^T \mapsto \sum_{\substack{k \geq 0\\ S = \{s_j\}_{j=1}^k \subseteq [n+k], |S|=k}} (-1)^{s_1+\cdots+s_k-\frac{k(k+1)}{2}} \left(\mu_{n+k}^{c, S}\right)^T,
\end{equation}
provides a (filtered) inclusion of cooperads $\Cc \rightarrowtail \Tc^c\left( {\scriptstyle s}\curvAd, {\scriptstyle s}\curvA, {\scriptstyle s}\as\right)$, where the decomposition map on the right is the cofree decomposition map. (This formula differs slightly, by a small modification of the sign, from the formula given in \cite[Theorem 6.9]{BMDC20} because of our change of sign conventions explained in Section \ref{section: sign conventions}.)
We have
\begin{multline*}
\Delta\left( \left(\mu_{n+k}^{c, S}\right)^T\right) =\\
\sum_{\substack{i_1 + \dots + i_{m'} = n-|T_0|\\ \{T_0, \dots, T_{m'}\}\\ \{S_0, \ldots, S_{m'}\}\\ |S_j| = k_j}} (-1)^{\epsilon_1+\epsilon_2} \left( \left(\mu_{m+k_0}^{c, S_0}\right)^{T_0}; \left(\mu_{i_1+k_1}^{c, S_1}\right)^{T_1}, \ldots, \left(\mu_{i_{m'}+k_{m'}}^{c, S_{m'}}\right)^{T_{m'}}\right)
\end{multline*}
where $m' = m-|T_0|$ and the sums run over the partitions $\{T_0, \dots, T_{m'}\}$ and $\{S_0, \ldots, S_{m'}\}$ such that
\begin{itemize}
\item
$T_0 = T_0^0 \sqcup \dots \sqcup T_0^{m'}$ with for all $j$ between 1 and $m'$, $\max T_0^{j-1} < \min T_0^j$ and
\begin{align*}
T & = T_0^0 \sqcup (T_1+|T_0^0|) \sqcup (T_0^1+i_1 -1) \sqcup \dots \sqcup\\
& \qquad \left(T_{m'} + |T_0^1|+ \sum_{j= 1}^{m'-1} (|T_0^j| + i_j)\right) \sqcup \left(T_0^{m'} + \sum_{j=1}^{m'} (i_j-1)\right)\\
& = \tilde{T}_0 \sqcup \dots \sqcup \tilde{T}_{m'}
\end{align*}
(where $\tilde{T}_j$ corresponds to the elements of $T_j$ reindexed)
\item
$S_0 = S_0^0 \sqcup \dots \sqcup S_0^{m'}$ and for all $j$ between 1 and $m'$, $\max S_0^{j-1} < \min S_0^j$ and
\begin{align*}
S & = S_0^0 \sqcup (S_1+|S_0^0|+|T_0^0|) \sqcup (S_0^1+i_1+k_1-1) \sqcup \dots \sqcup\\
& \qquad \left(S_{m'} +|S_0^0|+ |T_0^0| + \sum_{j =1}^{m'-1}(|S_0^j|+|T_0^j| + i_j + k_j)\right) \sqcup \left(S_0^{m'} + \sum_{j=1}^{m'} (i_j + k_j -1)\right)\\
& = \tilde{S}_0 \sqcup \dots \sqcup \tilde{S}_{m'}
\end{align*}
(where $\tilde{S}_j$ corresponds to the elements of $S_j$ reindexed)
\end{itemize}
and
\begin{itemize}
\item
$\epsilon_1 = \sum (i_j+k_j-1)(\sum_{l<j}(i_l+k_l) + \sum_{l < j}(|S_0^l|+|T_0^l|)) + k_0(n+k-m-k_0) + L_0 + K$, with 
$L_0 \coloneqq \sum_{s \in \tilde{S}_0} \sum_{\{t;\, \max \{\tilde{S}_t\} < s\}} k_t = \sum_j k_j \sum_{l\leq j}|S_0^l|$, and $K \coloneqq \sum k_t \sum_{j > t} (i_j+k_j -1)$ is computed as in Formula (9) in \cite{BMDC20} and modified by the change of convention explained in Lemma \ref{lem: change of convention},
\item
whereas
\begin{align*}
\epsilon_2 & = |T_0|\left(\sum_{j=1}^{m'}(i_j+k_j-1-k_j)\right)+\sum_{j=1}^{m'} |T_0^{j}|\sum_{l \leq j}|T_{l}| + \sum_j |T_j| \sum_{l > j}(i_l+2k_l-1)\\
& \equiv |T_0|\left(n-m\right)+\sum_{j=1}^{m'} |T_0^{j}|\sum_{l \leq j}|T_{l}| + \sum_j |T_j| \sum_{l > j}(i_l-1) \mod 2
\end{align*}
is computed as in Corollary 6.1.5 in \cite{HM12} using Proposition 6.1.4 and where we haven't considered the sign coming from the decomposition of $\tilde\mu_n^c$ already considered in the previous point. 
We remark that this sign is not modified by the presence of the element ${\scriptstyle s}\curvA$.
\end{itemize}
We remark that we have also add the terms $k_0(n+k-m-k_0)$ and $\sum_j |T_j| \sum_{l > j}(i_l+2k_l-1)$ which are missing in \cite{BMDC20, HM12}. 
Then, by similar arguments as in \cite{HM12, BMDC20} (and because we have $k- k_0 -\dots -k_{m'} =0$), the \emph{images} of the $\nu_n^T$ in $\Tc^c\left( {\scriptstyle s}\curvAd, {\scriptstyle s}\curvA, {\scriptstyle s}\as\right)$ satisfies an equation like Equation \eqref{eq: decomposition map cAs}, so Equation \eqref{eq: decomposition map cAs} is satisfied since $\Cc$ injects as a cooperad in $\Tc^c\left( sE\right)$ for $E = \langle \curvAd, \curvA, \as \rangle$. 
Moreover, looking at the image of the first $\nu_n^T$'s through $\upsilon$, a careful calculation of the signs shows that the composite
\begin{multline*}
\Cc \rightarrowtail \Tc^c(sE) \twoheadrightarrow\\
S \coloneqq \left. \left(I\oplus \Tc^c(sE)^{(2)} \right) \middle/ \left( {\scriptstyle s^2} \left(\ass, \lunit, \runit \right) \oplus (\id + {\scriptstyle s^2} \tilde \theta)(I)\right) \right.
\end{multline*}
is zero. 
As a byproduct, we get that $\upsilon(\nu_1^{\emptyset}) \in \Tc^c(sE)$ is sent to 0 in the coideal quotient $(S)$ (whose definition is given in \cite[Definition 5.3]{BMDC20}) so the counit of $\Tc^c(sE)$ cannot factor through the coideal quotient $(S)$. 
By the universal property of the cooperad $(\q\cuAs)^\antishriek$, it follows that there exists a unique (injective) morphism $\Cc \rightarrowtail (\q\cuAs)^\antishriek$ such that the following diagram commutes:
\[
\begin{tikzcd}
\Cc
\dar\ar[r, >->] & \Tc^c(sE).
\\
(\q\cuAs)^\antishriek \ar[ur, >->] &
\end{tikzcd} 
\]
The injective map $\Gr\Cc \rightarrowtail \mathrm{coFree}(sE)$ can be factored as $\Gr \Cc \rightarrowtail \Gr ((\q\cuAs)^\antishriek) \cong (\Gr \q\cuAs)^\antishriek \rightarrowtail \mathrm{coFree}(sE)$, where the second isomorphism is the Poincaré--Birkhoff--Witt type isomorphism which appears in the proof of Theorem \ref{thm: Koszul resolution}. 
Its image is precisely $(\Gr \q\cuAs)^\antishriek \cong s\curvA \oplus \q\uAs^\antishriek$ so $\Gr \Cc \cong \Gr ((\q\cuAs)^\antishriek) \cong (\Gr \q\cuAs)^\antishriek$. 
Since the previous inclusions are strict and we are working with complete objects, this proves that $\Cc \cong (\q\cuAs)^\antishriek$. 

The curvature is computed directly from its definition given in Section \ref{sec: CKD}.

In order to have an $\Sb$-cofibrant resolution, it is enough to remark that the curvature is non-zero, that for the weight filtration given by its generators, $\Gr\cuAs$ is connected bounded below weight graded and that the filtration $F_\bullet$ on $\cuAs$ is induced by the graduation given by the number of $\curvA$. 
We can therefore apply Theorem \ref{thm: Koszul resolution}.
\end{proof}

We make the algebras over the curved operad $\cuAi$ explicit and we compare them with the literature.

\begin{prop}
A $\cuAi$-algebra $\gamma_A : \cuAi \to \End_{(A, d_A)}$ is equivalent to a (complete) graded vector space $(A,\, F_\bullet)$ equipped with an operation $m_0 : \Kf \to F_1A$ of degree $-2$ and for all $n \geq 1$ and $T \subseteq [n]$, with filtered operations
\[
m_n^T : A^{\otimes (n-|T|)} \to A \text{ of degree } n+|T|-2,
\]
where $m_1 = d_A$, which together satisfy the following identities:
\begin{equation}
\label{eq: relation cuAi-algebras}
\left\{ \begin{array}{lcl}
\partial(m_{2}^{\{1 \}}) &=& m_{2} \circ(m_{1}^{\{1\}}, -) - \id_A, \\
\partial(m_{2}^{\{2 \}}) &= & m_{2} \circ (- ,m_{1}^{\{1\}}) - \id_A,
\end{array} \right.
\end{equation}
and for $(n,\, T) \neq (2,\, \{ 1\})$ and  $(n,\, T) \neq (2,\, \{ 2\})$
\begin{equation}
\label{eq: relation qcuAi-algebras}
\sum_{\substack{p+q+r=n\\ T = T_0' \sqcup (T_1 + p) \sqcup (T_0'' + q)\\ T_0 = T_0' \sqcup T''_0}} (-1)^{(p+|T_0''|)q +r + |T_0'|(|T_1|+1)} m_{p+1+r}^{T_{0}} \circ_{(p+1-|T_0'|)} m_{q}^{T_{1}} = 0,
\end{equation}
where $\max T_0' < p \leq p+q < \min T_0''$ and where we have identified the maps $m_n$ and $m_n^\emptyset$. 

Up to the sign convention, this notion of algebras coincides with the notions of homotopy-unital filtered $\Ai$-algebras used in \cite[Section 3.3.1]{FOOO07}.
\end{prop}

\begin{proof}
An operad morphism $\gamma_A : \hat\Omega \cuAs^\antishriek \to \End_{(A,\, d_A)}$ is characterized by $(A, d_A)$ and a degree $0$ map of $\Sb$-modules $s^{-1}\overline{\cuAs^\antishriek} \to \End_A$. Therefore it corresponds to a (complete) graded vector space $(A, F_\bullet)$ and a collection of filtered applications
\[ m_n^T : A^{\otimes (n-|T|)} \to A, \text{ of degree } n+|T|-2 \text{ for all } n \geq 0 \text{ and } T \subseteq [n],\]
such that $m_0 : \Kf \to F_1A$ and $m_1 = d_A$, where each $m_n^T$ corresponds to the image of $\nu_n^T$ through $\gamma_A$. 
The partial decomposition map on $\cuAs^\antishriek$ is given by
\[
\Delta_{(1)}^\theta(\nu_n^T) = \sum_{\substack{p+q+r=n\\ T = T_0' \sqcup (T_1 + p) \sqcup (T_0'' + q)\\ T_0 = T_0' \sqcup T''_0}} (-1)^{(q-1)(p+ |T_0''|)+ |T_0'||T_1|}\nu_{p+1+r}^{T_0} \circ_{p+1-|T_0'|} \nu_q^{T_1}.
\]
The fact that $\gamma_A$ commutes with the predifferentials and the fact that $\gamma_A$ sends the curvature to the curvature ensure that Equations \eqref{eq: relation cuAi-algebras} and \eqref{eq: relation qcuAi-algebras} are satisfied in accordance with the conditions $n \geq 0$ and $T \subseteq [n]$.

The correspondence with the notion of homotopy unit given in \cite[Definition 3.3.2]{FOOO07} is given as follows: the homotopy unit $\mathbf{e}$ in this reference coincides with the image $m_1^{\{1\}}(1_{\Kf}) = \gamma_A\left(\nu_1^{\{1\}}\right)(1_\Kf)$. 
The maps $\mathfrak{h}_k$ coincide with
\[
\sum_{\substack{n \geq 1\\ T \subseteq [n+k-1]\\ |T|=k-1}} m_{n+k-1}^T = \gamma_A\left( \sum_{\substack{n \geq 1\\ T \subseteq [n+k-1]\\ |T|=k-1}} \nu_{n+k-1}^T\right).
\]
\end{proof}

\subsection{Complex curved Lie algebras}

In \cite{jM14}, the author has defined the operads $\Cx$ and $\Cxi$ and he has used them to propose a new description of the notion of complexe structure. 
In order to define the notion of $\Cxi$-space, we add a $0$-ary generator to these operads so that our operads encode (homotopy) complex curved Lie algebras instead of (homotopy) complex Lie algebras.\\

We recall some notations and results from \cite{jM14}. 
We denote by
\[
\sAs := \left. \Tc({\scriptscriptstyle s^{-1}}\as) \middle/ \left({\scriptscriptstyle s^{-2}}\left(\assu\right) \right) \right.
\]
the operad encoding associative algebras whose product has cohomological degree $-1$ and by $\sLie \hookrightarrow \sAs$ the sub-operad encoding Lie algebras whose bracket has cohomological degree $-1$. 
The operad $\Cx$ is defined by
\[
\Cx \coloneqq \left(\Tc\left(\jn , \slie \right) \middle/ \left( R \right), 0 \right),
\]
where $\slie$ is a symmetric element and $R$ is the $\Sb$-module
\[
R \coloneqq \left\langle \jnr ,\, \sjlr - \sljr ,\, {\scriptscriptstyle s^{-2}}\left( \jac \right) \right\rangle.
\]
Its differential is zero and there is an isomorphism of $\Sb$-modules $\Cx \cong \Cb \hcirc \sLie$.
Its Koszul dual curved cooperad is given by
\[
\Cxa = \left(\Lie_{1}^{\antishriek} \hcirc \Rb \left[\qJc \right], \Delta_{\Lambda^{c}}, \theta^c\right).
\]
We denote by $\overline{\jmath}^c_{k_1, \ldots , k_n} := (\overline{l}_n^c ; \overline{\imath}_{k_1}^c, \ldots , \overline{\imath}^c_{k_n})$, where $\overline{l}_n^c$ lies in $\Lie_{1}^{\antishriek}$ and the elements $\overline{\imath}_{k}^c$ lie in $\Rb \left[\qJc \right]$, the generators of $\Cxa$. The decomposition map is given on generators by
\begin{multline}
\label{eq: coproduct Cx^!}
\Delta_{\Lambda^{c}} \left( \overline{\jmath}^c_{k_1, \ldots , k_n}\right) =\\
\sum \frac{1}{p!} \beta^{\sigma}_{k'_{j}, k''_{j}} \times \left( \overline{\jmath}_{l'_1, \ldots , l'_p}^{c} ; \overline{\jmath}_{k''_{\sigma(1)}, \ldots , k''_{\sigma(q_1)}}^{c}, \overline{\jmath}_{k''_{\sigma(q_1+1)}, \ldots , k''_{\sigma(q_1+q_2)}}^{c} , \ldots \right)^{\sigma^{-1}},
\end{multline}
where the sum $\sum$ is given by $\sum_{q_{1}+\cdots +q_{p}= n} \sum_{\sigma \in Sh_{q_{1}, \ldots , q_{p}}} \sum_{k'_{j}+k''_{j} = k_{j}}$ (where $Sh_{q_{1}, \ldots , q_{p}}$ is the set of $(q_1, \ldots, q_p)$-shuffles), the integers $l'_i$ are defined by $l'_i := k'_{\sigma(q_1+\cdots +q_{i-1}+1)}+ \cdots +k'_{\sigma(q_1+\cdots +q_i)}$ and the number $\beta^{\sigma}_{k'_{j}, k''_{j}}$ is a sign defined in \cite[Proposition 3.6]{jM14}.
Finally, the curvature on $\Cx^\antishriek$ is defined by
\[
\theta^c \left( \overline{\jmath}^c_{k_1, \ldots , k_n} \right) = \left\{ \begin{array}{cl}
-\unit & \textrm{if } n = 1 \textrm{ and } k_1 = 2,\\
0 & \textrm{otherwise.}
\end{array} \right.
\]
(The sign is different from the sign in \cite{HM12} since it is different in the definition of the Koszul dual curved cooperad given in Section \ref{sec: CKD}.)

We now define the \emph{curved versions} of these operads. 
The curved operad $\cAs$ and $\cLie$ already appear in \cite{BMDC20} and we denote by $\cAs_1$ and $\cLie_1$ there shifted versions. 
The. $0$-ary element ${\scriptstyle s^{-1}}\curvA$ is put in cohomological degree 1 so that the curvatures of the operads have cohomological degree 2. 
We make the case of the operad encoding complex curved Lie algebras explicit.

\begin{defn}
The curved operad encoding complex curved Lie algebras is defined by
\begin{align*}
\cCx & \coloneqq \left. c\Tc\left({\scriptstyle s^{-1}}\curvA,\, \jn ,\, \slie \right) \middle/ \left( R \oplus \langle \vartheta - \theta\rangle \right) \right. ,
\end{align*}
where the notation $c\Tc$ stands for the free curved operad defined in \cite[Section 2]{BMDC20}, the generators ${\scriptstyle s^{-1}}\curvA$ and $\slie$ have cohomological degree $1$ and the generator $\jn$ has cohomological degree $0$. Moreover $\theta \coloneqq {\scriptstyle s^{-1}}\curvsLie$ and the $\Sb$-module $R$ is defined above. 
We filter $\cCx$ by the number of $\curvA$, that is
\[
F_p \cCx \coloneqq \{ \mu \in \cCx \textrm{ s.t. the number of } \curvA \textrm{ in } \mu \textrm{ is greater than or equal to } p\}.
\]
(This filtration is in fact a graduation and $\cCx$ is the completion of its associated graded.) 
\end{defn}

An alternative description of the curved operad $\cCx$ is given by the following lemma.

\begin{lem}
\label{lem: alternative presentation of cCx}
We have the isomorphism of curved operads
\[
\cCx \cong \left. \left(\Tc\left({\scriptstyle s^{-1}}\curvA,\, \jn ,\, \slie \right) \middle/ \left( R \right), 0, \theta = {\scriptstyle s^{-1}}\curvsLie \right) \right. .
\]
\end{lem}

\begin{proof}
The predifferential of $\cCx$ is zero so the only thing to prove is that the map $[\theta, -]$ is zero. 
As previously, the relation $\sjlr - \sljr$ ensures that $\cCx \cong \Cb \hcirc \cLie_1$ as $\Sb$-modules. 
Thus it suffices to show that $[\theta, -]$ is zero on $\cLie_1$ and on $\jn \otimes \cLie$. 
The first case follows from the fact that $\cLie_1$ is a curved operad, and the second also by means of the relation $\sjlr - \sljr$.
\end{proof}

We put a \emph{weight} filtration $\cCx^{(\bullet)}$ on $\cCx$ by counting the number of the generators.

\begin{lem}
\label{lem: cCx-algebra}
A $\cCx$-algebra on a complete gr-dg $\Rb$-vector space $A$ is the same data as a complete curved Lie $\Cb$-algebra $(A, [-, -], d_A, \theta_A)$ with curvature $\theta_A \in F_1 A$.
\end{lem}

\begin{proof}
A map of complete curved operad $\cCx \to \End_A$ is characterized by the images of the generators ${\scriptscriptstyle s^{-1}}\curvA$, $\jn$ and $\slie$ which give respectively a degree $1$ map $\theta_A : \Kb \to A$, a degree $0$ map $J : A \to A$ and a degree $1$ anti-symmetric map $[-, -] : A^{\hotimes 2} \to A$. The relations defining $\cCx$ ensure that $J$ provides a $\Cb$-action, that $[-, -]$ is a Lie bracket of degree $1$ which is $\Cb$-linear. 
The fact that the curvature of $\cCx$ is sent to the curvature of $\End_A$ says that ${d_A}^2 = [\theta_A, -]$.
\end{proof}

\subsection{The Koszul dual curved cooperad associated with $\cCx$}

We show now that the curved operad $\cCx$ is Koszul and we describe its Koszul dual curved cooperad $\cCxa$. 

\begin{thm}
\label{thm: cCxa}
The presentation of the curved operad $\cCx$ given above is inhomogeneous quadratic (we could even say more precisely \emph{constant-quadratic}). 
It is moreover a Koszul curved operad. 
The Koszul dual curved cooperad $\cCxa$ is isomorphic to
\begin{align*}
\cCxa \cong & \left( \curvA \oplus \Cxa, \Delta^\theta,\ 0, \theta^c \right) \cong \left( \curvA \oplus \left(\sLie^\antishriek \hcirc \Rb \left[\qJc \right]\right), \Delta^\theta, 0, \theta^c \right)\\
\cong & \left( \left\{ \jmath_{k_1, \ldots, k_n}^c \right\}_{\substack{n \geq 0\\ k_j \geq 0}}, \Delta^\theta, 0, \theta^c \right),
\end{align*}
where the terms $\jmath_{k_1, \ldots, k_n}^c$ have cohomological degree $-(k_1 + \cdots + k_n)$ and with the notation $\jmath_{\emptyset}^c = \curvA$.
The curvature $\theta^c$ is defined by
\[
\theta^c \left( \jmath_{k_1, \ldots , k_n} \right) = \left\{ \begin{array}{cl}
\unit & \textrm{if } n = 1 \textrm{ and } k_1 = 2 ,\\
0 & \textrm{otherwise.}
\end{array} \right.
\]
The decomposition map $\Delta^\theta$ is given by the formula
\begin{multline*}
\Delta^\theta \left( \jmath_{k_1, \ldots, k_n}^c\right) =\\
\sum \frac{1}{p!} \beta^{\sigma}_{k'_{j}, k''_{j}} \times \left( \jmath_{l'_1, \ldots , l'_p}^{c} ; \jmath_{k''_{\sigma(1)}, \ldots , k''_{\sigma(q_1)}}^{c}, \jmath_{k''_{\sigma(q_1+1)}, \ldots , k''_{\sigma(q_1+q_2)}}^{c} , \ldots \right)^{\sigma^{-1}}
\end{multline*}
where the sum $\sum$ is given by $\sum_{q_{1}+\cdots +q_{p}= n} \sum_{\sigma \in Sh_{q_{1}, \ldots , q_{p}}} \sum_{k'_{j}+k''_{j} = k_{j}}$, the integers $l'_i$ are defined by $l'_i \coloneqq k'_{\sigma(q_1+\cdots +q_{i-1}+1)}+ \cdots +k'_{\sigma(q_1+\cdots +q_i)}$ and the number $\beta^{\sigma}_{k'_{j}, k''_{j}}$ is a number defined in \cite[Proposition 3.6]{jM14} and is recalled in the remark below.
Moreover, the natural map
\[ f_\kappa : \cCxi \coloneqq \hat\Omega \cCx^\antishriek \to \cCx \]
is an $\Sb$-cofibrant resolution.
\end{thm}

\begin{rem}
\begin{itemize}
\item
The difference between the formula given in Proposition 3.6 in \cite{jM14} and the formula here is that in the Theorem above the integers $p$ and $q_i$'s can be $0$. In this case, we denote by $\jmath^c_{\emptyset}$ the corresponding generator.
\item
Be aware of the fact that the curvature $\theta : I \to \cCx$ and $\theta^c : \cCxa \to I$ are unrelated. The curvature $\theta^c$ encodes that $\Cb$ isn't augmented. 
\item
The numbers $\beta^{\sigma}_{k'_{j}, k''_{j}}$ are equal to
\[
\beta^{\sigma}_{k'_{j}, k''_{j}} := \sgn_{k_{1}, \ldots , k_{n}}\sigma \times \varepsilon_{k'_{j}, k''_{j}}^{\sigma} \times \prod_{i=1}^p \alpha_{k'_{\sigma(q_1+\cdots +q_{i-1}+1)},\, \ldots , k'_{\sigma(q_1+\cdots +q_i)}}
\]
where $\sgn_{k_{1}, \ldots , k_{n}}\sigma$ is the signature of the restriction of $\sigma$ to the indices $j$ such that $k_{j}$ is odd (after relabeling the remaining $\sigma(j)$ in a way that the order of the $\sigma(j)$ does not change) and where
\[
\left\{ \begin{array}{lcl}
\varepsilon_{k'_{j},\, k''_{j}}^{\sigma} & \coloneqq & (-1)^{\sum_{i=1}^{n} k_{\sigma(i)}''(k_{\sigma(i+1)}' + \cdots + k_{\sigma(n)}')}\\
\alpha_{k'_{1},\, \ldots ,\, k'_{q}} & \coloneqq & \sum_{\sigma' \in Sh_{k'_{1},\, \cdots ,\,k'_{q}}} \sgn \sigma',\, \textrm{with convention } \alpha_{0,\, \ldots ,\, 0} := 1.
\end{array}\right.
\]
\end{itemize}
\end{rem}

\begin{proof}
The proof follows the same flow as the proof of Theorem \ref{thm: Koszul dual cuAs}. 
The presentation of $\cCx$ given in Lemma \ref{lem: alternative presentation of cCx} is easily seen to be inhomogeneous quadratic, except for Condition (3) in Definition \ref{def: inhomogeneous quadratic} that will be proven during the proof. 

To prove the Koszulness of $\cCx$ (see Definition \ref{def: Koszul}), we first remark that $\cCx$ is the completion of its associated graded (for the filtration $F_\bullet$) and that the presentation given above behaves correctly with the filtration. 
Then the associated quadratic curved operad $\q\cCx$ is also the completion of its associated graded and as quadratic operads $\Gr \q\cCx \cong \q\Gr\cCx$. 
Finally we prove that $\q\Gr\cCx$ is a quadratic Koszul operad. 
We have
\[
(\q\Gr\cCx)^\antishriek \circ_\kappa \q\Gr\cCx \cong \curvA \oplus {\scriptstyle s^{-1}}\idCurvA \oplus | \oplus \left\{\bigoplus_{T \subseteq [n]} ((\Gr\q\Cx)^{\antishriek} \circ_{\kappa} \Gr\q\Cx) (n) \right\}_{n \geq 2}.
\]
It is proved in \cite{jM14} that the operad $\Gr\q\Cx$ is Koszul (it is simply called $\q\Cx$ there) therefore we get that the quasi-isomorphism $(\q\Gr\cCx)^\antishriek \circ_\kappa \q\Gr\cCx \xrightarrow{\sim} I$ and the fact that $\q\Gr\cCx$ is a quadratic Koszul operad. 

We now compute the Koszul dual curved cooperad $\cCx^\antishriek$. 
We first prove that
\[
\Cc = \left( \left\{ \jmath_{k_1, \ldots, k_n}^c \right\}_{\substack{n \geq 0\\ k_j \geq 0}}, \Delta^\theta, 0, \theta^c \right)
\]
is a sub-cooperad of $\Tc^c\left( sE\right)$, for $sE = \langle \curvA, {\scriptstyle s}\jn, \lie \rangle$, satisfying the property for which $(\q\cCx)^\antishriek$ is universal (so that $\Cc \hookrightarrow (\q\cCx)^\antishriek$). 
We use the (cohomological degree 0) inclusion
\begin{equation}
\label{eq: injection of the Koszul dual coop - cCx}
\upsilon : \jmath_{k_1, \ldots, k_n}^c \mapsto \sum_{k \geq 0} \frac{1}{k!} \jmath_{\underbrace{{\scriptstyle 0}, \dots, {\scriptstyle 0}}_{k \textrm{ times}}, k_1, \ldots, k_n}^{c, [k]} \ ,
\end{equation}
where $\jmath_{0, \dots, 0, k_1, \ldots, k_n}^{c, [k]}$ is the element $\jmath_{0, \dots, 0, k_1, \ldots, k_n}^{c}$ on which we have grafted the element $\curvA$ in the $k$ first positions. 
The equality
\[
(\upsilon \circ \upsilon) (\Delta^\theta(\jmath_{k_1, \ldots, k_n}^c)) = \Delta (\upsilon (\jmath_{k_1, \ldots, k_n}^c))
\]
results from two reasons. First, the coefficients $\sgn_{k_{1}, \ldots , k_{n}}\sigma$, $\varepsilon_{k'_{j}, k''_{j}}^{\sigma}$ and $\alpha_{k'_1,\, \ldots , k'_q}$ are the same for $\jmath_{k_1, \ldots, k_n}^c$ and $\jmath_{0, \dots, 0, k_1, \ldots, k_n}^{c, [k]}$. 
Secondly the computation in the cofree cooperad of $\Delta \Big(\jmath_{0, \dots, 0, k_1, \ldots, k_n}^{c, [k]}\Big)$ provides an element
\[
\Big( \jmath_{0, \dots, 0, l'_1, \ldots , l'_p}^{c, [l_0]}\ ;\ \jmath_{0, \dots, 0, k''_{\sigma(1)}, \ldots , k''_{\sigma(q_1)}}^{c, [l_1]},\ \jmath_{0, \dots, 0, k''_{\sigma(q_1+1)}, \ldots , k''_{\sigma(q_1+q_2)}}^{c, [l_2]} ,\ \ldots \Big)
\]
with an extra factor $\frac{k!}{l_0! l_1! l_2! \dots} = \mathrm{Card} (Sh_{l_{0}, l_1, l_2, \ldots})$. 
It remains to show that the composite
\[
\Cc \rightarrowtail \Tc^c(sE) \twoheadrightarrow S \coloneqq \left. \left(I\oplus \Tc^c(sE)^{(2)} \right) \middle/ \left( {\scriptstyle s^2} \left(\q R \right) \oplus (\id + {\scriptstyle s^2} \tilde \theta)(I)\right) \right.
\]
where
\[
\q R \coloneqq \left\langle \qjnr ,\, \sjlr - \sljr ,\, {\scriptscriptstyle s^{-2}}\left( \jac \right) \right\rangle,
\]
is zero. 
A direct computation of the image of the different elements $\jmath_\emptyset^c$, $\jmath_0^c$, $\jmath_1^c$, $\jmath_2^c$, $\jmath_{0, 0}^c$, $\jmath_{0, 1}^c$ and $\jmath_{0, 0, 0}^c$ through $\upsilon$ shows the result. 
By the universal property of the cooperad $(\q\cCx)^\antishriek$, we get the (unique) factorisation $\Cc \rightarrowtail (\q\cCx)^\antishriek \rightarrowtail \Tc^c(sE)$. 

As a byproduct of the above computations, we get that $\upsilon(\jmath_0^c) \in \Tc^c(sE)$ is sent to 0 in the coideal quotient $(S)$ (whose definition is given in \cite[Definition 5.3]{BMDC20}) so the counit of $\Tc^c(sE)$ cannot factor through the coideal quotient $(S)$. 
This finishes to prove that $\cCx$ is inhomogeneous quadratic. 

The injective map $\Gr\Cc \rightarrowtail \mathrm{coFree}(sE)$ can be factored as $\Gr \Cc \rightarrowtail \Gr ((\q\cCx)^\antishriek) \cong (\Gr \q\cCx)^\antishriek \rightarrowtail \mathrm{coFree}(sE)$, where the second isomorphism is the Poincaré--Birkhoff--Witt type isomorphism which appears in the proof of Theorem \ref{thm: Koszul resolution}. 
Its image is precisely $(\Gr \q\cCx)^\antishriek \cong \curvA \oplus \q\Cx^\antishriek$ so $\Gr \Cc \cong \Gr ((\q\cCx)^\antishriek) \cong (\Gr \q\cCx)^\antishriek$. 
Since the previous inclusions are strict and we are working with complete objects, this proves that $\Cc \cong (\q\cCx)^\antishriek$. 

In order to have an $\Sb$-cofibrant resolution, it is enough to remark that the curvature is non-zero, that for the weight filtration given by its generators, $\Gr\cCx$ is connected bounded below weight graded and that the filtration $F_\bullet$ on $\cCx$ is induced by the graduation given by the number of ${\scriptstyle s^{-1}}\curvA$. 
We can therefore apply Theorem \ref{thm: Koszul resolution}.
\end{proof}

\begin{cor}
\label{cor: cCxa infinitesimal}
The infinitesimal decomposition map on $\cCxa$ is given by
\begin{multline*}
\Delta_{(1)}(\jmath_{k_1, \ldots, k_n}^c) = \\
\sum_{\substack{p+q=n+1\\ p, q \geq 0}} \sum_{\sigma \in Sh_{q, p-1}} \sum_{K} \alpha^{\sigma}_{k'_{j}, k''_{j}} \times \left( \jmath_{k', k_{\sigma(q+1)}, \ldots , k_{\sigma(n)}}^{c} \circ_{1} \jmath_{k''_{\sigma(1)}, \ldots , k''_{\sigma(q)}}^{c} \right)^{\sigma^{-1}},
\end{multline*}
where
\[
K \coloneqq \{k'_{\sigma(j)}+k''_{\sigma(j)} = k_{\sigma(j)},\ k'_{\sigma(j)} = k_{\sigma(j)} \textrm{ for } j > q \textrm{ and } k' = k'_{\sigma(1)}+ \cdots +k'_{\sigma(q)} \},
\]
and where $\alpha^{\sigma}_{k'_{j}, k''_{j}} := \sgn_{k_{1}, \ldots , k_{n}}\sigma \times \varepsilon_{k'_{j}, k''_{j}}^{\sigma} \times \alpha_{k'_{\sigma(1)}, \ldots , k'_{\sigma(q)}}$.
\end{cor}

\begin{proof}
The formula is the same as the formula given in Proposition 3.6 in \cite{jM14} except that $p, q \geq 0$ instead of $p, q \geq 1$ (the case $p=0$ appears when $n=0$). 
\end{proof}

\section{Algebraic curved twisting morphisms}

We now consider algebras and coalgebras in $\Rf$-modules where $\Rf$ is a unital $\Kf$-cdga.

\subsection{Formal constructions in the curved context}

\subsubsection{Free curved algebra}

We first recall from \cite[Proposition C.35]{BMDC20} the free-fogetful adjunction for curved algebras over a curved operad $(\Pc, d, \theta, \eta)$.

\begin{prop}
\label{prop: free-forgetful adjunction for curved algebras}
The forgetful functor $U : \Palg \to \ModAgr$ admits a left adjoint, called the \emph{free $\Pc$-algebra functor}, $F_{\Pc} : \ModAgr \to \Palg$ given by
\[ (V, d_V) \mapsto F_{\Pc}(V, d_V) \coloneqq \left(\Pc(V)/\left(\im\left({d_{\Pc(V)}}^2 - \gamma(\theta \otimes \id_\Pc) \otimes \id_V\right) \right), d_{\overline{\Pc(V)}}\right). \]
\end{prop}

\begin{rem}
We remark that the ideal $\left(\im\left({d_{\Pc(V)}}^2 - \gamma(\theta \otimes \id_\Pc) \otimes \id_V\right) \right)$ coincides with the ideal $\left(\im\left(\eta \otimes ({d_{V}}^2) - \theta \otimes \id_V\right) \right)$ appearing in \cite{BMDC20}.
\end{rem}

\begin{defn}
\label{def: quasi-free alg}
We call a $\Pc$-algebra $A$ \emph{quasi-free} if there exists a gr-dg module $V$ and a predifferential (see Definition \ref{def: der prediff and coder}) $d : \Pc(V) \to \Pc(V)$ such that
\[ A = \Pc(V)/\left(\im\left( d^2 - \gamma_{\Pc(V)}(\theta \otimes \id_{\Pc(V)})\right)\right). \]
\end{defn}

This notion is different from that of a free $\Pc$-algebra because in the quasi-free case the predifferential $d$ is not a priori induced by a map $d_V : V \to V$ (that is it might be different from any map $d_\Pc \circ \id_V + \id_\Pc \circ' d_V$).\\

Similarly, there exists a cofree-forgetful adjunction for curved coalgebras over a curved altipotent cooperad $(\Cc, d, \theta^c, \varepsilon)$ that we present in Section \ref{sec: cofree coalgebra}.

\subsubsection{(Co)Derivations and predifferentials}

In the context of a gr-dg operad $(\Qc, d_\Qc)$ and a gr-dg cooperad $(\Dc, d_\Dc)$, the notions of $\Qc$-derivation and of $\Dc$-coderivation are similar to the ones present in the classical literature. 
We recall some definitions that aren't always clear in the literature for us and propositions for the sake of completeness.

\begin{defn}\
\label{def: der prediff and coder}
\begin{enumerate}
\item
A \emph{derivation} on a $\Qc$-algebra $B$ is a linear map $d : B \to B$ satisfying the equation $d \cdot \gamma_B = \gamma_{B}\cdot (\id_\Qc \circ' d)$.
\item
A \emph{predifferential} on a $\Qc$-algebra $B$ is a linear map $d : B \to B$ of degree $-1$ satisfying the equation $d \cdot \gamma_B = \gamma_{B}\cdot (d_\Qc \circ \id_A + \id_\Qc \circ' d)$.
\item
A \emph{coderivation} on a $\Dc$-coalgebra $D$ is a linear map $d^c : D \to D$ satisfying the equation $\Delta_{D} \cdot d^c = (\id_\Dc \circ' d^c) \cdot \Delta_{D}$.
\item
A \emph{predifferential} on a $\Dc$-coalgebra $D$ is a linear map $d^c : D \to D$ of degree $-1$ satisfying the equation $\Delta_{D} \cdot d^c = (d_\Dc \circ \id_D + \id_\Dc \circ' d^c) \cdot \Delta_{D}$.
\end{enumerate}
\end{defn}

It follows from the above definitions that when $d_A$ is a predifferential of a $\Qc$-algebra $A$, then a degree $-1$ gr-dg linear map $d$ is a derivation if and only if $d_A + d$ is a predifferential of $B$. 
A similar result holds for $\Dc$-coalgebras (without the gr-dg assumption).

More specifically, when a degree $-1$ linear map $d$ is derivation of the free $\Qc$-algebra $\Qc(V)$, then $d_\Qc \circ \id_V + d$ is a predifferential of $\Qc(V)$. 
Again a similar result holds for coderivations and predifferentials of cofree $\Dc$-coalgebras.

\begin{prop}[Propositions 6.3.6 and 6.3.8 in \cite{LV12}]
\label{prop: der and coder}
Let $\ModAg$ the category of graded $\Rf$-modules.
\begin{enumerate}
\item
Any derivation on a free $\Qc$-algebra $\Qc(V)$ is completely characterised by its restriction on the graded module of generators
\[
\mathrm{Der}(\Qc(V)) \cong \Hom_{\ModAg}(V, \Qc(V)).
\]
\item
Any coderivation on a cofree $\Dc$-coalgebra $\Dc(V)$ is completely characterised by its projection onto the graded module of generators
\[
\mathrm{coDer}(\Dc(V)) \cong \Hom_{\ModAg}(\Dc(V), V).
\] 
\end{enumerate}
\end{prop}

\subsection{Bar (and cobar) constructions for curved (co)algebras}

In this section, $(\Pc, d, \theta, \eta)$ refers to a curved operad and $(\Cc, d, \theta^c, \varepsilon)$ refers to a curved altipotent cooperad. 
To any curved twisting morphism $\alpha : \Cc \to \Pc$, we associate two functors
\[
\B_\alpha : \Palg \rightleftharpoons \Calg : \Om_\alpha
\]
forming an adjunction and representing a notion of algebraic curved twisting morphisms.

\subsubsection{Bar construction of a $\Pc$-algebra}
\label{sec: bar constr alg}

Let us denote by $\Cc^{pg}$ the pg cooperad underlying the curved altipotent cooperad $\Cc$. 
Let $(A, d_A, \gamma_A)$ be a $\Pc$-algebra.

We denote by $d^r_\alpha$ the coderivation (of the cofree $\Cc^{pg}$-coalgebra $\Cc(A)$) which extends the composite $\Cc(A) \xrightarrow{\alpha \circ \id_A} \Pc(A) \xrightarrow{\gamma_A} A$.
We also write $d_{\Cc(A)} = d_\Cc \circ \id_A + \id_\Cc \circ' d_A$ which is a predifferential of the $\Cc$-coalgebra $\Cc(A)$.

\begin{lem}
\label{lem: bar prediff}
Because $\alpha$ is a curved twisting morphism, the predifferential $d_{\Cc(A)}+d_\alpha^r$ satisfies
\[
(d_{\Cc(A)}+d_\alpha^r)^2 = -(\theta^c \otimes \id_{\Cc(A)}) \cdot \Delta_{\Cc(A)}.
\]
\end{lem}

\begin{proof}
A direct computation gives
\begin{align*}
(d_{\Cc(A)}+d^r_\alpha)^2 & = {d_\Cc}^2 \circ \id_A + \id_\Cc \circ' {d_A}^2 + d^r_{\partial \alpha + \frac{1}{2}[\alpha, \alpha]}\\
& = ((\id_\Cc \otimes \theta^c - \theta^c \otimes \id_\Cc) \cdot \Delta_{(1)}) \circ \id_A + \id_\Cc \circ' (\gamma_A(\theta \circ \id_A)) + d^r_{-\Theta}\\
& = -(\theta^c \otimes \id_{\Cc(A)}) \cdot \Delta_{\Cc(A)},
\end{align*}
where we have used the fact that $\alpha$ is a curved twisting morphism for the second equality.
\end{proof}

\begin{defn}
The \emph{bar construction of a $\Pc$-algebra $(A, d_A, \gamma_A)$ with respect to $\alpha$} is
\[
\B_\alpha A \coloneqq (\Cc(A), d_b\coloneqq d_{\Cc(A)}+d^r_\alpha).
\]
It satisfies ${d_b}^2 = -(\theta^c \otimes \id_{\Cc(A)}) \cdot \Delta_{\Cc(A)}$ by the previous lemma and this therefore defines a functor $\B_\alpha : \Palg \to \Calg$.
\end{defn}

\begin{rem}
For the interested reader, we remark that the bar construction is quasi-cofree in the sense described in next section (Section \ref{sec: cofree coalgebra}) since it is equal to the sub-$\Cc$-coalgebra of $\Cc(A)$ with relations given by
\[
S = \Cc(A)/\ker\left( (d_{\Cc(A)}+d_\alpha)^2 + (\theta^c \otimes \id_{\Cc(A)}) \cdot \Delta_{\Cc(A)}\right) = \{ 0\}.
\]
\end{rem}

In the following proposition, we propose several equivalent definitions of the notion of $\scLi$-algebras.

\begin{prop}
\label{prop: equivalent definition of homotopie cLie algebras}
A $\scLi$-algebra structure on a gr-dg module $(A, d_A)$ is the data of a morphism of curved operad $\gamma_A : \scLi \to \End_{(A, d_A)}$. It is equivalent to each of the following data:
\begin{enumerate}
\item
an operadic curved twisting morphism $\alpha_A$ in $\Tw(\scLa, \End_{(A, d_A)})$ such that $(\alpha_A)_{|I} = 0$,
\item
a predifferential $D_A : \scLa(A) \to \scLa(A)$ (of degree $-1$) such that ${D_A}_{|A} = d_A$ and which satisfies ${D_A}^2 = 0$,
\item
a family of skew-symmetric maps $l_n : A^{\otimes n} \to A$ of degree $|l_n| = -1$, for all $n \geq 0$ with $l_1 = d_A$ which satisfy the relations
\[
\sum_{\substack{p+q=n+1\\ p \geq 1, q \geq 0}} \sum_{\sigma \in Sh_{q,\, p-1}^{-1}} ({l}_{p} \circ_{1} {l}_{q})^{\sigma} = 0.
\]
\end{enumerate}
\end{prop}

\begin{proof}
The equivalence between the definition of $\scLi$-algebra structure on a gr-dg module $(A, d_A)$ and the data of a twisting morphism $\alpha_A$ which cancels on $I$ is Theorem \ref{thm: bar-cobar adjunction} restricted to morphisms instead of lax morphisms.
Indeed, assuming $(1)$, we get a predifferential $d_b = d_{\scLa(A)} + d_{\alpha_A}^r = \id_{\scLa} \circ' d_{A} + d_{\alpha_A}^r$, where $d_{\alpha_A}^r$ is defined in Section \ref{sec: bar constr alg} with the operadic curved twisting morphism $\alpha_A$ and the fact that $A$ can be seen as an $\End_{(A, d_A)}$-algebra (by the identity map $\id_{\End_{(A, d_A)}}$). 
Assuming $(2)$,
$D_A-\id_{\scLa} \circ' d_{A}$ is a coderivation characterised by a linear map $\scLa(A) \to A$ whose component $A \to A$ is $0$, therefore by an $\Sb$-module morphism $\alpha_A : \scLa \to \End_{(A, d_A)}$ which cancels on $I$. 
It follows that $D_A-\id_{\scLa} \circ' d_{A} = d_{\alpha_A}^r$ and the computation made in the proof of Lemma \ref{lem: bar prediff} shows that $d^r_{\Theta_H + \partial \alpha_A + \frac12 [\alpha_A, \alpha_A]} = 0$, with $\Theta_H = \theta_{\End_{(A, d_A)}} \cdot \varepsilon_{\scLa} + \eta_{\End_{(A, d_A)}} \cdot \theta^c_{\scLa} = {d_A}^2 \cdot \varepsilon_{\scLa}$, so $\alpha_A$ is an operadic curved twisting morphism.

The equivalence between $(2)$ and $(3)$ is obtained by denoting $\{l_n\}_{n \geq 0}$ the maps which characterise the coderivation $D_A$. 
The desired equation in $(3)$ is the projection of the equality ${D_A}^2 = 0$ to $A$, which characterises the whole equality since ${D_A}^2$ is also a coderivation.
\end{proof}

\subsubsection{Cofree curved coalgebra}
\label{sec: cofree coalgebra}

Let $(\Dc, d, \varepsilon)$ be a pg cooperad (no curvature a priori). 
In order to describe the notion of sub-$\Cc$-coalgebra generated by a pg module in $\ModApg$, we first describe the notion of sub-$\Dc$-coalgebra of a $\Dc$-coalgebra $D$ with relations given by $S$.

We follow and adapt the definitions given in \cite[Section 5.1]{BMDC20} and based on Appendix B in \cite{bV08}.
As in \cite{BMDC20}, we consider strict morphisms (that is morphisms $f : C \to D$ such that for all $p \geq 0$, $f(F_pC) = f(C) \cap F_pD$) in order to identify certain coimages with cokernels (and in the dual case certain images with kernels).

\begin{defn}
In this definition, we consider a pg cooperad $(\Dc, d, \varepsilon)$ (no curvature a priori).
\begin{itemize}
\item
Let $J \rightarrowtail D \twoheadrightarrow Q$ be an exact sequence in $\ModApg$, where $D$ is a $\Dc$-coalgebra. The epimorphism $C \twoheadrightarrow Q$ in $\ModApg$ is a \emph{coideal epimorphism} if $J \rightarrowtail D$ is a monomorphism of $\Dc$-coalgebras in $\ModApg$.
\item
Let $\xi : D \twoheadrightarrow S$ be an epimorphism in $\ModApg$, where $D$ is a $\Dc$-coalgebra. We consider the category $\Sc_\xi$ of exact sequences $(\textbf{S}) : J \rightarrowtail D \twoheadrightarrow Q$ such that $J \rightarrowtail D$ is a monomorphism of $\Dc$-coalgebras and such that the composite $J \rightarrowtail D \twoheadrightarrow S$ is equal to $0$. A morphism between $(\textbf{S})$ and $(\textbf{S}') : J' \rightarrowtail D \twoheadrightarrow Q'$ is given by a pair $(i : J \to J',\, p : Q \to Q')$ such that $i$ is a morphism of $\Dc$-coalgebras and $p$ is a morphism in $\ModApg$, and such that the following diagram commutes:
\[
\begin{tikzcd}
J \dar[swap,"i"]\ar[dr] &&\\
J' \rar & D \rar\ar[dr] & Q \dar["p"]\\
&& Q'.
\end{tikzcd}
\]
\item
We aim to consider the largest sub-$\Dc$-coalgebra of $D$ such that the post-composition with $D \twoheadrightarrow S$ is zero. This notion is given by the terminal object $(\overline{\textbf{S}}) : C(S) \rightarrowtail D \twoheadrightarrow (S)^c$ in $\Sc_\xi$, when the latter admits one. 
When it exists, the sub-$\Dc$-coalgebra $C(S)$ is called the \emph{sub-$\Dc$-coalgebra of $D$ with relations given by $S$}.
\end{itemize}
\end{defn}

We now make the coideal quotient $(S)^c$ and the sub-$\Dc$-coalgebra $C(S)$ explicit.

\begin{defn}
We denote by $i_D^S : D \hookrightarrow S \oplus D$ the inclusion and by $\pi_D^S : S \oplus D \twoheadrightarrow D$ the projection.
\begin{itemize}
\item
The \emph{multilinear part} in $S$ of $\Dc (S \oplus D)$ is given either by the cokernel
\[ \coker \left(\Dc (D) \xrightarrow{\id_\Dc \circ i_D^S} \Dc (S \oplus D)\right), \]
or equivalently, by the kernel
\[ \ker \left(\Dc (S \oplus D) \xrightarrow{\id_\Dc \circ \pi_D^S} \Dc (D)\right), \]
by means of the fact that the monoidal structure commutes with colimits and $i_D^S$ is a section of $\pi_D^S$. 
It is denoted by $\Dc (\underline{S} \oplus D)$. 
\item
The \emph{coideal quotient $(S)^c$} of $D$ generated by $\xi : D \twoheadrightarrow S$ is given by the coimage
\[ (S)^c := \coim \left(D \xrightarrow{\Delta_D} \Dc(D) \xrightarrow{\mathrm{proj} \cdot (\id_\Dc \circ (\xi \oplus \id_D))} \Dc (\underline{S} \oplus D)\right), \]
where the map $\mathrm{proj}$ is the projection $\Dc (S \oplus D) \to \Dc (\underline{S} \oplus D)$.
\end{itemize}
\end{defn}

We define in the next proposition the sub-$\Dc$-coalgebra $C(S)$ as the kernel of $D \twoheadrightarrow (S)^c$. 
Because $(S)^c$ is a coimage, the map $D \twoheadrightarrow (S)^c$ is a strict epimorphism (that is the filtration on $(S)^c$ is the quotient filtration). 
It follows that we have $(S)^c \cong \coker (C(S) \to D)$.

\begin{prop}
\label{prop: cooperad generated by, revisited}
The \emph{sub-$\Dc$-coalgebra of $D$ with relations given by $S$} is
\[ C(S) \coloneqq \ker(D \to (S)^c), \]
that is
\[ C(S) = \ker\left(D \xrightarrow{\Delta_D} \Dc (D) \xrightarrow{\id_\Dc \circ (\id_D, \xi)} \Dc \hcirc (D, S)\right), \]
where $\Dc \hcirc (D, S)$ is the part of $\Dc (D\oplus S)$ linear in $S$ (see the remark below).
\end{prop}

\begin{rem}
\label{rem: linearised free algebra}
The pg module $\Dc \hcirc (D, S)$ is defined by
\[
\Dc \hcirc (D, S) \coloneqq \hat\bigoplus_{n\geq 0} \Dc \hat\otimes_{\Sb_n} \left( \bigoplus_{1\leq i\leq n} D \hotimes \cdots \hotimes D\hotimes \underbrace{S}_{i\text{th position}} \hotimes D \hotimes \cdots \hotimes D \right).
\]
\end{rem}

\begin{proof}
To verify that $\Delta_D$ induces (by restriction) a comultiplication $\Delta_{C(S)} : C(S) \to \Dc(C(S))$, we first remark that the composite $\flat$
\[
C(S) \xrightarrow{\Delta_D} \Dc(D) \xrightarrow{\id_\Dc \circ \Delta_D} (\Dc \hcirc \Dc)(D) \xrightarrow{\mathrm{proj} \cdot (\id_{\Dc \hcirc \Dc} \circ (\xi \oplus \id_D))} (\Dc \hcirc \Dc)(\underline{S} \oplus D)
\]
is zero since it is equal, by coassociativity of $\Delta_D$, to
\[
(\mathrm{proj} \cdot (\id_{\Dc \hcirc \Dc} \circ (\xi \oplus \id_D))) \cdot (\Delta_\Dc \circ \id_D) \cdot \Delta_D = (\Delta_\Dc \circ \id_D) \cdot (\mathrm{proj} \cdot (\id_{\Dc} \circ (\xi \oplus \id_D))) \cdot \Delta_D
\]
which is zero on $C(S) = \ker(D \twoheadrightarrow (S)^c)$. 
Moreover, the projection of the map $\flat$ to each
\[
\Dc(n) \hotimes_{\Sb_n} (\Dc(C)^{\hotimes k} \hotimes C(S) \hotimes \Dc(C)^{\hotimes (n-1-k)})
\]
is given by applying the map $\id_\Dc \otimes \id_D^{\otimes k} \otimes (\mathrm{proj} \cdot (\id_{\Dc} \circ (\xi \oplus \id_D))) \otimes \id_D^{\otimes (n-1-k)}$ to the correct component of the image of $\Delta_D$. 
This ensures that we can corestrict $\Delta_{C(S)} = \Delta_D : C(S) \to \Dc(C(S))$.

Then, by definition, $C(S) = \ker\left((\mathrm{proj} \cdot (\id_{\Dc} \circ (\xi \oplus \id_D))) \cdot \Delta_D \right)$ so
\[
C(S) \subseteq C(S)' \coloneqq \ker\left(D \xrightarrow{\Delta_D} \Dc (D) \xrightarrow{\id_\Dc \circ (\id_D, \xi)} \Dc \hcirc (D, S)\right).
\]
For the same reason as before, the comultiplication $\Delta_D$ restricts (and corestricts properly) to $C(S)'$ and by an iterated application of the comultiplication, we see that an element in $C(S)'$ is in fact in $C(S)$, so the wanted equality.
\end{proof}

We now use this notion of sub-$\Dc$-coalgebra with generators given by a pg module $S \in \ModApg$ to build a cofree $\Cc$-coalgebra functor, where $\Cc$ is a curved altipotent cooperad.

\begin{prop}
\label{prop: cofree-forgetful adjunction for curved coalgebras}
The forgetful functor $U^c : \Calg \to \ModApg$ admits a left adjoint, called the \emph{cofree $\Cc$-coalgebra functor}, $F_{\Cc}^c : \ModApg \to \Calg$ given by
\[
(V, d_V) \mapsto F_{\Cc}^c(V, d_V) \coloneqq C(S) = \ker\left(\Cc(V) \twoheadrightarrow (S)^c\right) \subset \left(\Cc(V), d_{\Cc(V)}\right),
\]
where $\Cc(V)$ is seen as a coalgebra over the pg cooperad underlying $\Cc$ and for
\[
S \coloneqq \Cc(V)/\ker\left({d_{\Cc(V)}}^2 + \left(\theta^c \otimes \id_{\Cc(V)}\right)\cdot \Delta_{\Cc(V)}\right).
\]
We therefore have for all $(V, d_V) \in \ModApg$ and $(C, d_C) \in \Calg$ the bijections
\[
\Hom_{\Calg}\left( (C, d_C), F_{\Cc}^c(V, d_V)\right) \cong \Hom_{\ModApg}\left( U^c(C, d_C), (V, d_V)\right).
\]
\end{prop}

\begin{proof}
The construction $F_\Cc^c$ is functorial in $(V, d_V)$ since the map $\Cc(V) \to S$ is functorial in $(V, d_V)$. 
By definition of $C(S)$, the composition $C(S) \to S$ is zero, that is
\[
C(S) \subseteq \ker\left({d_{\Cc(V)}}^2 + \left(\theta^c \otimes \id_{\Cc(V)}\right)\cdot \Delta_{\Cc(V)}\right).
\]
It follows that $C(S)$, which is a priori a coalgebra over the pg cooperad underlying $\Cc$, is in fact a coalgebra over the curved altipotent cooperad $\Cc$. 

The universal property that $C(S)$ satisfies ensures that the $\Cc$-coalgebra morphisms $(C, d_C) \to (C(S), d_{C(S)})$, that is also a morphisms of the coalgebras over the pg cooperad underlying $\Cc$, coincide with morphisms $(C, d_C) \to (\Cc(V), d_{\Cc(V)})$ of coalgebras over the pg cooperad underlying $\Cc$ such that the composite $C \to \Cc(V) \twoheadrightarrow S$ is zero. 
This last condition is automatically satisfied because of the fact that the relation ${d_{C}}^2 + \left(\theta^c \otimes \id_{C}\right)\cdot \Delta_{C}$ is satisfied on $C$ and that the morphisms of coalgebras in the pg context commute with the predifferentials and the comultiplication.
\end{proof}

\begin{rem}
For the curious reader, it is certainly a good exercise to understand the cofree cooperad recalled in Proposition \ref{prop: cofree cooperad} as a particular case of a cofree coalgebra over a colored cooperad for which altipotent cooperads are example of coalgebras. Hint: considering the infinitesimal coideal $\coideal{\Cc}$ is a way to describe the cooperad structure of $\Cc$ as a coalgebra over a certain colored cooperad.
\end{rem}

\begin{defn}
\label{def: quasi-cofree coalg}
We call a $\Cc$-coalgebra $C$ \emph{quasi-cofree} if there exists a graded module $V$ and a predifferential (see Definition \ref{def: der prediff and coder}) $d : \Cc(V) \to \Cc(V)$ such that $C = C(S)$ for
\[ S \coloneqq \Cc(V)/\ker\left( d^2 + \left(\theta^c \otimes \id_{\Cc(V)}\right)\cdot \Delta_{\Cc(V)} \right). \]
\end{defn}

This notion is different from that of a cofree $\Cc$-coalgebra because in the quasi-cofree case the predifferential $d$ is not a priori induced by a map $d_V : V \to V$ (that is it might be different from any map $d_\Cc \circ \id_V + \id_\Cc \circ' d_V$).

\subsubsection{Cobar construction of a $\Cc$-coalgebra}

Let $(C, d_C, \Delta_C)$ be a $\Cc$-coalgebra. 
As in the previous section, we denote by $d_\alpha^l$ the derivation of the free $\Pc$-algebra which extends the composite $C \xrightarrow{\Delta_C} \Cc(C) \xrightarrow{\alpha \circ \id_C} \Pc(C)$. 
We also write $d_{\Pc(C)} = d_\Pc \circ \id_C + \id_\Cc \circ' d_C$ which is a predifferential of the $\Pc$-algebra $\Pc(C)$.

\begin{lem}
Because $\alpha$ is a curved twisting morphism, the predifferential $d_{\Pc(C)} + d_\alpha^l$ satisfies
\[
(d_{\Pc(C)}-d_\alpha^l)^2 = \gamma_{\Pc(C)}(\theta \circ \id_{\Pc(C)}).
\]
\end{lem}

\begin{proof}
The proof is similar to the one of Lemma \ref{lem: bar prediff}. 
We compute
\begin{align*}
(d_{\Pc(C)}-d^l_\alpha)^2 & = {d_\Pc}^2 \circ \id_C + \id_\Pc \circ' {d_C}^2 + d^l_{-\partial \alpha - \frac{1}{2}[\alpha, \alpha]}\\
& = [\theta, \id_\Pc ] \circ \id_C - \id_\Pc \circ' ((\theta^c \otimes \id_C) \cdot \Delta_C) + d^l_{\Theta}\\
& = \gamma_{\Pc(C)}(\theta \circ \id_{\Pc(C)}).
\end{align*}
\end{proof}

\begin{defn}
The \emph{cobar construction of a $\Cc$-coalgebra $(C, d_C, \Delta_C)$ with respect to $\alpha$} is
\[
\Om_\alpha C \coloneqq (\Pc(C), d_o\coloneqq d_{\Pc(C)}-d^l_\alpha).
\]
It satisfies ${d_o}^2 = \gamma_{\Pc(C)}(\theta \circ \id_{\Pc(C)})$ by the previous lemma and this therefore defines a functor $\Om_\alpha : \Calg \to \Palg$.
\end{defn}

\begin{rem}
The cobar construction is quasi-free since it is equal to the quotient $\Pc$-algebra
\[
\Pc(C)/\im\left( {d_o}^2 - \gamma_{\Pc(C)}(\theta \circ \id_{\Pc(C)})\right) = \Pc(C).
\]
\end{rem}

\subsection{Convolution algebra}
\label{sec: convolution algebra}

Let $(\Pc, \gamma_\Pc, d_\Pc, \theta_\Pc)$ be a curved sa operad and let $(\Cc, d_\Cc, \Delta_\Cc, \theta_\Cc^c)$ be a curved altipotent cooperad. \\

Let $C$ be a $\Cc$-coalgebra and $A$ be a $\Pc$-algebra. We consider the graded module $\homf(C, A)$ of filtered $\Rf$-linear maps from $C$ to $A$. 
As for the convolution operad, there is a natural filtration on $\homf(C, A)$ given by
\[
F_p \homf(C, A) \coloneqq \{ \varphi \in \homf(C, A) \textrm{ s.t. } \forall q\geq 0,\ \varphi(F_qC) \subset F_{p+q}A\}.
\]
For $f \in \Hom(\Cc, \Pc)(k)$ and $\varphi_1, \ldots, \varphi_k \in \homf(C, A)$, we define a $\Hom(\Cc, \Pc)$-algebra structure $\gamma_h$ by the composite
\begin{multline*}
\gamma_h(f ; \varphi_1, \ldots, \varphi_k) : C \xrightarrow{\Delta_C} \Cc(C) \twoheadrightarrow \Cc(k) \hotimes_{\Sb_k} C^{\hotimes k}\\
\xrightarrow{\sum_{\sigma \in \Sb_k} f^\sigma \otimes \varphi_{\sigma(1)}\otimes \cdots \otimes \varphi_{\sigma(k)}} \Pc(k)\hotimes_{\Sb_k} A^{\hotimes k} \xrightarrow{\gamma_A} A,
\end{multline*}
which is well defined and which is a filtered map (when we consider the filtration $F_\bullet$ on $\Hom(\Cc, \Pc)$) since all the involved maps are. 
For $\varphi \in \homf(C, A)$, we define the degree $-1$ map
\[
\partial(\varphi) \coloneqq d_A \cdot \varphi - (-1)^{|\varphi|} \varphi \cdot d_C.
\]

For our filtration on $\homf(C, A)$ to be compatible with the filtration $\Gamma_\bullet$ on $\Hom(\Cc, \Pc)$, we need a second filtration. 
We consider on $\homf(C, A)$ all the filtrations $G_\bullet$ compatible with the filtration $\Gamma_\bullet$ on $\Hom(\Cc, \Pc)$, with $\gamma_h$, and with $\partial$ such that
\begin{equation}
\label{cond: filtration conv alg}
	\left\lbrace
\begin{aligned}
G_0 \homf(C, A) & = \homf(C, A)\\
G_1 \homf(C, A) & \supset F_1 \homf(C, A) + \gamma_h \left(\Gamma_1 \Hom(\Cc, \Pc) (\homf(C, A))\right)\\
G_p \homf(C, A) & \supset F_p\homf(C, A) \text{ for all } p \geq 2.
\end{aligned}
	\right.
\end{equation}
The constant filtration given by $G_p \homf(C, A) = \homf(C, A)$, for $p \geq 0$, is a filtration satisfying the above condition \eqref{cond: filtration conv alg}. 
Therefore there exists a smallest such filtration, namely the intersection of all filtrations satisfying \eqref{cond: filtration conv alg}. 
We call it the \emph{lower central series} of $\homf(C, A)$ and we denote it by $\varGamma_\bullet$. 

For $a : I \to \Gamma_1 \Hom(\Cc, \Pc)(1)$, we fix $\partial_a \coloneqq \partial + \gamma_h(a;-)$ on $\homf(C, A)$.

\begin{lem}
\label{lem: convolution algebra is complete}
\begin{enumerate}
\item
The gr-dg module $(\homf(C, A), \partial)$ is complete for the filtration $\varGamma_\bullet$.
\item
Similarly the gr-dg module $(\homf(C, A), \partial_a)$ is complete for the filtration $\varGamma_\bullet$.
\end{enumerate}
\end{lem}

\begin{proof}
\begin{enumerate}
\item
The square of $\partial$ increases the filtration degree because $\theta_\Pc \in F_1\Pc$ and $\theta_\Cc^c \cdot \varepsilon_\Pc \in \Gamma_1 \Hom(\Cc, \Pc)$. 
The fact the filtration is complete follows from the fact that the lower central filtration $\Gamma_\bullet$ on $\Hom(\Cc, \Pc)$ is complete and that the maps $\Delta_C$ and $\gamma_A$ and the applications that we consider in $\homf(C, A)$ are filtered.
\item
We have
\begin{align*}
{\partial_a}^2 & = \partial_a(\partial + \gamma_h(a ; -)) = \partial^2 + \partial \gamma_h(a ; -) + \gamma_h(a ; \partial -) + \gamma_h(a ; \gamma_h(a ; -))\\
& = \partial^2 + \gamma_h(\partial a + \frac12 [a, a] ; -)
\end{align*}
so $\partial_a$ is also a gr-differential by the definition of $\varGamma_\bullet$ and the fact that $a \in \Gamma_1 \Hom(\Cc, \Pc)(1)$.
\end{enumerate}
\end{proof}

\begin{prop}\
\label{prop: algebra over the convolution curved operad}
\begin{enumerate}
\item
The gr-dg module $(\homf(C, A), \varGamma_\bullet, \gamma_h, \partial)$ is an algebra over the convolution curved sa operad $(\Hom(\Cc, \Pc), \Gamma_\bullet, \gamma_H, \partial, \Theta_H)$.
\item
Similarly, given a map
\[
a : I \to \Gamma_1 \Hom(\Cc, \Pc)(1) = \Hom(\coideal{\Cc}, \Pc)(1),
\]
the gr-dg module $(\homf(C, A), \varGamma_\bullet, \gamma_h, \partial_{a})$ is an algebra over the convolution curved sa operad $\Hom(\Cc, \Pc)_{a} = (\Hom(\Cc, \Pc), \Gamma_\bullet, \gamma_H, \partial_{a}, \Theta_a)$.
\end{enumerate}
\end{prop}

\begin{proof}
\begin{enumerate}
\item
The algebra structure $\gamma_h$ is associative since $\Delta_C$ is coassociative and $\gamma_A$ is associative. 
The composition $\gamma_h$ commutes with the predifferentials.
Finally we compute
\begin{align*}
\partial^2(\varphi) & = {d_A}^2 \cdot \varphi - \varphi \cdot {d_C}^2\\
& = \gamma_A \cdot (\theta ; \varphi) + (\theta^c \otimes \varphi) \cdot \Delta_C\\
& = \gamma_h (\Theta_H ; \varphi).
\end{align*}
This shows the first result.
\item
Similarly, $\gamma_h$ commutes with the predifferentials $\partial_a$ since its commutes with the predifferentials $\partial$ and since $\Delta_C$ is coassociative and $\gamma_A$ is associative (the proof is similar to the proof given in Lemma \ref{lem: lax morphism unlax}). 
Moreover by the computation made in the proof of Lemma \ref{lem: convolution algebra is complete}, we have on $\homf(C, A)$
\[
{\partial_{a}}^2 = \partial^2 + \gamma_h(\partial a + \frac12 [a, a] ; -) = \gamma_h(\Theta_a ; -).
\]
\end{enumerate}
\end{proof}

\begin{defn}
Let $\alpha : \Cc \to \Pc$ be an operadic curved twisting morphism. 
\begin{enumerate}
\item
We denote by $\{ l_n^\alpha \}_{n\geq 0}$ the $\scLi$-algebra structure on $(\homf(C, A), \partial_{a_\alpha})$ associated with $\alpha$ by means of Propositions \ref{prop: twisting morphisms as morphisms}, \ref{prop: algebra over the convolution curved operad} and \ref{prop: equivalent definition of homotopie cLie algebras}. 
In particular $l_1^\alpha =\partial_{a_\alpha}$, where $a_\alpha : I \to \Gamma_1 \Hom(\Cc, \Pc)$ is given by $a_\alpha(1) = \alpha(1) : \Cc(1) \to \Pc(1)$. 
We denote this shifted curved $\Li$-algebra by $\homa(C, A)$.
\item
An \emph{algebraic twisting morphism with respect to $\alpha$} is a map $\varphi : C \rightarrow A$ in $\homa(C, A)$ to degree $0$ solution of the Maurer--Cartan equation
\[
\sum_{n \geq 0} \frac{1}{n!}l_n^{\alpha}(\varphi, \ldots, \varphi) = 0.
\] 
We denote by $\Tw_{\alpha}(C, A)$ the set of algebraic twisting morphisms with respect to $\alpha$.
\end{enumerate}
\end{defn}

\begin{rem}
\begin{itemize}
\item
Going through the isomorphisms, we get that the map $l_k^\alpha$ applied to $\varphi_1 \otimes \cdots \otimes \varphi_k$ is given by the composite
\[
C \xrightarrow{\Delta_C} \Cc(C) \twoheadrightarrow \Cc(k) \hotimes_{\Sb_k} C^{\hotimes k} \xrightarrow{\sum_{\sigma \in \Sb_k} \alpha(k)^\sigma \otimes \varphi_{\sigma(1)}\otimes \cdots \otimes \varphi_{\sigma(k)}} \Pc(k)\hotimes_{\Sb_k} A^{\hotimes k} \xrightarrow{\gamma_A} A.
\]
\item
The presence of an infinite sum in the Maurer--Cartan equation doesn't cause any trouble since it corresponds to the application of an infinite sum to different terms of a convergent sum.
\end{itemize}
\end{rem}

As for operadic curved twisting morphisms, we have the following theorem which generalises the bar-cobar adjunction given in \cite[Proposition 2.18]{GJ94} when there are no curvatures.

\begin{thm}
\label{thm: algebraic bar-cobar adjunction}
Let $\alpha : \Cc \to \Pc$ be an operadic curved twisting morphism. 
For any coalgebra $C$ over the curved altipotent cooperad $\Cc$ and any algebra $A$ over the curved sa operad $\Pc$, there are natural bijections
\[
\Hom_{\Palg}(\Om_\alpha C, A) \cong \Tw_\alpha(C, A) \cong \Hom_{\Calg}(C, \B_\alpha A).
\]
\end{thm}

\begin{proof}
Let us denote by $\Pc^{g}$ the underlying graded operad to the curved sa operad $\Pc$. 
By definition, a morphism of $\Pc$-algebras is a morphism of the underlying graded $\Pc^{g}$-algebras commuting with the predifferentials. 
Since $\Pc(C)$ is a free $\Pc^{g}$-algebra, any morphism of $\Pc^{g}$-algebras $\Phi : \Pc(C) \to A$ is characterised by its restriction $\varphi : C\to A$ on $C$. 
The morphism $\Phi$ commutes with the predifferentials when $\Phi \cdot d_o = d_A \cdot \Phi$. 
As already used implicitly in the proof of Proposition \ref{prop: cobar const is a functor}, the difference $\Phi \cdot d_o - d_A \cdot \Phi$ is a derivation. 
Therefore to prove that it is zero, it is enough to prove that the restriction to $C$ is zero. 
We have
\[
(\Phi \cdot d_\alpha^l)_{|C} : C \xrightarrow{\Delta_C} \Cc(C) \xrightarrow{\alpha \circ \varphi} \Pc (A) \xrightarrow{\gamma_A} A.
\]
It follows that $\Phi$ commutes with the predifferentials if and only if $0 = (\Phi \cdot d_o - d_A \cdot \Phi)_{|C} = \varphi \cdot d_C - \gamma \cdot (\alpha \circ \varphi) \cdot \Delta_C - d_A \cdot \varphi = -\sum_{n \geq 0} \frac{1}{n!}l_n^{\alpha}(\varphi, \ldots, \varphi)$, that is $\varphi$ is an algebraic twisting morphism with respect to $\alpha$.

Similarly, a $\Cc^g$-coalgebra morphism $C \to \Cc(A)$ is characterised by its corestriction $\varphi : C \to A$ to $A$ and it commutes with the predifferentials if and only if $\varphi$ is an algebraic twisting morphism with respect to $\alpha$.
\end{proof}

\begin{example}
\begin{enumerate}
\item
Through the bar-cobar adjunction, the identity morphism $\id : \B_\alpha A \to \B_\alpha A$ corresponds to the algebraic twisting morphism $\pi_\alpha : \B_\alpha A \twoheadrightarrow A$.
\item
Through the bar-cobar adjunction, the identity morphism $\id : \Om_\alpha C \to \Om_\alpha C$ corresponds to the algebraic twisting morphism $\iota_\alpha : C \rightarrowtail \Om_\alpha C$.
\end{enumerate}
\end{example}

\subsection{Cofibrant resolution of an algebra over a curved operad}
\label{section: resolution of an algebra over a curved operad}

In this section, given a Koszul inhomogeneous quadratic curved operad $\Pc$, we provide a functorial cofibrant resolution $R\to A$ of a $\Pc$-algebra $A$. Explicitly, $\Pc(0) \to R$ is a cofibration of $\Pc$-algebras and
$R \to A$ is a graded quasi-isomorphism and a strict surjection. This resolution is the counit of the bar-cobar adjunction.

\begin{prop}
\label{prop: algebraic bar-cobar resolution}
Let $\Pc$ be a curved sa operad, $\Cc$ be a curved altipotent cooperad and $\alpha : \Cc \to \Pc$ be an operadic curved twisting morphism. 
We assume that $\Gr \Pc \circ_{\Gr \alpha} \Gr \Cc \circ_{\Gr \alpha} \Gr \Pc$ (described for example in \cite[Section 5]{HM12}) is cofibrant as a $\Gr \Pc$-right module (in the model category structure described in \cite{bF09a}) and that $\Gr \Pc \circ_{\Gr \alpha} \Gr \Cc \circ_{\Gr \alpha} \Gr \Pc \to \Gr \Pc$ is a quasi-isomorphism.

In this situation, the bar-cobar construction $\Om_{\alpha} \B_{\alpha} A$ is a resolution of the $\Pc$-algebra $A$, that is, the counit of the bar-cobar adjunction
\[
\Phi_\alpha : \Om_{\alpha} \B_{\alpha} A \xrightarrow{\sim} A
\]
is a graded quasi-isomorphism and a strict surjection.
\end{prop}

\begin{proof}
The functor $\Gr$ is strong monoidal since $\Kf$ is a field of characteristic 0 so $\Gr (\Om_{\alpha} \B_{\alpha} A) \cong \Omega_{\Gr \alpha} \mathrm{B}_{\Gr \alpha} \Gr A$ as $\Gr \Pc$-algebras (by a careful inspection of the differentials). The map $\Gr \alpha : \Gr \Cc \to \Gr \Pc$ is an operadic twisting morphism between the curved cooperad $\Gr \Cc$ and the semi-augmented dg operad $\Gr \Pc$. 
In the proof of Proposition 5.2.8 in \cite{HM12}, we use Theorem 15.1.A in \cite{bF09a} to prove that the bar-cobar construction is a quasi-isomorphism. Assuming that $\Gr \Pc \circ_{\Gr \alpha} \Gr \Cc \circ_{\Gr \alpha} \Gr \Pc$ is cofibrant as a $\Gr \Pc$-right module permits to use \cite[Theorem 15.1.A]{bF09a} more directly and to prove, as in \cite[Proposition 5.2.8]{HM12}, that we have a quasi-isomorphism $\Omega_{\Gr \alpha} \mathrm{B}_{\Gr \alpha} \Gr A \xrightarrow{\sim} \Gr A$. 
The morphism $\Phi_\alpha$ is a strict surjection since it coincides with the projection $\Pc \hcirc \Cc (A) \twoheadrightarrow I\circ (A) \cong A$.
\end{proof}

We recall that $\B_\alpha A = (\Cc(A), d_b = d_{\Cc(A)} + d_\alpha^r)$ and that the predifferential on $\Om_\alpha \B_\alpha A$ is $d_o = d_\Pc \circ \id_{\B_\alpha A} + \id_\Pc \circ' d_b - d_\alpha^l$. 

\begin{prop}
\label{prop: cofibrant bar-cobar resolution}
Assume that $\Rf = \Kf$. 
Let $\Pc$ be a curved sa operad and $\Cc$ be a curved altipotent cooperad. Let $\alpha : \Cc \to \Pc$ be an operadic curved twisting morphism. 
Assume that the filtration $F_\bullet$ on $\Cc(A)$ comes from a graduation $\Gr_\bullet$ (on the module level). 
Assume that $\Cc(A) = \oplus_{l\geq 0} T_l$ with $T_0 = \{0\}$ and the predifferential $d_b$ admits a decomposition $d_b = \delta_0 + \delta_1$ such that for any $l \geq 0$, we have 
\begin{enumerate}
\item
$\delta_0 : \Gr_\bullet T_l \to \Gr_{\bullet + 1}(\oplus_{k\geq l+1}T_{k})$ and $\im \delta_0 = \ker \delta_0$,
\item
\label{eq: delta1}
$\delta_1-d_\alpha^l : T_l \to \Pc(\oplus_{k \leq l}T_k)$ with $\Gr_\bullet \delta_1 : \Gr_\bullet T_l \to \Gr_\bullet \Pc(\oplus_{k < l}T_k)$.
\end{enumerate}
Then the map $\Pc(0) \to \Om_\alpha \B_\alpha A$ is a cofibration in the model category structure on $\Pc$-algebras defined in \cite[Appendix C]{BMDC20}.
\end{prop}

\begin{proof}
We define on $\B_\alpha A$ a filtration $(S_l)_{l \geq 0}$ as follows
\[
\left\{\begin{array}{lcll}
S_0 & \coloneqq & \{0\},\\
S_1 & \coloneqq & \ker(\Gr ({d_o}_{|T_1})) \oplus d_b\ker(\Gr ({d_o}_{|T_1})), &\\ 
S_2 & \coloneqq & S_1 \oplus W_1 \oplus d_b W_1 & \supset T_1,
\end{array}\right.
\]
where $W_1$ is a direct complement of $\ker({\Gr d_o}_{|T_1})$ into $T_1$ (such a direct complement exists since $\Rf = \Kf$ is a field),  so $T_1 = \ker({\Gr d_o}_{|T_1}) \oplus W_1$.
A direct sum appears in the definitions of $S_1$ and $S_2$ since $T_1 \cap \ker \delta_0 = T_1 \cap \im {\delta_0}_{|T_0} = \{ 0\}$. 
We remark that $T_1 \oplus d_b T_1 \subset S_2$. 
For $l \geq 2$, we prove by induction that $S_{2(l-1)} \supset T_0 \oplus \cdots \oplus T_{l-1}$ and we fix by induction a direct complement $U_l \subset T_l$ of $S_{2(l-1)} \cap (T_0 \oplus \cdots \oplus T_l)$ in $T_0 \oplus \cdots \oplus T_l$, which implies
\[
T_0 \oplus \cdots \oplus T_l \subset S_{2(l-1)} \oplus U_l.
\]
Let also $W_l$ be a direct complement of $\ker (\Gr ({d_o}_{|U_l}))$ into $U_l$, so
\[
U_l = \ker (\Gr ({d_o}_{|U_l})) \oplus W_l. 
\]
Then we define for each $l \geq 2$
\[
\left\{\begin{array}{lcll}
S_{2l-1} & \coloneqq & S_{2(l-1)} \oplus \ker({\Gr d_o}_{|U_l}) \oplus d_b\ker({\Gr d_o}_{|U_l}), & \\
S_{2l} & \coloneqq & S_{2l-1} \oplus W_l \oplus d_b W_l & \supset T_1 \oplus \cdots \oplus T_{l}.
\end{array}\right.
\]
We remark that $T_l + d_b T_l \subset S_{2l}$. 
Direct sums appear in the definitions of $S_{2l-1}$ and $S_{2l}$ since $U_l \cap \ker \delta_0 = U_l \cap \im ({\delta_0}_{|\oplus_{k\leq l-1}T_{k}}) = \{0\}$. 
Indeed, let $t \in \oplus_{k\leq l-1}T_{k}$ such that $\delta_0(t) \in U_l$, then $d_b(t) = \delta_0(t) + \delta_1(t) \in S_{2(l-1)}$ and, by assumption on $\delta_1$, $\delta_1(t) \in \oplus_{k\leq l-1}T_{k} \subset S_{2(l-1)}$. It follows that $\delta_0(t) \in S_{2(l-1)}  \cap U_l = \{0\}$ by definition of $U_l$. 

We have thus constructed an exhaustive increasing filtration $(S_l)_l$ on $\B_\alpha A$ of free $\Kf$-modules such that $S_{l-1} \hookrightarrow S_l$ are split monomorphisms of (complete) modules with cokernels isomorphic to a sum of (complete) modules
\[
S_l / S_{l-1} \cong \oplus_\alpha \left(\xi^\alpha \cdot \Kf \oplus \zeta^\alpha \cdot \Kf\right)
\]
where $\xi^\alpha$ is in some homological degree $n_\alpha +1$ and some filtration degree $q_\alpha$ and $\zeta^\alpha$ is in homological degree $n_\alpha $ and filtration degree $q_\alpha+1$ (because (1) and (2)). 
The predifferential $d_o$ is such that $d_o(\xi^\alpha) + \zeta^\alpha \in \left(\Pc (S_{l-1}),\, d_o\right)$ because $\delta_0$ increases the filtration degree exactly by 1 and because of condition \eqref{eq: delta1} ($\zeta^\alpha$ is defined to be the sum of $\delta_0 \xi^\alpha$ with the filtration degree $q_\alpha + 1$ part of $(\delta_1-d_\alpha^l)(\xi^\alpha)$), and $d_o(\zeta^\alpha)$ is obtained by the fact that ${d_o}^2(\xi^\alpha) = \theta\otimes \xi^\alpha$. 
It follows by \cite[Corollary C.41]{BMDC20} that the map $\Pc(0) \to \Om_\alpha \B_\alpha A$ is a cofibration.
\end{proof}

\begin{rem}
In the previous Proposition, Conditions (1) and (2) and the fact that $\Om_\alpha \B_\alpha A$ is a $\Pc$-algebra imply that $\theta_\Pc \in \Gr_1 \Pc$.
\end{rem}

We now prove three corollaries corresponding to the following situations:
\begin{itemize}
\item
$\pi : \B\Pc \to \Pc$,
\item
$\iota : \Pc^\antishriek \to \hat \Omega \Pc^\antishriek = \Pc_\infty$, for $\Pc = \cAs$, $\cuAs$, $\cLie$ or $\cCx$,
\item
$\kappa : \Pc^\antishriek \to \Pc$ 
for $\Pc = \cAs$, $\cuAs$, $\cLie$ or $\cCx$.
\end{itemize}

\begin{cor}
\label{cor: cof res pi}
Let $\Pc$ be a curved sa operad such that:
\begin{enumerate}
\item
$\Gr \Pc$ is a connected bounded below weight filtered operad,
\item
there exists $d \geq 0$ such that for each $w$, $(\Gr \Pc)^{(\leq w)}$ is concentrated in homological degree greater than or equal to $-dw$.
\end{enumerate} 
 Then its bar construction $\B \Pc$ is a curved altipotent cooperad such that $\Gr \B \Pc \cong \Bm \Gr \Pc$ is connected bounded below weight filtered (by the weight filtration induced by the one on $\Pc$). 

Let $\pi : \B\Pc \to \Pc$ be the operadic curved twisting morphism defined in Example \ref{ex: twisting morph}. 
Let $A$ be a $\Pc$-algebra. 
We assume that $\Rf = \Kf$. 
Assume moreover that:
\begin{enumerate}
\setcounter{enumi}{2}
\item
the filtration $F_\bullet$ on $\Pc$ and on $A$ comes from a graduation $\Gr_\bullet$ (on the module level) so that the induced filtration on $\B \Pc (A)$ also comes from a graduation,
\item
$\theta_\Pc \in \Gr_1 \Pc$.
\end{enumerate}

Then the counit of the bar-cobar adjunction
\[
\Phi_\pi : \Om_\pi \B_\pi A \to A
\]
is a cofibrant resolution, that is an acyclic fibration such that $\Pc(0) \to \Om_\pi \B_\pi A$ is a cofibration.
\end{cor}

\begin{proof}
Using Assumption (1), we obtain by \cite[Proposition 5.1.5]{HM12} that $\Gr \Pc \circ_{\Gr \pi} \Bm \Gr\Pc \circ_{\Gr \pi} \Gr \Pc \to \Gr\Pc$ is a quasi-isomorphism. 
Moreover, by Assumption (2), both domain and codomain are cofibrant as $\Gr \Pc$-right modules. 
For the domain, this follows from \cite[Proposition 14.2.2]{bF09a} with the filtration on $K \coloneqq \Gr \Pc \circ_{\Gr \pi} \Gr \B\Pc$ (starting at $\lambda \geq 0$) given by
\[
K_\lambda \coloneqq \oplus_{m+dw\leq \lambda} (\Gr \Pc \circ_{\Gr \pi} \Gr \B\Pc)^{(\leq w)}_m. 
\]
Indeed, the differentials on $\Gr \Pc$ and $\Bm \Gr \Pc$ decreases the homological degree by 1 and do not increase the weight degree so gives a map $K_\lambda \to K_{\lambda-1}$ and the differentials $d_{\Gr \alpha}^l$ and $d_{\Gr \alpha}^r$ can at most increase the homological degree by $dw-1$ when it decreases the weight degree by $w$ (for $d_{\Gr \alpha}^r$) so it gives a map $K_\lambda \to K_{\lambda-1}$. 
It follows by Proposition \ref{prop: algebraic bar-cobar resolution} that $\Phi_\alpha$ is a graded quasi-isomorphism and a strict surjection.

Then, assuming (3) and (4), we have to check the hypotheses of Proposition \ref{prop: cofibrant bar-cobar resolution}. 
We define $\tilde T_{l+1} \coloneqq \B \Pc^{(l)}(A)$ the module built from trees with $l$ vertices and
\[
T_{2l-1} \coloneqq \im\left((d_1 \circ \id_A + \id_{\B \Pc} \circ' d_A)_{|\tilde T_{l}}\right) \subset \tilde T_l
\]
Then let $T_{2l}$ be a direct complement of $T_{2l-1}$ in $\tilde T_{l}$. 
We therefore have obtained a decomposition of $\B \Pc(A)$ as the direct sum $\oplus_{l\geq 0}T_l$ with $T_0 = \{0\}$ and satisfying the following properties. 
We fix $\delta_0 \coloneqq d_0 \circ \id_A$ and $\delta_1 \coloneqq (d_1 + d_2) \circ \id_A + \id_{\B \Pc} \circ' d_A$. 
Then $\delta_0$ satisfies $\delta_0 : \Gr_\bullet T_l \to \Gr_{\bullet + 1}(\oplus_{k \geq l+1}T_{k})$ since it increases the number of vertices in $\B \Pc$ (and $\theta_\Pc \in \Gr_1 \Pc$). 
Moreover, using the homotopy $h : \B \Pc \to \B \Pc$ defined in Theorem \ref{thm: bar-cobar resolution}, we get that $h\circ \id_A$ is a homotopy for the differential $d_0 \circ \id_A$. 
It follows that $\im \delta_0 = \ker \delta_0$. 
We also have $\delta_1-d_\pi^l : T_l \to \Pc(\oplus_{k \leq l}T_k)$ with $\Gr_\bullet \delta_1 : \Gr_\bullet T_l \to \Gr_\bullet \Pc(\oplus_{k < l}T_k)$ by construction (since ${d_1}^2$ and ${d_A}^2$ increases the filtration degree by 1). 
It follows that Proposition \ref{prop: cofibrant bar-cobar resolution} applies so $\Om_\pi \B_\pi A$ is cofibrant.
\end{proof}

To prove the next corollary, we require the following proposition.

\begin{prop}
\label{prop: augbarres}
Let $\Cc$ be a curved coaugmented conilpotent cooperad which is connected bounded below weight filtered. The operadic curved twisting morphism $\iota : \Cc \rightarrow \Omega \Cc$ provides a quasi-isomorphism
\[
\xi : \Omega \Cc \circ_{\iota} \Cc \circ_{\iota} \Omega \Cc \xrightarrow{\sim} \Omega \Cc.
\]
\end{prop}

\begin{proof}
The proof is similar to the proof of \cite[Proposition 5.1.5]{HM12}. 
We first remark that the curvature $\theta^c : \Cc \to I$ strictly decreases the weight because the weight filtration on $\Cc$ is connected ($\Cc^{(\leq 0)} = I$). 
The weight filtration on $\Cc$ induces a weight filtration on $\Omega \Cc$ given by the total weight. This gives a weight filtration $(\Omega \Cc \circ_{\iota} \Cc \circ_{\iota} \Omega \Cc)^{(\leq p)}$ on $\Omega \Cc \circ_{\iota} \Cc \circ_{\iota} \Omega \Cc$ and a weight filtration $\Omega \Cc^{(\leq p)}$ on $\Omega \Cc$. These weight filtrations are filtrations of chain complexes since the differentials either preserve or decrease the weight. The filtrations are exhaustive and bounded below and the map $\Omega \Cc \circ_{\iota} \Cc \circ_{\iota} \Omega \Cc \rightarrow \Omega \Cc$ preserves the filtrations. We apply the classical theorem of convergence of spectral sequences (cf. Theorem 5.5.1 of \cite{cW94}) to obtain
\[
\left\{ \begin{array}{l}
E_{p,q}^{\bullet} \Rightarrow \mathrm{H}_{p+q}(\Omega \Cc \circ_{\iota} \Cc \circ_{\iota} \Omega \Cc)\\
{E'}_{p,q}^{\bullet} \Rightarrow \mathrm{H}_{p+q}(\Omega \Cc).
\end{array} \right.
\]
Since the differential of $E_{\bullet, \bullet}^{0}$ is the weight preserving part of the differential of $\Omega \Cc \circ_{\iota} \Cc \circ_{\iota} \Omega \Cc$, the isomorphism of graded vector spaces $E_{\bullet, \bullet}^{0} \cong \Omega \Cc^{(\bullet)} \circ_{\iota} \Cc^{(\bullet)} \circ_{\iota} \Omega \Cc^{(\bullet)}$, where $\Cc^{(\bullet)} = \Cc^{(\leq \bullet)}/\Cc^{(\leq\bullet -1)}$ is a dg coaugmented conilpotent cooperad (since $\theta^c$ strictly decreases the weight), is an isomorphism of dg modules. Because $\Cc^{(\bullet)}$ is a dg coaugmented conilpotent cooperad, it remains to prove that $E_{\bullet, \bullet}^{0} \xrightarrow{\sim} \Omega \Cc^{(\bullet)}$ for a dg coaugmented conilpotent cooperad $\Cc^{(\bullet)}$. 
This follows from the same proof as Theorem 4.18 in \cite{bV07}  (where the author proves the quasi-isomorphism $\Omega \Cc \circ_{\iota} \Cc \xrightarrow{\sim} I$) with a slight modification on the homotopy map which has to be symmetrise in order to apply to non reduced cooperad ($\Cc(0) \neq \{0\}$ in general, see the proof of Theorem 3.4.4 in \cite{HM12} for an explicit description). 
As a consequence, $E^{1}_{p,q} = \mathrm{H}_{p+q}(\Omega \Cc^{(p)}) = {E'}_{p,q}^{1}$. 
Then $E_{p,q}^{r}$ and ${E'}_{p,q}^{r}$ coincide for $r \geq 1$ and $\xi$ induces an isomorphism between $E_{p,q}^{\infty}$ and ${E'}_{p,q}^{\infty}$. This concludes the proof.
\end{proof}

\begin{cor}
\label{cor: cof res iota}
Let $\Cc$ be a curved altipotent cooperad such that:
\begin{enumerate}
\item
$\Gr \Cc$ is a connected bounded below weight filtered cooperad,
\item
there exists $d \geq 0$ such that for each $w$, $(\Gr \Cc)^{(\leq w)}$ is concentrated in homological degree greater than or equal to $-dw$. 
\end{enumerate}
Then its cobar construction $\Om \Cc$ is a curved sa operad such that $\Gr \Om \Cc \cong \Omega \Gr \Cc$ is connected bounded below weight filtered (by the weight filtration induced by the one on $\Cc$). 

Let $\iota : \Cc \to \Om \Cc$ be the operadic curved twisting morphism defined in Example \ref{ex: twisting morph}. 
Let $A$ be a $\Om \Cc$-algebra. 
Then the counit of the bar-cobar adjunction
\[
\Phi_\iota : \Om_\iota \B_\iota A \to A
\]
is a resolution, that is an acyclic fibration.

We assume that $\Rf = \Kf$. 
Assume moreover that:
\begin{enumerate}
\setcounter{enumi}{2}
\item
$\Pc$ be one of the curved sa operad $\cAs$, $\cuAs$, $\cLie$ or $\cCx$, and $\Cc = \Pc^\antishriek$ is the curved altipotent Koszul dual cooperad, and
\item
the filtration $F_\bullet$ on the $\Pc_\infty$-algebra $A$ comes from a graduation $\Gr_\bullet$ (on the module level) so that the induced filtration on $\Pc^\antishriek (A)$ also comes from a graduation. 
\end{enumerate}
In these situations, the filtration on $\Pc$ comes from a graduation $\Gr_\bullet$, $\theta_\Pc \in \Gr_1 \Pc$, and the filtration $F_\bullet$ on $\Cc$ also comes from a graduation $\Gr_\bullet$ (on the module level). 
The underlying $\Sb$-module of $\Cc$ is also weight graded by the number of generators describing the elements in $\Pc^\antishriek$. 

Then $\Om_\iota \B_\iota A$ is cofibrant, that is $\Pc_\infty(0) \to \Om_\iota \B_\iota A$ is a cofibration. 

Finally, in these situations $\Phi_\iota : \Om_\iota \B_\iota A \to A$ is a cofibrant resolution of the $\Pc_\infty$-algebra $A$.

\end{cor}

\begin{proof}
Because $\Gr \Om \Cc \cong \Omega \Gr \Cc$ and $\Gr \Cc$ is a curved coaugmented conilpotent cooperad  that is connected bounded below weight filtered (Assumption (1)), $\Gr \Om \Cc \circ_{\Gr \iota} \Gr\Cc \circ_{\Gr \iota} \Gr \Om\Cc \to \Gr\Om\Cc$ is a quasi-isomorphism by Proposition \ref{prop: augbarres}. 
Moreover both domain and codomain are cofibrant as $\Gr \Om\Cc$-right modules (by Assumption (2)). 
As in Corollary \ref{cor: cof res pi}, for the domain, this follows from \cite[Proposition 14.2.2]{bF09a} with the filtration on $K \coloneqq \Gr \Om\Cc \circ_{\Gr \iota} \Gr \Cc$ (starting at $\lambda \geq 0$) given by
\[
K_\lambda \coloneqq \oplus_{m+dw\leq \lambda} (\Gr \Om \Cc \circ_{\Gr \iota} \Gr \Cc)^{(\leq w)}_m.
\]  
It follows by Proposition \ref{prop: algebraic bar-cobar resolution} that $\Phi_\alpha$ is a graded quasi-isomorphism and a strict surjection.
The four examples satisfy the conditions of the corollary so $(1)$ applies and $\Om_\iota \B_\iota A \to A$ is a resolution. 

Then we have to check that under Assumptions (3) and (4), the hypotheses of Proposition \ref{prop: cofibrant bar-cobar resolution}. 
Remark that in the four examples $d_{\Pc^\antishriek} = 0$. 
We define $\tilde T_{l+1} \coloneqq \Cc^{(l)}(A)$ and
\[
T_{2l-1} \coloneqq \im\left((\id_{\Cc} \circ' d_A)_{|\tilde T_{l}}\right) \subset \tilde T_l
\]
Then let $T_{2l}$ be a direct complement of $T_{2l-1}$ in $\tilde T_{l}$. 
We therefore have obtained a decomposition of $\Cc(A)$ as the direct sum $\oplus_{l\geq 0}T_l$ with $T_0 = \{0\}$. 
For the four examples, the partial decomposition product of $\Cc$, used in the definition of $d_\iota^r$ decomposes as $\Delta_{(1)} = \Delta_{(1)}^0 + \Delta_{(1)}^1$ where $\Delta_{(1)}^0 : \Gr_\bullet \Cc \to \Gr_\bullet (\Cc \circ_{(1)} \Cc)$ and $\Delta_{(1)}^1 : \Gr_\bullet \Cc \to \Gr_{\bullet+1} (\Cc \circ_{(1)} \Cc)$. 
This gives a decomposition $d_\iota^r = d_\iota^{r, 0} + d_\iota^{r, 1}$. 
We fix $\delta_0 \coloneqq d_\iota^{r, 1}$ and $\delta_1 \coloneqq \id_{\Cc} \circ' d_A + d_\iota^{r, 0}$. 
Then $\delta_0$ satisfies $\delta_0 : \Gr_\bullet T_l \to \Gr_{\bullet + 1}(\oplus_{k \geq l+1}T_{k})$ since it increases the weight in $\Cc$ by 1. 

Let us describe the case $\Pc = \cAs$ and $\Cc = \cAs^\antishriek$. We denote by $\theta_A$ the degree $-2$ element $\gamma_A(\curvA) \in A$. Let $V \subset A$ such that $A = \Kf \theta_A \oplus V$. On $\Cc(A)$, we define
\begin{align*}
h\left( \tilde\mu_n^c(\theta_A, a_2, \ldots, a_n)\right) & = \tilde\mu_{n-1}^c(a_2, \ldots, a_n);\\
h\left( \tilde\mu_n^c(a_1, a_2, \ldots, a_n)\right) & = 0 \text{ when } a_1 \in V.
\end{align*}
A direct computation shows that this is a homotopy for $\delta_0$. 
It follows that $\im \delta_0 = \ker \delta_0$. 
We also have $\delta_1-d_\iota^l : T_l \to \Om \Cc(\oplus_{k \leq l}T_k)$ with $\Gr_\bullet \delta_1 : \Gr_\bullet T_l \to \Gr_\bullet \Pc(\oplus_{k < l}T_k)$ by construction (since ${d_1}^2$ and ${d_A}^2$ increases the filtration degree by 1). 
It follows by Proposition \ref{prop: cofibrant bar-cobar resolution} that $\Om_\iota \B_\iota A$ is cofibrant.

The cases $\Pc = \cuAs$, $\cLie$  and $\cCx$ are similar (for $\cLie$ and $\cCx$, we stresses that $\theta_A$ as homological degree $-1$ and that the generators of $\Cc$ are based on the cocommutative corolla $\mu_n^c$ of homological degree 0 which satisfies therefore $\mu_n^c(\ldots, \theta_A, \ldots, \theta_A, \ldots) = 0$).
\end{proof}

\begin{cor}
\label{cor: cof res kappa}
We assume that $\Rf = \Kf$. 
Let $\Pc$ be one of the curved sa operad $\cAs$, $\cuAs$, $\cLie$ or $\cCx$, $\Cc = \Pc^\antishriek$ be the curved altipotent Koszul dual cooperad and $\kappa : \Pc^\antishriek \to \Pc$ be the associated operadic curved twisting morphism (see Definition \ref{def: Koszul}). 

In these situations, the filtration $F_\bullet$ on $\Cc$ comes from a graduation $\Gr_\bullet$ (on the module level) and the underlying $\Sb$-module of $\Cc \cong \Gr_\bullet \Cc$ is also weight graded by the number of generators describing the elements in $\Pc^\antishriek$. 
Let $A$ be a $\Pc$-algebra whose filtration $F_\bullet$ comes from a graduation $\Gr_\bullet$ (on the module level) so that the induced filtration on $\Pc^\antishriek (A)$ also comes from a graduation. 

Then the counit of the bar-cobar adjunction
\[
\Phi_\alpha : \Om_\kappa \B_\kappa A \to A
\]
is a cofibrant resolution, that is an acyclic fibration such that $\Pc(0) \to \Om_\kappa \B_\kappa A$ is a cofibration.
\end{cor}

\begin{proof}
The proof follows the same lines as the two previous ones (Corollaries \ref{cor: cof res pi} and \ref{cor: cof res iota}). 
The fact that $\Gr \Pc \circ_{\Gr \kappa} \Gr\Pc^\antishriek \circ_{\Gr \kappa} \Gr \Pc \to \Gr\Pc$ is a quasi-isomorphism is proved in \cite[Proposition 5.1.6]{HM12}.
\end{proof}

\begin{rem}
The corollaries \ref{cor: cof res pi}, \ref{cor: cof res iota} and \ref{cor: cof res kappa} extend to the context of $\Pc$-algebras in $\ModA$ by adding the assumption that $A$ is a cofibrant as an $\Rf$-module (where $\Rf$ is a $\Kf$-cdga) because Theorem 15.1.A in \cite{bF09a} works with this assumption (and because $\Kf$ is a field).
\end{rem}

\subsection{$\infty$-morphisms between curved algebras}

Given a curved $\Pc$ as in Theorem \ref{thm: Koszul resolution}, we obtain an $\Sb$-cofibrant resolution $\Om \Pc^\antishriek \to \Pc$ and we note $\Pc_\infty \coloneqq \Om \Pc^\antishriek$. 
Let $A$ and $B$ be two $\Pc_\infty$-algebras such that $\Om_\iota \B_\iota A \to A$ is a graded quasi-isomorphism, then we can compare $A$ and $B$ in the homotopy category of $\Pc_\infty$-algebras (instead of the category of $\Pc_\infty$-algebras) by means of the zig-zag
\[
A \xleftarrow{\sim} \Om_\iota \B_\iota A \to B.
\]
By the adjunction presented in Theorem \ref{thm: algebraic bar-cobar adjunction}, the map $ \Om_\iota \B_\iota A \to B$ is characterised by a $\Pc^\antishriek$-coalgebra morphism $\B_\iota A \to \B_\iota B$. 

We therefore define the notion of $\infty$-morphism as follows.

\begin{defn}
Let $A$ and $B$ be two $\Pc_{\infty}$-algebras. An \emph{$\infty$-morphism $A \rightsquigarrow B$} is a $\Pc^{\antishriek}$-coalgebra map
\[
F: \B_\iota A \to \B_\iota B.
\]
By Theorem \ref{thm: algebraic bar-cobar adjunction}, this is equivalent to an algebraic twisting morphism $\B_\iota A \to B$ with respect to $\iota$. 
We denote by $\infty \text{-}\mor(A, B)$ the set of $\infty$-morphisms from $A$ to $B$.
\end{defn}

\section{Cohomology of curved algebras}

André--Quillen cohomology is a cohomology theory first developed for commutative algebras \cite{mA74, dQ70} and then extended to algebras over an operad \cite{GH00, jM11}. 
In this section, we developed the theory for algebras over a curved operad. 
We make some general computations that we use in the next section to compute the André--Quillen cohomology in the context of $\cCxi$-algebras.

\subsection{André--Quillen cohomology}

We adapt the definitions given in \cite{GH00} and \cite{jM11} to the curved setting.

Let $(\Pc, \gamma_\Pc, \eta_\Pc, \varepsilon_\Pc, d_\Pc, \theta_\Pc)$ be a semi-augmented curved operad. 
Let $(A, \gamma_A, \eta_A, d_A)$ be a $\Pc$-algebra (whose underlying gr-dg module is a $\Rf$-module).

\subsubsection{Module over a $\Pc$-algebra}

Let $(M, d_M)$ be a gr-dg $\Rf$-module. 
We define $\Pc(A, M)$ to be the submodule of $\Pc(A\oplus M)$ made of terms linear in $M$, that is
\[
\Pc(A, M) \coloneqq
\bigoplus_n \Pc(n) \otimes_{\Sb_n} \left( A \hotimes \cdots \hotimes \underbrace{M}_{j\text{th position}} \hotimes \cdots \hotimes A\right).
\]
(It is also well-defined when $\Pc$ is a gr-dg $\Sb$-module without a curved operad structure.)

\begin{defn}
\label{def: A-mod over P}
An \emph{$A$-module} $(M, \gamma_M, \iota_M, d_M)$, or \emph{$A$-module over $\Pc$}, is a gr-dg $\Rf$-module $(M, d_M)$ endowed with two maps
\[
\gamma_M : \Pc(A, M) \to M \text{ and } \iota_M : M \to \Pc(A, M)
\]
such that the following diagrams of associativity and unitarity commute
\[
\xymatrix@R=2pt@C=16pt{(\Pc \circ \Pc) (A, M) \cong \Pc (\Pc A, \Pc (A, M)) \ar[rr]^{\hspace{54pt} \id_{\Pc} (\gamma_A, \gamma_M)} \ar[dd]_{\gamma_\Pc(\id_A, \id_M)} & & \Pc (A, M) \ar[dd]^{\gamma_M}\\
& & & \textrm{and} \\
\Pc (A, M) \ar[rr]_{\hspace{14pt} \gamma_M} & & M}
\xymatrix@R=16pt{M \ar[r]^{\hspace{-12pt} \iota_M} \ar[dr]_= & \Pc (A, M) \ar[d]^{\gamma_M}\\
& M.}
\]
We moreover require the following equation to hold: ${d_M}^2 = \gamma_M(\theta_\Pc \otimes \id_M)$. 

A \emph{morphism of $A$-modules} $f : (M, \gamma_M, \iota_M, d_M) \to (N, \gamma_N, \iota_N, d_N)$ is an $\Rf$-modules map $f : (M, d_M) \to (N, d_N)$ such that the following diagrams commutes
\[\begin{tikzcd}
	{\Pc(A, M)} & M \\
	{\Pc(A, N)} & N
	\arrow["{\id_\Pc (\id_A, f)}"', from=1-1, to=2-1]
	\arrow["{\gamma_M}", from=1-1, to=1-2]
	\arrow["{\gamma_N}"', from=2-1, to=2-2]
	\arrow["f", from=1-2, to=2-2]
\end{tikzcd} \quad\text{ and }\quad \begin{tikzcd}
	M & {\Pc(A, M)} \\
	N & {\Pc(A, N).}
	\arrow["{\id_\Pc (\id_A, f)}", from=1-2, to=2-2]
	\arrow["{\iota_M}", from=1-1, to=1-2]
	\arrow["f"', from=1-1, to=2-1]
	\arrow["{\iota_N}"', from=2-1, to=2-2]
\end{tikzcd}\]

We denote by $\ModAP$ (resp. by $\ModAPz$) the category of $A$-modules over the operad $\Pc$ in which morphisms are morphisms of $A$-modules (resp. in which morphisms are degree 0 gr-dg morphisms of $A$-modules). 
\end{defn}

\begin{example}
Let $\Kf =  \Rf = \Rb$, $\Pc= \cCx$ and $A$ be a $\cCx$-algebra. Then $A$ is a Lie algebra (with a bracket of cohomological degree  $-1$) and a $\Cb$-module such that the Lie bracket is $\Cb$-linear. An $A$-module is a $\Cb$-module endowed with a $\Cb$-linear action of the Lie algebra $A$. 
\end{example}

Let's now construct what will be the free $A$-module on a gr-dg $\Rf$-module $(N, d_N)$. We define the coequalizer in $\grModR$
\[
A\hotimes^{\Pc} N \coloneqq \coeq\left(\xymatrix@C=50pt{\Pc (\Pc (A), N) \ar@<0.5ex>[r]^{\widetilde{\gamma_\Pc}} \ar@<-0.5ex>[r]_{\overline{\id_\Pc(\gamma_A, \id_N)}} & \overline{\Pc (A, N)}}\right),
\]
where
\[
\overline{\Pc(A, N)} \coloneqq \Pc(A, N)/\left(\Pc(A, N) \cap \left(\im\left(\eta_\Pc \otimes ({d_{A\oplus N}}^2) - \theta_\Pc \otimes \id_{A\oplus N}\right)\right)\right)
\]
and
\[
\widetilde{\gamma_\Pc} : \Pc (\Pc (A), N) \hookrightarrow \Pc( \Pc A, \Pc(A, N)) \xrightarrow{\gamma_\Pc (\id_{A}, \id_{N})} \Pc (A, N) \twoheadrightarrow \overline{\Pc(A, N)}.
\]
With respect to the dg situation, the quotient $\overline{\Pc(A, N)}$ appears to get the equation ${d_M}^2 = \gamma_M(\theta_\Pc \otimes \id_M)$ that an $A$-module $M=A\hotimes^{\Pc} N$ have to satisfy. 
The predifferentials $d_{A\hotimes^\Pc N}$  is induced by the predifferentials on $\Pc$, $A$ and $N$.

As in the classical case (see for example Lemma 12.3.3 in \cite{LV12}), the $A$-module structure is obtained from the map
\[
\Pc (A, \Pc(A, N)) \rightarrowtail \Pc ( \Pc A, \Pc(A, N)) \xrightarrow{\gamma_\Pc (\id_A, \id_N)} \Pc (A, N)
\]
and the map
\[
\Pc(A, N) \xrightarrow{\cong} \Pc (A, I\circ N) \xrightarrow{\id_\Pc (\id_A, \eta_\Pc \circ \id_N)} \Pc (A, \Pc(A, N)) 
\]
which pass to the quotient $\Pc(A, N) \twoheadrightarrow A \hotimes^\Pc N$ (by  associativity of $\gamma_\Pc$ and $\gamma_A$, and by the fact that $d_\Pc$ and $d_A$ are  derivations with respect to the composition structures). 

We denote by $(\grModR)^0$ the category of gr-dg modules with morphisms the degree 0 gr-dg morphisms.

\begin{prop}
\label{prop: free AP-mod}
The forgetful functor $U : \ModAPz \to (\grModR)^0$ has a left adjoint, denoted by
\[
(N, d_N) \mapsto (A\hotimes^{\Pc} N, d_{A\hotimes^\Pc N}).
\]
That is, we have an isomorphism of $\Rf$-modules
\[
\Hom_{\ModAPz}\left( A \hotimes^{\Pc} N, M\right) \cong \Hom_{(\grModR)^0} \left( N, U M\right),
\]
for all $N \in (\grModR)^0$ and $M \in \ModAPz$.
\end{prop}

\begin{proof}
To any morphism $f : A\hotimes^\Pc N \to M$, we associate its restriction
\[
N \cong I\circ N \xrightarrow{\eta_\Pc (\id_A, \id_N)} \Pc(A, N) \twoheadrightarrow A \hotimes^\Pc N \xrightarrow{f} M.
\]
In the other direction, given $g : N \to U(M)$, we define the map
\[
\Pc(A, N) \xrightarrow{\id_\Pc(\id_A, g)} \Pc(A, M) \twoheadrightarrow A\hotimes^\Pc M \to M.
\]
This map factors through the quotient map $\Pc(A, N) \twoheadrightarrow A\hotimes^\Pc N$ (because $g$ commutes with $d_N$) to give the wanted map of $A$-modules. 

The two processes are natural linear isomorphisms between the spaces of modules maps. Provided $f$ (resp. $g$) commutes with the predifferential, the same is true for the restriction (resp. extension) since $\gamma_\Pc$, $\gamma_A$ and $\gamma_M$ commute with the predifferentials.
\end{proof}

We also define a quasi-free $A$-module. 
Let $N$ be an $\Rf$-module and $\tau : N \to \Pc(A, N)$ be a degree $-1$ map which induces a gr-differential still denoted $\tau$ on $\Pc(A, N)$. 
Similarly as above, we define the coequalizer in $\grModR$
\[
A\hotimes^{\Pc}_\tau N \coloneqq \coeq\left(\xymatrix@C=50pt{\Pc (\Pc (A), N) \ar@<0.5ex>[r]^{\widetilde{\gamma_\Pc}} \ar@<-0.5ex>[r]_{\overline{\id_\Pc(\gamma_A, \id_N)}} & \overline{\Pc (A, N)}^\tau}\right),
\]
where
\[
\overline{\Pc(A, N)}^\tau \coloneqq \Pc(A, N)/\left(\Pc(A, N) \cap \left(\im\left(\eta_\Pc \otimes (d_A + \tau)^2 - \theta_\Pc \otimes \id_{A\oplus N}\right)\right)\right).
\]
It is endowed with the gr-differential $d_{A\hotimes^{\Pc}_\tau N}$ depending on $d_A$, $d_\Pc$ and $\tau$.

\begin{prop}
\label{prop: quasi-free AP-mod}
For all $N \in \grModR$ and $M \in \ModAP$, we have an isomorphism of $\Rf$-modules
\[
\Hom_{\ModAPz}\left( A \hotimes^{\Pc}_\tau N, M\right) \cong \ker\left( \partial_\tau : \Hom_{\grModR}^0 \left( N, U M\right) \to \Hom_{\grModR}^{1} \left( N, U M\right) \right),
\]
where $\partial_\tau(f) \coloneqq d_M \cdot f - (-1)^{|f|}\gamma_M \cdot (\id_\Pc(\id_A, f)) \cdot \tau$.
\end{prop}

\begin{proof}
The proof is similar to the proof of Proposition \ref{prop: free AP-mod}. The only differences come from the gr-differential $\tau$.
\end{proof}
 
Given two $A$-modules over a curved operad $\Pc$, $(N, d_N, \gamma_N)$ and $(M, d_M, \gamma_M)$, the graded space of maps $\Hom_{\ModAP}(M, N)$ is naturally endowed with a predifferential $\partial$ defined by $\partial(f) \coloneqq d_N \cdot f - (-1)^{|f|}f \cdot d_M$. 
Gr-dg morphisms coincide with maps in the kernel of $\partial$. 

\begin{lem}
\label{lem: partial differential for Amodules}
Let $(N, d_N, \gamma_N)$ and $(M, d_M, \gamma_M)$ be two $A$-modules over a curved operad $\Pc$.
The predifferential $\partial$ on the graded $\Rf$-module $\Hom_{\ModAP}(N, M)$ is in fact a differential, that is $\partial^2 = 0$.
\end{lem}

\begin{proof}
Since $N$ and $M$ are $A$-modules, we have for any $f : N\to M$ of $A$-modules
\[
\partial^2(f) = {d_N}^2 \cdot f - f\cdot {d_M}^2 = \gamma_N(\theta_\Pc \otimes f)-f\cdot \gamma_M(\theta_\Pc \otimes \id_M) = 0.
\]
\end{proof}

\begin{rem}
This lemma is not true when $(N, d_N)$ and $(M, d_M)$ are only assumed to be gr-dg $\Rf$-modules.
\end{rem}

\begin{lem}
\label{lem: U creates limits and colimits}
 The forgetful functor $\ModAPz \to (\grModR)^0$ creates all limits and colimits, and $\ModAPz$ is an additive subcategory of $(\grModR)^0$.
\end{lem}

\begin{proof}
Because the direct sum and the tensor product preserve colimits in $(\grModR)^0$, it  follows that $M \mapsto \Pc(A, M)$ preserves colimits in $(\grModR)^0$. 
If $M : I \to \ModAPz$ is any diagram, the colimit in $(\grModR)^0$, $\colim_I M$, acquires a map
\[
\Pc(A, \colim_I M) \cong \colim_I \Pc(A, M) \to \colim_I F_\Pc(A, M) \to \colim_I M.
\]
Because ${d_A}^2 = \gamma_A(\theta_\Pc, -)$ and ${d_{M(i)}}^2 = \gamma_{M(i)}(\theta_\Pc, -)$ for all $i \in I$, the same equation is satisfied by $\colim_I M$. 
This endows $\colim_I M$ with an $A$-module structure over $\Pc$. With this structure, $\colim_I M$ becomes the colimit in $\ModAPz$. 
The case of limits can be treated in a similar way since the limit in $(\grModR)^0$ is a sub-$\Rf$-module of the direct product of the modules involved and $M \mapsto \Pc(A, M)$ preserves products. 
It follows that the forgetful functor creates all limits and colimits. 

Then, the category $\ModAPz$ is additive since the direct sum of two objects is made in $(\grModR)^0$.
\end{proof}

\begin{lem}
\label{lem: moduleAP and monad}
The category $\ModAPz$ is isomorphic to the category of algebras over the monad $U(A\hotimes^{\Pc} -)$. 
\end{lem}

\begin{proof}
This follows from Lemma \ref{lem: U creates limits and colimits} and from Beck's Theorem (see \cite[VI, \S 7]{ML98}).
\end{proof}

Let $A$ be a curved associative algebra. By making explicit the notion of $A$-modules over $\cAs$, we get that it coincides, for an $A$-module $M$, with the following data:
\[
\gamma_l : A \hotimes M \to M \text{ and } \gamma_r : M \hotimes A \to M,
\]
satisfying the usual associativity relations and such that
\[
{d_M}^2 = \gamma_l(\theta_A\otimes \id_M) - \gamma_r(\id_M \otimes \theta_A) = [\theta_A, -].
\]
It is natural to extract from this definition the notions left-$A$-module and right-$A$-module. The gr-dg module $M$ is a left-$A$-module if there exists a map $\gamma_l : A\hotimes M \to M$ satisfying the usual associativity relation such that ${d_M}^2 = \gamma_l(\theta_A \otimes \id_M) = \theta_A \cdot -$. 
The notion of right-$A$-module is similar but a minus sign appears in the relation that $d_M$ satisfies: ${d_M}^2 = -( -\cdot \theta_A)$. 
These notions already appear in the literature, see for example \cite[Section 4 of Chapter 5]{PP05}. \\

We now define the enveloping algebra of a $\Pc$-algebra $A$. It is a curved associative algebra. 

The coequaliser $U_\Pc(A) \coloneqq \coeq\left(\xymatrix@C=50pt{\Pc (\Pc (A), \Rf) \ar@<0.5ex>[r]^{\widetilde{\gamma_\Pc}} \ar@<-0.5ex>[r]_{\id_\Pc(\gamma_A, \id_\Rf)} & \Pc (A, \Rf)}\right)$ is endowed with a filtration coming from the one on $\Pc$ and $A$, and a multiplication coming from the map
\[
\Pc (A, \Rf) \otimes \Pc (A, \Rf) \cong \Pc (A, \Pc (A, \Rf)) \to \Pc (A, \Rf).
\]
This multiplication is associative since $\gamma_\Pc$ and $\gamma_A$ are. It has a unit coming from the map $\Rf \to \Pc(A, \Rf)$ and a predifferential $d_{U_\Pc(A)}$ induced by the predifferentials $d_\Pc$ and $d_A$ and the differential $d_\Rf$. 
The element $\theta_{U_\Pc(A)} \coloneqq \overline{\theta_\Pc \otimes 1_\Rf} \in U_\Pc(A)$ has homological degree $-2$ and filtration degree 1 (as $\theta_\Pc$) and we have ${d_{U_\Pc(A)}}^2 = [\theta_{U_\Pc(A)}, -]$ because $\Pc$ and $A$ are curved.

\begin{defn}
Let $A$ be a $\Pc$-algebra. The \emph{enveloping algebra} of the $\Pc$-algebra $A$ is the curved associative algebra 
\[
U_\Pc(A) \coloneqq \left(\coeq\left(\xymatrix@C=50pt{\Pc (\Pc (A), \Rf) \ar@<0.5ex>[r]^{\widetilde{\gamma_\Pc}} \ar@<-0.5ex>[r]_{\id_\Pc(\gamma_A, \id_\Rf)} & \Pc (A, \Rf)}\right), d_{U_\Pc(A)}, \theta_{U_\Pc(A)}\right).
\]
\end{defn}

\begin{prop}
\label{prop: APmod and left mod over the enveloping algebra}
The category of $A$-modules over $\Pc$ is isomorphic to the category of left-modules over the (curved associative) enveloping algebra $U_\Pc(A)$.
\end{prop}

\begin{proof}
We have
\[
U_\Pc(A) \hotimes M \cong \coeq\left(\xymatrix@C=50pt{\Pc (\Pc (A), M) \ar@<0.5ex>[r]^{\widetilde{\gamma_\Pc}} \ar@<-0.5ex>[r]_{\id_\Pc(\gamma_A, \id_M)} & \Pc (A, M)}\right) \eqqcolon \coeq
\]
therefore a left-$U_\Pc(A)$-module structure $\gamma_l$ gives a map $\gamma_M : \Pc(A, M) \twoheadrightarrow \coeq \xrightarrow{\gamma_l} M$ which endows $M$ with a structure of $A$-module over $\Pc$ (the associativity diagram is satisfied by the fact that the map factors through $\coeq$). 
The equality ${d_M}^2 = \gamma_l(\theta_{U_\Pc(A)}, -)$ implies the equality ${d_M}^2 = \gamma_M(\theta_\Pc, -)$.

Conversely, a structure $\gamma_M$ of $A$-module over $\Pc$ is precisely a map which factors through $\coeq$ to give a map $\gamma_l$, and which satisfies ${d_M}^2 = \gamma_l(\theta_{U_\Pc(A)}, -)$ (because ${d_M}^2 = \gamma_M(\theta_\Pc, -)$ and associativity of $\gamma_A$). 
This endows $M$ with a left-$U_\Pc(A)$-module structure.
\end{proof}

\begin{rem}
Compared with the situation when there is no curvature, we emphasise that in the curved setting the free $A$-module on $\Rf$ is a quotient of the enveloping algebra $U_\Pc(A)$.
\end{rem}

\begin{prop}
For any curved operad $\Pc$, the enveloping algebra construction provides a functor
\[
U_\Pc : \Pc\textsf{-Alg} \to \cuAs\textsf{-Alg},
\]
where $\cuAs$ is the operad encoding curved unital associative algebras.
\end{prop}

\begin{proof}
Let $B \to A$ be a morphism of $\Pc$-algebras. The map $U_\Pc(f) : U_\Pc(B) \to U_\Pc(A)$ induced by $\Pc(B, \Rf) \xrightarrow{\id_\Pc(f, \id_{\Rf})} \Pc(A, \Rf)$ is a map of unital associative algebras. 
This map commutes with the predifferentials and preserves the curvature by construction.
\end{proof}

Let $f : B \to A$ be a morphism of $\Pc$-algebras and $(N, \gamma_N) \in \ModBP$. 
The composite
\begin{multline*}
\Pc (A, \Pc (B, N)) \xrightarrow{\id_{\Pc} (\id_{A}, \id_{\Pc} (f, \id_{N}))} \Pc (A, \Pc (A, N)) \rightarrowtail (\Pc \hcirc \Pc)(A, N)\\
\xrightarrow{\gamma_\Pc (\id_{A}, \id_{N})} \Pc (A, N) \twoheadrightarrow A\hotimes^{\Pc} N
\end{multline*}
factors through the quotient $A\hotimes^{\Pc} (B\hotimes^{\Pc} N)$ to produce a map $\rho_l : A\hotimes^{\Pc} (B\hotimes^{\Pc} N) \to A\hotimes^\Pc N$. 
Similarly, the composite
\[
\Pc (A, \Pc (B, N)) \xrightarrow{\id_{\Pc} (\id_{A}, \gamma_{N})} \Pc (A, N) \twoheadrightarrow A\hotimes^{\Pc} N,
\]
induces a second map $\rho_r : A\hotimes^{\Pc} (B\hotimes^{\Pc} N) \to A\hotimes^\Pc N$. 
Using these two maps, we define the $A$-module $A\hotimes_{B}^{\Pc} N$ as the following coequalizer
\[
A\hotimes_{B}^{\Pc} N \coloneqq \coeq\left(\xymatrix{A\hotimes^{\Pc} (B\hotimes^{\Pc} N) \ar@<0.5ex>[r]^{\quad \rho_l} \ar@<-0.5ex>[r]_{\quad \rho_r} & A\hotimes^{\Pc} N}\right).
\]

\begin{prop}
\label{prop: extension and restriction of scalars}
Let $f : B \to A$ be a morphism of $\Pc$-algebras. The functor of restriction of scalars
\[
f^* : \ModAP \to \ModBP
\]
admits a left adjoint
\[
f_! \coloneqq A\hotimes^\Pc_B - : \ModBP \to \ModAP.
\]
We mean that we have natural isomorphisms of gr-dg $\Rf$-modules
\[
\left(\Hom_{\ModAP}\left( f_!(N), M\right), \partial\right) \cong \left(\Hom_{\ModBP}\left( N, f^*(M)\right), \partial\right),
\]
for any $A$-module $M$ and any $B$-module $N$. 

The same functors restrict to the kernels of $\partial$ to provide isomorphisms of $\Rf$-modules
\[
\Hom_{\ModAPz}\left( f_!(N), M\right) \cong \Hom_{\ModBPz}\left( N, f^*(M)\right).
\]
Moreover, for any gr-dg $\Rf$-module $L$ and map $\tau : L \to \Pc(B, L)$, we have the isomorphism of gr-dg $\Rf$-modules
\[
A\hotimes^\Pc_B(B\hotimes^\Pc_\tau L) \cong A\hotimes^\Pc_{\tau'} L,
\]
where $\tau'$ is the composite $L \xrightarrow{\tau} \Pc(B, L) \xrightarrow{\id_\Pc(f, \id_L)} \Pc(A, L)$.
\end{prop}

\begin{proof}
The proof is analog to the proof of Proposition \ref{prop: free AP-mod}. 
To any morphism $h : A\hotimes^\Pc_B N \to M$, we associate its restriction
\[
N \to A \hotimes^\Pc N \twoheadrightarrow A\hotimes^\Pc_B N \xrightarrow{h} M.
\]
In the other direction, given $g : N \to f^*(M)$, we define the map
\[
A\hotimes^\Pc N \xrightarrow{g} A\hotimes^\Pc M \to M.
\]
This maps factors through the quotient map $A\hotimes^\Pc N \twoheadrightarrow A \hotimes^\Pc_B N$ since $g$ is a morphism of $B$-modules to give the wanted map of $A$-modules.

The two processes are natural (graded) linear isomorphisms between the spaces of (graded) modules maps. Moreover, they commute with the differentials $\partial$ (see Lemma \ref{lem: partial differential for Amodules} for the fact that they are differentials) since we consider morphisms compatible with the structures of modules over $\Pc$ (and $\gamma_\Pc$, $\gamma_A$ and $\gamma_M$ commute with the predifferentials).  
Finally, when $h$ (resp. $g$) commutes with the predifferentials, the same is true for the restriction (resp. the extension).

The last isomorphism follows from the adjunctions just shown and from Proposition \ref{prop: free AP-mod} by the Yoneda Lemma.
\end{proof}

\subsubsection{Quasi-free $\Pc$-algebra}

Let $(M, d_M)$ be a gr-dg $\Rf$-module. 
We recall from \cite[Proposition C.35]{BMDC20} that the free $\Pc$-algebra on a gr-dg module $M$ is given by
\[
F_\Pc(M) = \Pc(M)/\left(\im\left(\eta_\Pc \otimes ({d_{M}}^2) - \theta_\Pc \otimes \id_{M}\right)\right)
\]
where $\left(\im\left(\eta_\Pc \otimes ({d_{M}}^2) - \theta \otimes \id_{M}\right)\right)$ is the ideal of $\Pc(M)$ generated by the image of $\eta_\Pc \otimes ({d_{M}}^2) - \theta \otimes \id_{M}$. 

The bar-cobar resolutions are not free $\Pc$-algebras since the predifferential in theses cases isn't internal to the generators. 
We define the notion of quasi-free $\Pc$-algebras in order to include these examples.

\begin{defn}
Let $\Pc$ be a curved operad.
We say that a $\Pc$-algebra $S$ is \emph{quasi-free} when there exists a complete filtered $\Rf$-module $V$ and a predifferential $d : \Pc(V) \to \Pc(V)$ such that
\[
S \cong \left(\Pc(V), \bar d\right)/\left( \im \left( d^2 - \theta_\Pc \otimes \id_V\right)_{|V}\right),
\]
where $\bar d$ is induced by $d$.
\end{defn}

We remark that even if we do not assume that $d$ is a gr-differential in the above definition
, a quasi-free $\Pc$-algebra is nevertheless a gr-dg $\Rf$-module. 
The notion of quasi-free $\Pc$-algebra is different from the notion of free $\Pc$-algebra 
since the predifferential $d$ isn't a priori induced by a map $V \to V$.

\subsubsection{Derivation and Kähler differential forms}

We call a \emph{$\Pc$-algebra over $A$} a $\Pc$-algebra $B$ endowed with a map of $\Pc$-algebras $B \xrightarrow{f} A$.

\begin{defn}
Let $B$ be a $\Pc$-algebra over $A$ and $M$ be an $A$-module. An \emph{$A$-derivation from $B$ to $M$} is a linear map $d : B \to M$ such that the following diagram commutes
\[
\xymatrix@C=2cm@R=0.5cm{\Pc (B) \ar[d]_{\gamma_{B}} \ar[r]^{\hspace{-1cm} \id_{\Pc} \circ' d} & \Pc (B, M) \ar[r]^{\id_{\Pc} \circ (f, \id_{M})} & \Pc (A, M) \ar[d]^{\gamma_{M}}\\
B \ar[rr]_d && M,}
\]
where the infinitesimal composite of morphisms $\circ'$ is defined in \cite[Section 6.1.3]{LV12}. We denote by $\Der_A(B, M)$ (resp. by $\Der_A^0(B, M)$) the set of $A$-derivations from $B$ to $M$ (resp. which moreover have degree 0 and commute with the predifferentials $d_B$ and $d_M$).
\end{defn}

\begin{lem}
Let $B$ be a $\Pc$-algebra over $A$ and $M$ be an $A$-module. 
The predifferential $\partial$ on the graded $\Rf$-module $\Der_A(B, M)$, defined by
\[
\partial(D) = d_M \cdot D - (-1)^{|D|}D \cdot d_B
\]
for $D \in \Der_A(B, M)$, is in fact a differential, that is $\partial^2 = 0$.
\end{lem}

\begin{proof}
We compute
\begin{align*}
\partial^2(D) & = {d_M}^2 \cdot D - D \cdot {d_B}^2\\
& = \gamma_M(\theta_\Pc, D(-)) - D \cdot \gamma_B(\theta_\Pc, -) = 0
\end{align*}
since $D$ is an $A$-derivation from $B$ to $M$.
\end{proof}

When $\Mc = \{ \Mc(0), \Mc(1), \ldots\}$ is a gr-dg $\Sb$-module, we denote by $\Mc_{>0}$ the gr-dg $\Sb$-module $\{ \{0\}, \Mc(1), \ldots\}$.

\begin{prop}
\label{prop: product for P-alg}
The category of $\Pc$-algebras admits finite products.
\begin{enumerate}
\item
The terminal object is given by $(\{0\}, \gamma = 0, d=0)$.
\item
The product of two $\Pc$-algebras $(A, \gamma_A)$ and $(B, \gamma_B)$ is given by $A \wedge B \coloneqq (A \oplus B, \gamma)$ where $\gamma$ is defined by the composite
\[
\Pc(A\wedge B) \twoheadrightarrow \Pc(0) \oplus \Pc_{>0}(A) \oplus \Pc_{>0}(B) \xrightarrow{(\gamma_A + \gamma_{B}) + \gamma_A + \gamma_{B}} A \oplus B.
\]
The maps from $A \wedge B$ to $A$ and $B$ are the projections.
\end{enumerate}
Similarly, given a $\Pc$-algebra $A$, the category $\PalgA$ of $\Pc$-algebras over $A$ admits finite products.
\begin{enumerate}
\item
The terminal object is the $\Pc$-algebra $A$.
\item
The product of two $\Pc$-algebras $(B, \gamma_B)$ and $(B', \gamma_B')$ over $A$ is given by $B \wedge_A B' \coloneqq (B \oplus_A B', \gamma)$ where $\gamma$ is defined by the composite
\[
\Pc(B\wedge_A B') \twoheadrightarrow \Pc(0) \oplus \Pc_{>0}(B) \oplus_{\Pc_{>0}(A)} \Pc_{>0}(B') \xrightarrow{(\gamma_B +_A \gamma_{B'}) + \gamma_B + \gamma_{B'}} B \oplus_A B'.
\]
\end{enumerate}
\end{prop}

\begin{proof}
We first prove the universal property satisfied by $A \wedge B$. Let $f_A : C \to A$ and $f_B : C \to B$ be two $\Pc$-algebra maps. Using the fact that $A\oplus B$ is the product in modules, we get that the requested map $C\to A\wedge B$ is necessarily $f_A + f_B$. It is then a direct check to verify that $f_A + f_B$ is a $\Pc$-algebra map. 

We can show similarly that $B \wedge_A B'$ satisfies the universal property of the product.
\end{proof}

\begin{defn}
Let $A$ be a $\Pc$-algebra and let $M$ be an $A$-module. The \emph{abelian extension of $A$ by $M$}, denoted by $A\ltimes M$, is the $\Pc$-algebra over $A$ whose underlying $\Rf$-module is $A\oplus M$ and whose $\Pc$-algebra structure is given by
\[
\Pc(A\oplus M) \twoheadrightarrow \Pc(A) \oplus \Pc(A, M) \xrightarrow{\gamma_A + \gamma_M} A \oplus M.
\]
The morphism $A \ltimes M$ is the projection on the first summand.
\end{defn}

\begin{prop}
\label{prop: derivation and abelian extension}
Let $\Pc$ be a curved operad and $A$ be a $\Pc$-algebra. There are $\Rf$-linear isomorphisms
\[
\Der_A^0(B, M) \cong \Hom_{\PalgA}\left( B, A \ltimes M\right),
\]
for any $\Pc$-algebra $B$ over $A$ and any $A$-module $M$.
\end{prop}

\begin{proof}
Let $f : B \to A$ be the structural morphism of $B \in \PalgA$. Any morphism $B \to A \ltimes M$ in $\PalgA$ is of the form
\[
\xymatrix@R=4pt{B \ar[rr]^{f+d} \ar[rd]_f && A\ltimes M \ar@{->>}[dl]\\ & A, &}
\]
where $d : B \to M$. 
Then the map $f+d$ is a morphism of $\Pc$-algebras if and only if $d$ is an $A$-derivation.
\end{proof}

We denote by $\AbPA$ the category of abelian group objects in the category $(\PalgA, \wedge_A)$ of $\Pc$-algebras over $A$.

\begin{prop}
\label{prop: abelian group object and module}
Let $A$ be a $\Pc$-algebra. There is a canonical equivalence between the category of $A$-modules and the category of abelian group objects in $\PalgA$
\[
\left\{\begin{array}{ccc}
\AbPA & \simeq & \ModAPz,\\
(B \xrightarrow{f} A) & \mapsto & \ker(f).
\end{array}\right.
\]
\end{prop}

\begin{proof}
As we will see, the standard proof works in the setting of curved operads. 
Let $B \xrightarrow{f} A$ be an abelian group object in $\PalgA$, that is there is a $\Pc$-algebra map $\eta : A \to B$ such that $f \cdot \eta = \id_A$ and an abelian multiplication $m : B \wedge_A B \to B$. 
We define on $\ker (f)$ an $A$-module structure $\gamma_{\ker}$ by the composite
\[
\Pc(A, \ker (f)) \xrightarrow{\id_\Pc \circ (f, \mathrm{incl})} \Pc(B, B) \twoheadrightarrow \Pc(B) \xrightarrow{\gamma_B} B.
\]
Using the fact that $f$ is a $\Pc$-algebra map, we obtain $\im(\gamma_{\ker}) \subset \ker (f)$. Moreover, since $B$ is a $\Pc$-algebra and $\eta$ is a $\Pc$-algebra map, we get the commutativity of the associativity and unitarity diagrams. 
For $b \in \ker(f)$, we also have ${d_{\ker (f)}}^2(b) = \gamma_B(\theta_B, b) = \gamma_B(f(\theta_A), b) = \gamma_{\ker(f)}(\theta_A, b)$ as required.

We have the $\Rf$-module isomorphism $B \wedge_A B \simeq A\oplus \ker(f) \oplus \ker(f)$. By the fact that $m$ is a $\Pc$-algebra map, we obtain that $m$ is the (abelian) $\Rf$-module sum structure and that the $\Pc$-algebra $\gamma_B$ on $B \simeq A\oplus \ker(f)$ vanishes on the module $\hoplus_{k \geq 2} \Pc(k) \hotimes_{\Sb_k} (\ker(f)^{\hotimes 2} \hotimes B^{\hotimes k-2})$. 
This says that $B$ is an abelian extension of $A$ by $\ker(f)$. Conversely, every abelian extension provides an abelian group object and the two processes are opposite to each other.
\end{proof}

As in the dg context, we define now the $A$-module of Kähler differential forms in order to represent the module of $A$-derivations. 
It is given by the following coequalizer
\begin{equation}
\label{eq: Kahler diff forms}
\xymatrix{A\hotimes^{\Pc} \Pc(A) \ar@<0.5ex>[r]^{\quad A\hotimes^\Pc \gamma_A} \ar@<-0.5ex>[r]_{\quad \tilde \gamma_{\Pc, (1)}} & A\hotimes^{\Pc} A \ar@{->>}[r] & \Omega_\Pc A,}
\end{equation}
where $\tilde \gamma_{\Pc, (1)}$ is the map induced by the composite
\[
\Pc (A, \Pc (A)) \to \Pc (A, \Pc (A, A)) \rightarrowtail (\Pc \hcirc \Pc)(A, A) \xrightarrow{\gamma_\Pc (\id_{A}, \id_{A})} \Pc (A, A) \twoheadrightarrow A\hotimes^{\Pc}A
\]
since it factors through $A\hotimes^{\Pc}\Pc (A)$.

\begin{prop}
\label{prop: derivation and Kahler differential form}
Let $\Pc$ be any curved operad and $A$ be any $\Pc$-algebra. 
There are natural gr-dg $\Rf$-linear isomorphisms
\[
\left(\Der_{A}(B, M), \partial\right) \cong \left(\Hom_{\ModBP}(\Omega_{\Pc}B, f^*(M)), \partial\right),
\]
for any $\Pc$-algebra $B$ over $A$ and any $A$-module $M$. 
The isomorphisms restrict to the kernels of $\partial$ to provide isomorphisms of $\Rf$-modules
\[
\Der_{A}^0(B, M) \cong \Hom_{\ModBPz}(\Omega_{\Pc}B, f^*(M)).
\]
Moreover, when $B$ is a quasi-free $\Pc$-algebra on $(V, \tau : \Pc(V) \to \Pc(V))$, we get $\Omega_{\Pc}B \cong B\otimes^{\Pc}_\tau V$ as gr-dg $\Rf$-modules.
\end{prop}

\begin{proof}
Let $d : B \to M$ be an $A$-derivation. Then the composite
\[
B\hotimes^\Pc B \xrightarrow{B\hotimes^\Pc d} B\hotimes^\Pc M \to M
\]
factors through the quotient $B\hotimes^\Pc B \twoheadrightarrow \Omega_\Pc B$ since $d$ is an $A$-derivation $B\to M$ and since the last map is induced by the $B$-module structure on $f^*(M)$. The resulting map is a map of $B$-modules since the above composite is. 

In the other way round, given a $B$-module map $\rho : \Omega_\Pc B \to f^*(M)$, the restriction $B \cong I\hcirc B \xrightarrow{\eta_\Pc \circ \id_B} \Omega_\Pc B \xrightarrow{\rho} f^*(M)$ gives a map $B \to M$.
Using the definition of $\Omega_\Pc B$, the module structure on $f^*(M)$ and the fact that $\rho$ is a map of $B$-modules, we get that its restriction is an $A$-derivation.

The two processes are natural (graded) linear isomorphisms between the spaces of (graded) modules maps. As in Proposition \ref{prop: extension and restriction of scalars}, they are compatible with the differentials $\partial$.

When
\[
B = \Pc (V)/ \left( \im (\tau^2 - \theta_\Pc \otimes \id_V)_{|V} \right)
\]
is a quasi-free $\Pc$-algebra on $(V, \tau)$, we can compute $\Omega_\Pc B$ by computing the coequaliser \eqref{eq: Kahler diff forms}. It says that elements in $\Omega_\Pc B$ are elements $\mu(b_1, \ldots, d b_j, \ldots, b_n) \in B \hotimes^\Pc B$ subject to the relation
\begin{multline*}
\mu(b_1, \ldots, b_{j-1}, d (\nu(v_1, \ldots, v_m)), b_{j+1}, \ldots, b_n) \approx\\
\sum_{i=1}^m (\mu \circ_j \nu)(b_1, \ldots, b_{j-1}, v_1, \ldots, v_{i-1}, dv_i, v_{i+1}, \ldots, v_m, b_{j+1}, \ldots, b_n).
\end{multline*}
The relation $\tau^2 = \gamma_B(\theta_\Pc, \id_B)$ in $B$ doesn't create any extra relation since we already have the relation $\eta_\Pc \otimes {d_B}^2 = \theta_\Pc \otimes \id_B$ in $B\otimes^{\Pc} B$. 
This gives the isomorphism $\Omega_{\Pc}B \cong B\otimes^{\Pc}_\tau V$.
\end{proof}

\subsubsection{Cotangent complex and André--Quillen cohomology}

We consider here that $\Rf = \Kf$ is a field of characteristic 0. 
It follows that every curved operad $\Pc$ is $\Sb$-split. 
Then Theorem C.39 in \cite{BMDC20} (see Theorem \ref{thm: model cat struct for curved algebras}) endows the category of $\Pc$-algebras with a model category structure. 
By Theorem 7.6.5 in \cite{pH03}, the category $\PalgA$ of $\Pc$-algebras over $A$ admits a model category structure such that the weak equivalences (resp. fibrations) are the graded quasi-isomorphisms (resp. strict surjections) when we forget the maps over $A$. 
We use the same notation $\PalgA$ for the category of $\Pc$-algebras over $A$ endowed with this model structure. 
Similarly, we prove in Theorem \ref{thm: model cat struct on modules} that the category of $A$-modules over $\Pc$ is endowed with a model category structure by transferring the model category structure from gr-dg $\Rf$-modules. 
Every objet in $\AbPA$ is fibrant since fibration are strict surjections.

\begin{thm}\label{thm: adjunction alg-mod}
Let $\Pc$ be a curved operad and $A$ be a $\Pc$-algebra. 
The following two functors
\[
A\hotimes_{-}^{\Pc}\Omega_{\Pc} -\  :\ \PalgA \rightleftharpoons \ModAPz \cong \AbPA\ :\ A\ltimes -
\]
form an adjoint pair which is a Quillen adjunction. 
The two functors represent the bi-functor of $A$-derivations, \emph{i.e.} there is an isomorphism of (graded) modules
\[
\Hom_{\ModAPz}\left(A\hotimes_{B}^{\Pc} \Omega_{\Pc}B, M\right) \cong \emph{Der}_{A}^0(B, M) \cong \Hom_{\PalgA}\left( B, A\ltimes M\right),
\]
for all $\Pc$-algebra $B$ over $A$ and for all $A$-module $M$.
\end{thm}

\begin{proof}
The part on the fact that we have an adjoint pair of functors follows from Propositions \ref{prop: derivation and Kahler differential form}, \ref{prop: free AP-mod}, \ref{prop: extension and restriction of scalars}, \ref{prop: derivation and abelian extension} and \ref{prop: abelian group object and module}.

To prove that we have a Quillen adjunction, it is enough by Lemma 1.3.4 of \cite{mH99} to show that the right adjoint functor $A\ltimes -$ preserves the fibrations and the acyclic fibrations. Since $A \ltimes M = A\oplus M$, the proof is straightforward form the definition of the two model category structures considered here.
\end{proof}

This adjunction provides therefore an adjunction on the level of the homotopy categories. We define the cohomology of an algebra over a curved operad $\Pc$ by means of its derived abelianisation.

\begin{defn}
\label{def: cotangent complex and AQ cohom}
The \emph{cotangent complex} of a $\Pc$-algebra $A$, denoted by $\Lb_A$, is defined to be the derived abelianisation $\Lb_A \coloneqq \Lb (A\hotimes_{-}^{\Pc}\Omega_{\Pc} -)(A)$ of $A$. 

Given an $A$-module $M$ in $\AbPA \cong \ModAPz$, the \emph{André--Quillen cohomology of $A$ with coefficients in $M$} is defined by
\[
H_{AQ}^n \left(A, M \right) \coloneqq \Hom_{\Ho(\ModAPz)}\left( \Lb_A, M[n]\right) = H^n \left(\Hom_{\ModAP}\left( \Lb_A, M\right), \partial\right).
\]
(The last equality follows from the fact that $\partial f = 0$ if and only if $f$ is a morphism of gr-dg modules, that is commutes with the predifferentials, and because the homotopy relation between two maps in the category of $A$-modules is given by the existence of an $h$ such that $\partial h$ is the difference of the two maps.)
\end{defn}

\subsection{Computing the André--Quillen cohomology}

We can make explicit one representative of the cotangent complex by means of a cofibrant replacement $R \xrightarrow{\sim} A$ in the category of $\Pc$-algebras over $A$. It is given by the class of
\[
\Lb_{R/A} \coloneqq A\hotimes_{R}^{\Pc}\Omega_{\Pc} R \in \Ho\left( \ModAP\right).
\]

As in \cite{jM14}, we use the bar-cobar resolutions $\Om_\alpha \B_\alpha A \xrightarrow{\sim} A$ described in Section \ref{section: resolution of an algebra over a curved operad} to describe the complex computing the André--Quillen cohomology.

\begin{thm}
\label{thm: cotangent complex}
Let $\Pc$ be a curved sa operad, $\Cc$ be a curved altipotent cooperad and $\alpha : \Cc \to \Pc$ be an operadic curved twisting morphism. Let $A$ be a $\Pc$-algebra. Assume that $R = \Om_\alpha \B_\alpha A$ be a cofibrant resolution of $A$. 
We denote by $\tau$ the composite $\B_\alpha A \xrightarrow{\Delta_{(1)} \circ \id_A} (\Cc \hcirc_{(1)} \Cc)(A) \xrightarrow{-(\alpha \circ_{(1)} \id_\Cc) \circ \id_A} \Pc(A, \B_\alpha A)$. 
Then the corresponding cotangent complex has the form
\[
\Lb_{\Om_\alpha \B_\alpha A/A} \cong (A\otimes^{\Pc}_\tau \B_\alpha A, d_\alpha),
\]
where $A\otimes^{\Pc}_\tau \B_\alpha A$ is defined before Proposition \ref{prop: quasi-free AP-mod} and the predifferential $d_{\alpha}$ is induced by $d_{\Pc(A, \B_\alpha A)}+\tau$.
\end{thm}

\begin{proof}
The $\Pc$-algebra $R = \Om_\alpha \B_\alpha A$ is quasi-free  on $\B_\alpha A$. 
It follows that the cotangent complex is isomorphic to
\[
\begin{array}{llll}
A\otimes_{R}^{\Pc}\Omega_{\Pc}R & \cong & A\otimes_{R}^{\Pc}(R\otimes^{\Pc}_\tau\B_\alpha A) & (\textrm{Proposition \ref{prop: derivation and Kahler differential form}})\\
& \cong & A\otimes^{\Pc}_\tau\B_\alpha A & (\textrm{Proposition \ref{prop: extension and restriction of scalars}}).
\end{array}
\]
The map $\tau$ is induced by the map $d_\alpha^l$ on the cobar construction and the isomorphism $A\otimes_{R}^{\Pc} (R\otimes^{\Pc}_\tau\B_\alpha A) \cong A\otimes^{\Pc}_\tau\B_\alpha A$ is obtained (on the lifts) by the maps $\eta_\Pc$ and the projection $\Om_\alpha \B_\alpha A$. 
The Theorem finally follows from the definition of the cobar predifferential $d_o$.
\end{proof}

\begin{prop}
\label{prop: AQ cochain complex}
Let $\Pc$ be a curved sa operad, $\Cc$ be a curved altipotent cooperad and $\alpha : \Cc \to \Pc$ be an operadic curved twisting morphism. Let $A$ be a $\Pc$-algebra. Assume that $R = \Om_\alpha \B_\alpha A$ be a cofibrant resolution of $A$. 
Let $M$ be an $A$-module. 
Then the André--Quillen cohomology is equal to
\[
H_{AQ}^n (A, M) =H^n \left(\Hom_{\ModAP}\left(A\hotimes^\Pc_\tau \B_\alpha A, M\right), \partial \right),
\]
where $\partial f = d_M \cdot f - (-1)^{|f|} f \cdot d_\alpha$.
\end{prop}

\begin{proof}
This follows from Definition \ref{def: cotangent complex and AQ cohom} and Theorem \ref{thm: cotangent complex}.
\end{proof}

We propose another description for the André--Quillen cohomology that we use in the next section.

\begin{thm}
\label{thm: AQ cochain complex convolution}
Let $\Pc$ be a curved sa operad, $\Cc$ be a curved altipotent cooperad and $\alpha : \Cc \to \Pc$ be an operadic curved twisting morphism. Let $A$ be a $\Pc$-algebra and $M$ be an $A$-module. 

There is an isomorphism of $\Rf$-modules
\begin{multline*}
\Hom_{\ModAPz}\left( A \hotimes^{\Pc}_\tau \B_\alpha A, M\right) \cong\\
\ker\left( \partial_\tau : \Hom_{\grModR}^0 \left( \B_\alpha A, U M\right) \to \Hom_{\grModR}^{1} \left( \B_\alpha A, U M\right) \right),
\end{multline*}
where $\partial_\tau \coloneqq \partial + \delta_\alpha^l$ with $\delta_\alpha^l(f)$ given by the composite
\begin{multline*}
\Cc(A) \xrightarrow{\Delta_{(1)} \circ \id_A} (\Cc \hcirc_{(1)} \Cc)(A) \xrightarrow{(\alpha \circ_{(1)} \id_{\Cc}) \circ \id_{A}} (\Pc \hcirc_{(1)} \Cc)(A) \to \Pc(A, \Cc(A))\\
\xrightarrow{(-1)^{|f|}\id_\Pc (\id_A, f)} \Pc(A, M) \xrightarrow{\gamma_M} M.
\end{multline*}
The predifferential $\partial_\tau$ satisfies
\[
{\partial_\tau}^2 = \delta_{\Theta + \partial \alpha + \frac12 [\alpha, \alpha]}^l = 0
\]
since $\alpha$ is an operadic twisting morphism, so it is a differential. 
It follows that when $R = \Om_\alpha \B_\alpha A$ is a cofibrant resolution of $A$, we have
\[
H_{AQ}^n \left(A, M \right) = H^n \left( \Hom_{\grModR} \left( \B_\alpha A, M\right), \partial_\tau\right).
\]
\end{thm}

\begin{proof}
The theorem follows from Propositions \ref{prop: AQ cochain complex} and \ref{prop: quasi-free AP-mod}. 
The equality ${\partial_\tau}^2 = \delta_{\Theta + \partial \alpha + \frac12 [\alpha, \alpha]}^l$ comes from the equalities:
\[
\partial^2 = \delta_{\Theta}^l,\qquad \partial \delta_\alpha^l + \delta_\alpha^l \partial = \delta_{\partial \alpha}^l \qquad \text{ and }\qquad {\delta_\alpha^l}^2 = \delta_{\frac12 [\alpha, \alpha]}^l
\]
using the coassociativity of $\Delta$ and the fact that $\alpha$ has degree $-1$ for the last equality.
\end{proof}

\section{The case of $\cCxi$-algebras and some perspectives}

\subsection{Bar construction of a $\Cxi$-algebra}

In this section, we fix $\Kf = \Rf = \Rb$. We make the bar construction associated with the operadic twisting morphism $\iota : \cCxa \to \cCxi = \hat \Omega \cCxa$ explicit. We recall that it is a functor
\[ \B_{\iota} : \cCxi \textsf{-algebras} \rightarrow \cCxa \textsf{-coalgebras}. \]

\begin{prop}
We have
\[
\B_{\iota} A \cong \left(\Sym\left(\hoplus_{k \geq 0}A[-k]\right), d_b \right),
\]
where $A[k]$ is the graded module shifted in cohomological degree as follows: $A[k]^n \coloneqq A^{n-k}$. 
In the specific case where $d_A = 0$ and the $\cCxi$-algebra structure is a $\cCx$-algebra structure given by a Lie bracket $\{-, -\}$ of cohomological degree $-1$ and an endomorphism $j$ of cohomological degree $0$, we have that
\begin{multline*}
d_b \left(a_1[k_1] \cdots a_n[k_n]\right) = \\
\sum_{i < j}\pm \alpha_{k_i, k_j} \{a_i, a_j\}[k_i + k_j] \cdot a_1[k_1] \cdots \check{a}_i[k_i] \cdots \check{a}_j[k_j] \cdots a_n[k_n]\quad +\\
\sum_{i} \pm j(a_i)[k_i-1] \cdot a_1[k_1] \cdots \check{a}_i[k_i] \cdots a_n[k_n],
\end{multline*}
where the sign is given by the Koszul sign rule for the elements $a_i[k_i]$.

In general, we have $d_b = d_{\cCx^\antishriek(A)}+d_\iota^r$ with
\begin{multline*}
d_\iota^r \left(a_1[k_1] \cdots a_n[k_n]\right) = \sum_{\substack{p+q=n+1\\ p, q \geq 1}} \sum_{\sigma \in Sh_{q, p-1}} \sum_{K} \alpha^{\sigma}_{k'_{j}, k''_{j}} \times\\
\hat{\jmath}_{k', k_{\sigma(q+1)}, \ldots , k_{\sigma(n)}}^{c}(\hat{\jmath}_{k''_{\sigma(1)}, \ldots , k''_{\sigma(q)}}^{c}(a_{\sigma(1)},\ldots, a_{\sigma(q)}), a_{\sigma(q+1)},\ldots, a_{\sigma(n)}),
\end{multline*}
where the coefficient $\alpha^{\sigma}_{k'_{j}, k''_{j}}$ is given in  Corollary \ref{cor: cCxa infinitesimal}. 

The endomorphism $d_b$ is a coderivation satisfying
\[
{d_b}^{2} = -(\theta_{\cCxa} \circ id_{\B_\iota A}) \cdot \Delta_{\B_\iota A}.
\]
\end{prop}

\begin{proof}
Theorem \ref{thm: cCxa} gives the isomorphism $\cCx^\antishriek (A) \cong \Sym\left(\hoplus_{k \geq 0}A[-k]\right)$. 
Then we recall from Section \ref{sec: bar constr alg} that the predifferential $d_{\cCxa (A)}$ is induced by the predifferentials on $\cCxa$ and on $A$, and
\begin{multline*}
d_{\iota}^{r} : \cCxa (A) \xrightarrow{\Delta_{(1)} \circ id_{A}} (\cCxa \hcirc_{(1)} \cCxa) (A) \xrightarrow{(id_{\cCxa} \circ_{(1)} \iota) \circ id_{A}}\\
\cCxa \hcirc \cCx (A) \xrightarrow{id_{\cCxa}\circ \gamma_{A}} \cCxa (A).
\end{multline*}
By calculating the coefficients by means of Corollary \ref{cor: cCxa infinitesimal}, we get the advertised formula.
\end{proof}

\subsection{A resolution of the $\infty$-morphisms between some $\cCxi$-algebras}

In this section, we again fix $\Kf = \Rf = \Rb$, we consider $\cCxi$-algebras and the operadic twisting morphism $\iota : \cCx^\antishriek \to \Om \cCx^\antishriek$.

\begin{lem}
\begin{enumerate}
\item
The multiplication by $i$ on the $\Rb$-algebra on $\Cb$ endows the $\Rb$-module $\Cb$ with a $\cCx$-algebra structure denoted $\gamma_\Cb$: the action of the shifted Lie bracket on $\Cb$ is zero and the action of $\jn$ is given by the multiplication by $i$. 
\item
Let $(A, \gamma_A)$ and $(B, \gamma_B)$ be $\cCxi$-algebras and let $f : A \to B$ be a $\cCxi$-algebra morphism. 
Then we endow $B$ with an $A$-module structure $\gamma_B^A$ defined by the restriction of the $\cCxi$-algebra structure $\gamma_B$ along $f$, that is
\[
\cCxi(A, B) \xrightarrow{\id_{\cCxi}(f, \id_B)} \cCxi(B, B) \twoheadrightarrow \cCxi(B) \xrightarrow{\gamma_B} B.
\]
\item
In the case where $B = \Cb$ and $f = 0$, we obtain that $\Cb$ is  endowed with the $A$-module structure $\gamma_\Cb^A$ defined on the generators of $\cCxi$ by
\[
\gamma_\Cb^A(j_{k_1, \ldots, k_n} \otimes a_1 \otimes \cdots \otimes c \otimes \otimes a_{n-1}) \coloneqq
\left\{\begin{array}{ll}
c & \text{when } (n, k_1) = (1, 0),\\
i \cdot c & \text{when } (n, k_1) = (1, 1),\\
0 & \text{otherwise.}
\end{array}\right.
\]
\end{enumerate}
\end{lem}

\begin{proof}
\begin{enumerate}
\item
The assertion follows for example from Lemma \ref{lem: cCx-algebra}.
\item
The described map defines an $A$-module map since $f$ is a $\cCxi$-algebra morphism between $\cCxi$-algebras.
\item
This is an explicitation of the previous definition in the specified  situation.
\end{enumerate}
\end{proof}

\begin{prop}
\label{prop: H0 AQ}
Let $(A, \gamma_A)$ be a $\cCxi$-algebra and $\B_\iota A$ its bar construction. 
The $\Rb$-module $\Cb$ is endow with its $\Cx$-algebra structure $\gamma_\Cb$ (which can also be viewed as a $\cCxi$ structure) and with its trivial $A$-module structure $\gamma_\Cb^A$ described in the previous lemma. 
We have the isomorphism of $\Rb$-modules
\[
Z^{0}_{AQ} ((A, \gamma_A), (\Cb, \gamma_\Cb^A)) \cong \infty\text{-}\mor((A, \gamma_A), (\Cb, \gamma_\Cb))
\]
where $Z^0_{AQ} ((A, \gamma_A), (\Cb, \gamma_\Cb^A))$ is 
\[
\ker\left( \partial_\tau : \Hom_{\grModRb}^0 \left( \B_\alpha A, U \Cb\right) \to \Hom_{\grModRb}^{1} \left( \B_\alpha A, U \Cb\right) \right)
\]
described in Theorem \ref{thm: AQ cochain complex convolution}.
\end{prop}

\begin{proof}
Let $\varphi \in \Hom_{\grModRb}^0(\B_\iota A, \Cb)$ such that $0 = \partial_\tau (\varphi) = \partial (\varphi) + \delta_\iota^l(\varphi)$, where $\delta_\iota^l(\varphi)$ is the composite
\begin{multline*}
\cCx^\antishriek(A) \xrightarrow{\Delta_{(1)} \circ \id_A} (\cCx^\antishriek \hcirc_{(1)} \cCx^\antishriek)(A) \xrightarrow{(\iota \circ_{(1)} \id_{\cCx^\antishriek}) \circ \id_{A}} (\cCxi \hcirc_{(1)} \cCx^\antishriek)(A) \to \\
\cCxi(A, \cCx^\antishriek(A)) \xrightarrow{\id_{\cCxi} (\id_A, \varphi)} \cCxi(A, \Cb) \xrightarrow{\gamma_\Cb^A} \Cb.
\end{multline*}
The map $\cCxi(A, \Cb) \xrightarrow{\gamma_\Cb^A} \Cb$ factors through $\cCxi(A, \Cb) \to I\circ \Cb \oplus \jn \otimes \Cb \to \Cb$. 
Similarly, the composition product $\gamma_\Cb : \cCxi(\Cb) \to \Cb$ factors through $\cCxi(\Cb) \to I\circ \Cb \oplus \jn \otimes \Cb \to \Cb$. 
We therefore have
\[
0 = \partial_\tau(\varphi) = \partial (\varphi) + \delta_\iota^l(\varphi) = \sum_{n \geq 0} \frac1{n!} l_n^\iota(\varphi, \ldots, \varphi)
\]
and $\varphi$ is an algebraic twisting morphism $\B_\iota A \to \Cb$, or equivalently an $\infty$-morphism $A \rightsquigarrow \Cb$.
\end{proof}

\begin{rem}
In the light of Theorem 4.1 in \cite{jM14}, when $A$ is the algebra encoding a complex structure on a formal pointed manifold (in particular, it is concentrated in degree 0), we get that the 0-th André--Quillen cohomology group computes formal holomorphic functions on it:
\begin{multline*}
H^{0}_{AQ} ((A, \gamma_A), (\Cb, \gamma_\Cb^A)) \cong \\
\ker \left(\partial_\tau : \Hom_{\grModRb}^0 \left( \B_\iota A, U\Cb\right) \to \Hom_{\grModRb}^{1} \left( \B_\iota A, U\Cb\right) \right) \cong \\
\infty\text{-}\mor((A, \gamma_A), (\Cb, \gamma_\Cb))
\end{multline*}
since $\B_\iota A$ is concentrated in non positive cohomological degrees (and therefore $\Hom_{\grModRb}^{-1} \left( \B_\iota A, U\Cb\right) = \{ 0\}$).
\end{rem}

\subsection{A $\Cb$-cdga structure on the André--Quillen complex in the context of $\cCxi$-algebras}

As in the two previous sections, we fix $\Kf = \Rf = \Rb$. 
In the following proposition, we use the usual associative and commutative structure on $\Cb$ to endow the André--Quillen cochain complex $\left(\Hom_{\grModRb}(\B_\iota A, \Cb), \partial_\tau \right)$ with a $\Cb$-cdga structure.

\begin{prop}
\label{prop: C-algebra structure}
Let $A$ be a $\cCxi$-algebra and $\B_\iota A$ its bar construction. 
We endow $\Cb$ with its trivial $A$-module structure $\gamma_\Cb^A$ obtained by restriction of $A \xrightarrow{0} \Cb$ (and described above). 
The dg module $\left(\Hom_{\grModRb}(\B_\iota A, \Cb), \partial_\tau \right)$ is endowed with a structure of $\Cb$-cdga denoted by $m$ and defined, for $\varphi_1$, \ldots, $\varphi_n$ in $\Hom_{\grModRb}(\B_\iota A, \Cb)$, by
\begin{multline*}
\B_\iota A \cong \cCx^\antishriek (A) \xrightarrow{\Delta_{\cCx^\antishriek} \circ \id_A} \cCxa(\B_\iota A) \twoheadrightarrow \sLie^\antishriek (\B_\iota A) \cong \Com^c(\B_\iota A)\\
\twoheadrightarrow \Com^c(n) \otimes_{\Sb_n} (\B_\iota A)^{\otimes n} \xrightarrow{\varphi_1 \otimes \cdots \otimes \varphi_n} \Com^c(n) \otimes_{\Sb_n} \Cb^{\otimes n} \xrightarrow{m_\Cb} \Cb,
\end{multline*}
where $m_\Cb$ is the usual (commutative) multiplication on $\Cb$ (we recall that $\Com^c(n) \cong \Rb$). 
The $\Cb$-action is given by the action of $\Cb$ on the codomain $\Cb$.
\end{prop}

\begin{proof}
The morphism $\Com(\Hom_{\grModRb}(\B_\iota A, \Cb)) \to \Hom_{\grModRb}(\B_\iota A, \Cb)$ described above defines a commutative and associative structure on $\Hom_{\grModRb}(\B_\iota A, \Cb)$ because the projection $\cCxa \twoheadrightarrow \sLie^\antishriek \cong \Com^c$ is a morphism of curved cooperads. 
The compatibility with the $\Cb$-action is direct to check.

It remains to show that the differential $\partial_\tau$ is a derivation with respect to the $\Com$-algebra structure $m$. 
We have
\[
\partial \cdot m(\varphi_1, \ldots, \varphi_n) = m(\varphi_1, \ldots, \varphi_n) \cdot d_{\B_\iota A}  = \sum_h m(\varphi_1, \ldots, \varphi_h \cdot d_{\B_\iota A}, \cdots, \varphi_n)
\]
since $\Delta_{\cCx^\antishriek}$ is coassociative and $d_{\cCx^\antishriek} = 0$. 
Moreover, using the fact that the action $\gamma_\Cb^A$ is trivial and the fact that $\iota_{|I} = 0$, we get that the trees appearing in the computation of $\delta_\iota^l (m(\varphi_1, \ldots, \varphi_n))$ have the form
\[
t_1 = {\tiny \vcenter{\xymatrix@R=10pt@C=2pt{
& \jmath_{k''_1, \ldots}^c &&\\
\jmath_{k_1, \ldots}^c & \jmath_{1}^c \ar@{-}[d] \ar@{-}[u] & \jmath_{k_{q_1}, \ldots}^c & \jmath_{k_{q_1+\cdots}, \ldots}^c \\
& \jmath_{0, \ldots, 0}^c \ar@{-}[d] \ar@{-}[ul] \ar@{-}[u] \ar@{-}[ur] \ar@{-}[urr] && \\
&&&}}}.
\]
Similarly, the trees appearing in the computation of $\sum_j m(\varphi_1, \ldots, \delta_\iota^l (\varphi_h), \ldots, \varphi_n)$ have the form
\[
t_2 = {\tiny \vcenter{\xymatrix@R=10pt@C=2pt{
\jmath_{k_1, \ldots}^c & \jmath_{k''_1, \ldots}^c & \jmath_{k_{q_1}, \ldots}^c & \jmath_{k_{q_1+\cdots}, \ldots}^c \\
& \jmath_{0, \ldots, 0}^c \ar@{-}[d] \ar@{-}[ul] \ar@{-}[u] \ar@{-}[ur] \ar@{-}[urr] && \\
& \jmath_{1}^c \ar@{-}[u] \ar@{-}[d] &&\\
&&&}}}.
\]
Because $\gamma_\Cb^A$ is $\Cb$-linear, applying the maps $\varphi_j$ and the action $\gamma_\Cb^A$, the two corresponding trees give the same result. 
It remains to compute the coefficients in front of these trees. These coefficients appear after applying the decomposition map two times. 
We shortly recall the formula (given in Theorem \ref{thm: cCxa}) here
\begin{multline*}
\Delta^\theta \left( \jmath_{k_1, \ldots, k_n}^c\right) =\\
\sum \frac{1}{p!} \beta^{\sigma}_{k'_{j}, k''_{j}} \times \left( \jmath_{l'_1, \ldots , l'_p}^{c} ; \jmath_{k''_{\sigma(1)}, \ldots , k''_{\sigma(q_1)}}^{c}, \jmath_{k''_{\sigma(q_1+1)}, \ldots , k''_{\sigma(q_1+q_2)}}^{c} , \ldots \right)^{\sigma^{-1}}
\end{multline*}
where the sum $\sum$ is given by $\sum_{q_{1}+\cdots +q_{p}= n} \sum_{\sigma \in Sh_{q_{1}, \ldots , q_{p}}} \sum_{k'_{j}+k''_{j} = k_{j}}$, the integers $l'_i$ are defined by $l'_i := k'_{\sigma(q_1+\cdots +q_{i-1}+1)}+ \cdots +k'_{\sigma(q_1+\cdots +q_i)}$ and the number $\beta^{\sigma}_{k'_{j}, k''_{j}}$ is equal to
\[
\beta^{\sigma}_{k'_{j}, k''_{j}} \coloneqq \sgn_{k_{1}, \ldots , k_{n}}\sigma \times \varepsilon_{k'_{j}, k''_{j}}^{\sigma} \times \prod_{i=1}^p \alpha_{k'_{\sigma(q_1+\cdots +q_{i-1}+1)},\, \ldots , k'_{\sigma(q_1+\cdots +q_i)}}
\]
where $\sgn_{k_{1}, \ldots , k_{n}}\sigma$ is the signature of the restriction of $\sigma$ to the indices $j$ such that $k_{j}$ is odd (after relabeling the remaining $\sigma(j)$ in a way that the order of the $\sigma(j)$ does not change) and where
\[
\left\{ \begin{array}{lcl}
\varepsilon_{k'_{j},\, k''_{j}}^{\sigma} & \coloneqq & (-1)^{\sum_{i=1}^{n} k_{\sigma(i)}''(k_{\sigma(i+1)}' + \cdots + k_{\sigma(n)}')}\\
\alpha_{k'_{1},\, \ldots ,\, k'_{q}} & \coloneqq & \sum_{\sigma' \in Sh_{k'_{1},\, \cdots ,\,k'_{q}}} \sgn \sigma',\, \textrm{with convention } \alpha_{0,\, \ldots ,\, 0} := 1.
\end{array}\right.
\]
First we compute the coefficient in front of the term $t_1$. 
We calculate $\Delta^\theta \left( \jmath_{k_1, \ldots, k_n}^c\right)$ and compute the terms having the form $(\jmath_{0, \ldots, 0}^c ; \jmath_{k_1'', \ldots}^c, \ldots)$. 
We get the sum
\begin{multline*}
\sum_{q_{1}+\cdots +q_{p}= n} \sum_{\sigma \in Sh_{q_{1}, \ldots , q_{p}}} \frac{1}{p!} \sgn_{k_{1}, \ldots , k_{n}}\sigma \times \varepsilon_{0, k_{j}}^{\sigma} \times \prod_{i=1}^p 1 \times\\
\left( \jmath_{0, \ldots , 0}^{c} ; \jmath_{k_{\sigma(1)}, \ldots , k_{\sigma(q_1)}}^{c}, \jmath_{k_{\sigma(q_1+1)}, \ldots , k_{\sigma(q_1+q_2)}}^{c} , \ldots \right)^{\sigma^{-1}}
\end{multline*}
\[
= \sum_{q_{1}+\cdots +q_{p}= n} \sum_{\sigma \in Sh_{q_{1}, \ldots , q_{p}}} \frac{1}{p!} \sgn_{k_{1}, \ldots , k_{n}}\sigma \times \left( \jmath_{0, \ldots , 0}^{c} ; \jmath_{k_{\sigma(1)}, \ldots , k_{\sigma(q_1)}}^{c}, \ldots \right)^{\sigma^{-1}}.
\]
Then we apply the infinitesimal decomposition on a term $\jmath_{k_{\sigma(q_1+\ldots +q_s+1)}, \ldots}^{c}$ and compute the terms having the form of $t_1$. 
In front of the tree $(j_1^c ; \jmath_{k''_{\sigma(q_1+\ldots +q_s+1)}, \ldots}^c)$, we obtain the coefficient
\[
\varepsilon_{k_{\sigma(q_1+\ldots +q_s+j)}', k_{\sigma(q_1+\ldots +q_s+j)}''}^{\id}
\]
corresponding to the Koszul sign coming from the fact that the tree $\jn$ goes down and where only one $k'_j$ is equal to 1 and the others are equal to 0. 
Finally, every tree $t_1$ which appears has a coefficient equal to the sign given by the Koszul sign rule (due to the fact that we permute the generators $\jn$).

We now compute the coefficient in front of the term $t_2$. 
In the sum $\Delta^\theta \left( \jmath_{k_1, \ldots, k_n}^c\right)$, the coefficient in front of the tree $(\jmath_1^c ; \jmath_{k_1, \ldots, k_h-1, \ldots, k_n}^c)$ is $\varepsilon_{k_j', k_j''}^{\id}$ corresponding to the Koszul sign coming from the fact that the tree $\jn$ goes down. (All the  other coefficients in the formula are equal to 1.) 
In the decompositions $k'_{j}+k''_{j} = k_{j}$, only $k'_h$ is equal to 1 and the others $k_j'$ are equal to 0. 
Then applying $\Delta^\theta$ to $\jmath_{k_1, \ldots, k_h-1, \ldots, k_n}^c$ and considering the terms $(\jmath_{0, \ldots, 0}^c ; \ldots)$ gives the sum
\[
\sum_{q_{1}+\cdots +q_{p}= n} \sum_{\sigma \in Sh_{q_{1}, \ldots , q_{p}}} \frac{1}{p!} \sgn_{k_{1}, \ldots, k_h-1, \ldots, k_{n}}\sigma \times \left( \jmath_{0, \ldots , 0}^{c} ; \jmath_{k_{\sigma(1)}, \ldots , k_{\sigma(q_1)}}^{c}, \ldots \right)^{\sigma^{-1}}.
\]
Again, every tree $t_2$ which appears has a coefficient equal to the sign given by the Koszul sign rule due to the fact that we permute the generators $\jn$. 

It follows form the previous explanations that
\[
\delta_\iota^l (m(\varphi_1, \ldots, \varphi_n)) = \sum_h m(\varphi_1, \ldots, \delta_\iota^l (\varphi_h), \ldots, \varphi_n). 
\]
As a consequence, the differential $\partial_\tau$ is a derivation with respect to the $\Com$-algebra structure $m$. This concludes the proof.
\end{proof}

\subsection{Relation with dg algebras with entire functional calculus}

We conclude this section by showing that the functor given by the André--Quillen cochains from $\cCxi$-algebras with $\infty$-morphisms to EFC-dga in the sense of \cite[Definitions 1.4 and 2.10]{Pri20} is fully faithful.

We recall the definition of dg algebra with entire functional calculus.

\begin{defn}
\begin{enumerate}
\item
Let $\mathsf{EFC}$ the category with objects $\{ \Cb^n \}_{n \geq 0}$ and morphisms consisting of complex-analytic maps.
\item
A $\Cb$-algebra $A$ is a \emph{EFC $\Cb$-algebra} if it endowed with a product-preserving set-valued functor $\Psi  : \mathsf{EFC} \to \mathsf{Set}$.
\item
A chain complex $(A, d_A)$ of $\Qb$-vector spaces is an \emph{EFC-dga} if it is a dg commutative algebra such that $Z^0 A = \ker(d_A : A^0 \to A^1)$ is an EFC $\Cb$-algebra.
\end{enumerate}
\end{defn}

\begin{thm}
\label{thm: link EFC}
The functor
\[ \begin{array}{rccl}
C_{AQ}^\bullet : & (\cCxi\textsf{-alg.}, \infty\text{-}\mor) & \to & \textsf{EFC-dga},\\
& (A, \gamma_A, d_A) & \mapsto & \left(\Hom_{\grModRb}(\B_\iota A, \Cb), \partial_\tau \right)
\end{array} \]
is well-defined and fully faithful.
\end{thm}

\begin{proof}
Proposition \ref{prop: C-algebra structure} ensures that $C_{AQ}^\bullet (A)$ is endowed with a $\Cb$-cdga structure. 
By Proposition \ref{prop: H0 AQ}, we obtain that $Z^0 C_{AQ}^\bullet (A) \cong \infty\text{-}\mor(A, \Cb)$. 
From \cite[Theorem 4.1]{jM14}, we know that complex-analytic (or equivalently holomorphic) functions on a formal pointed manifold endowed with a complex structure coincide with $\infty$-morphisms $A \rightsquigarrow \Cb$, where $A$ is the $\Cxi$-algebra encoding the (formal) complex structure. Because we can compose $\infty$-morphisms, we get that $Z^0 C_{AQ}^\bullet (A)$ is endowed with an EFC $\Cb$-algebra structure. 
Indeed, $B_\iota$ is a right adjoint functor so $\B_\iota (\Cb^n) \cong (\B_\iota \Cb)^n$ and given a complex-analytic function $f : \Cb^n \to \Cb$, its Taylor series at 0 is given by an $\infty$-morphism $T_0^f :\B_\iota \Cb^n \to \B_\iota \Cb$. 
Given $\cCx^\antishriek$-coalgebra morphisms $\varphi_1, \ldots, \varphi_n : \B_\iota A \to \B_\iota \Cb$, the composition $T_0^f \cdot (\varphi_1, \ldots, \varphi_n) \cdot \Delta^{n-1}_{diag}$ is the desired $\infty$-morphism (where $\Delta^{n-1}_{diag} : \B_\iota A \to (\B_\iota A)^n$ is the set diagonal morphism). 
It is compatible with the composition of complex-analytic functions since the equivalence given \cite[Theorem 4.1]{jM14} is functorial. 
The association $C_{AQ}^\bullet$ is functorial since the bar construction and the Hom-functor are and  since the compatibility with the differentials follows from the fact that morphisms of algebras preserve the algebra structures and the predifferentials. 
It follows that $C_{AQ}^\bullet$ is a well-defined functor.

The functor $C_{AQ}^\bullet$ is fully faithful because
\[
\Hom_{\grModRb}(\B_\iota A, \Cb) \cong \Hom_{\grModRb}(\B_\iota A, \Rb) \oplus \Hom_{\grModRb}(\B_\iota A, \Rb)
\]
and $\Hom_{\grModRb}(-, \Rb)$ is fully faithful. (We recall that we consider $\infty$-morphisms between $\cCxi$-algebras, that is $\cCx^\antishriek$-coalgebra morphisms between the bar constructions.)
\end{proof}

\appendix

\section{Model category structures}
\label{appendix: model cat struct}

In this section, we recall the model category structure on gr-dg modules, on curved operads and on algebras over a curved operad presented in \cite[Appendix C]{BMDC20}, and we extend it by replacing the category of gr-dg $\Kf$-modules by the category of gr-dg $\Rf$-modules (where $\Rf$ is a unital $\Kf$-algebra).

\subsection{Reminder on a model category structures on gr-dg modules}

Let $\Kf$ be a field. 
In \cite[Appendix C]{BMDC20}, the authors endow the category of gr-dg $\Kf$-modules with a proper cofibrantly generated symmetric monoidal model category structure. 
We recall this structure.

We denote by $\Kf^{(q)}$ the complete $\Kf$-module given by
\[ \Kf^{(q)} = F_0 \Kf^{(q)} = F_q \Kf^{(q)} \supset F_{q+1} \Kf^{(q)} = 0. \]
(The filtration is induced by a graduation concentrated in degree $q$.) 
The notation $\Kf_n^{(q)}$ means that we consider it in degree $n$ within a (complete) filtered complex. 

We define, for all $n\in \Zb$ and $q \in \Nb$, the complete gr-dg $\Kf$-modules
\begin{align*}
\hat \Zc^{0, \infty}_{q, n} \coloneqq & \left( \Kf^{(q)}_n \xrightarrow{1} \Kf^{(q)}_{n-1} \xrightarrow{1} \Kf^{(q+1)}_{n-2} \xrightarrow{1} \Kf^{(q+1)}_{n-3} \xrightarrow{1} \Kf^{(q+2)}_{n-4} \to \dots \right)^{\wedge}\\
= & \hoplus_{k \in \Nb} \Kf_{n - k}^{\left(q + \lfloor \frac{k}{2} \rfloor \right)},
\end{align*}
and 
\begin{align*}
\hat \Zc^{1, \infty}_{q, n} \coloneqq & \left( \Kf^{(q)}_n \xrightarrow{1} \Kf^{(q+1)}_{n-1} \xrightarrow{1} \Kf^{(q+1)}_{n-2} \xrightarrow{1} \Kf^{(q+2)}_{n-3} \xrightarrow{1} \Kf^{(q+2)}_{n-4} \to \dots \right)^{\wedge}\\
= & \hoplus_{k \in \Nb} \Kf_{n - k}^{\left(q + \lceil \frac{k}{2} \rceil \right)}.
\end{align*}
We also define the complete gr-dg $\Kf$-module
\[
\hat \Bc^{1, \infty}_{q, n} \coloneqq \hat \Zc^{0, \infty}_{q, n+1} \oplus \hat \Zc^{0, \infty}_{q+1, n}.
\]
We denote by $\varphi^\infty_{q, n} : \hat \Zc^{1, \infty}_{q, n} \to \hat \Bc^{1, \infty}_{q, n}$ the morphism of complete gr-dg $\Kf$-modules defined by the following diagram
\[ \xymatrix@C=16pt{
& \Kf_n^{(q)} \ar[r] \ar[d]^{{\tiny \begin{pmatrix} 1\\ 1 \end{pmatrix}}} & \Kf_{n-1}^{(q+1)} \ar[r] \ar[d]^{{\tiny \begin{pmatrix} 1\\ 1 \end{pmatrix}}} & \Kf_{n-2}^{(q+1)} \ar[d]^{{\tiny \begin{pmatrix} 1\\ 1 \end{pmatrix}}} \ar[r] & \cdots\\
\Kf_{n+1}^{(q)} \ar[r] & \Kf_n^{(q)} \oplus \Kf_n^{(q+1)} \ar[r] & \Kf_{n-1}^{(q+1)} \oplus \Kf_{n-1}^{(q+1)} \ar[r] & \Kf_{n-2}^{(q+1)} \oplus \Kf_{n-2}^{(q+2)} \ar[r] & \cdots.
} \]
We fix the sets
\begin{align*}
I_0^\infty & \coloneqq \{ \varphi^\infty_{q, n} : \hat \Zc^{1, \infty}_{q, n} \to \hat \Bc^{1, \infty}_{q, n} \}_{n \in \Zb,\, q \in \Nb}\\
J_0^\infty & \coloneqq \{ 0 \to \hat \Zc^{0, \infty}_{q, n} \}_{n \in \Zb,\, q \in \Nb}.
\end{align*}

We say that a map $p : (X,\, F) \to (Y,\, F')$ is \emph{strict} when it satisfies $p(F_qX) = p(X) \cap F'_qY$ for all $q$. When $p$ is a surjection, this means that $p(F_qX) = F'_qY$ for all $q$. 

Then (Theorems C.15 and Proposition C.21 in \cite{BMDC20}), the category of (complete) gr-dg $\Kf$-modules is endowed with a symmetric monoidal model category structure where:
\begin{enumerate}
\item
weak equivalences are graded quasi-isomorphisms (that is a quasi-isomorphism after application of the graded functor),
\item
fibrations are strict surjections, and
\item
$I_0^\infty$ and $J_0^\infty$ are the sets of generating cofibrations and generating acyclic cofibrations respectively.
\end{enumerate}

Let $\Rf$ be a unital $\Kf$-algebra. 

\begin{prop}
\label{prop: new model structure on complete gr-dg A-modules}
By means of the free-forgetful adjunction
\[
\xymatrix{\Rf \otimes_\Kf - : \ModKgr \ar@<.5ex>@^{->}[r] & \ModAgr : U \ar@<.5ex>@^{->}[l]}
\]
we can transfer the model category structure presented above on the category $\ModRgr$ of gr-dg $\Kf$-modules to the category $\ModRgr$ of gr-dg $\Rf$-modules. 
Weak equivalences (resp. fibrations) are graded  quasi-isomorphisms (resp. strict surjections) and the model structure is right proper.
\end{prop}

\begin{proof}
We apply Theorem 3.3 in \cite{sC95} to the above adjunction. 
The category $\ModRgr$ admits limits and colimits which are computed in sets. 
It is enough to remark that $U$ preserve all colimits so filtered $\aleph_1$-colimits and that it maps relative $\Rf \otimes_\Kf J_0^\infty$-cell complexes to weak equivalences ($\Kf$ is a field so $\Rf$ is a free $\Kf$-module).

Every object is fibrant, so by \cite[Corollary 13.1.3]{pH03} the model category structure is right proper.
\end{proof}

\subsection{Model category structures on curved operads and algebras over them}

Applying Proposition \ref{prop: new model structure on complete gr-dg A-modules} to the ring $\Kf [\Sb_m]$, $m \in \Nb$, we obtain a right proper cofibrantly generated model category structure on the category $\Mc_m$ of complete gr-dg $\Kf[\Sb_m]$-modules. 
Considering this collection $(\Mc_m,\, W_m,\, I_m,\, J_m)_{m \in \Nb}$ of cofibrantly generated model category structures, we have that the product
\[ \SMod \coloneqq \left(\prod_{m \in \Nb} \Mc_m,\, \prod_{m \in \Nb} W_m,\, \prod_{m \in \Nb} I_m \right) \]
is also a cofibrantly generated model category structure (see for example \cite[11.6]{pH03}). A morphism $f : M \to N$ in $\Mc$ is a weak equivalence (resp. fibration) if the underlying collection of morphisms $\{ M(m) \to N(m)\}_{m \in \Nb}$ consists of weak equivalences (resp. fibrations). Moreover the set $\texttt{I}$ (resp. $\texttt{J}$) of generating cofibrations (resp. acyclic cofibrations) can be described explicitly as follows:
\[
\texttt{I} = \{ \hat \Zc^{1, \infty}_{q, n}(m) \to \hat \Bc^{1, \infty}_{q, n}(m)\} \text{ and } \texttt{J} = \{ 0 \to \hat \Zc^{0, \infty}_{q, n}(m)\},
\]
where $\hat \Zc^{k, \infty}_{q, n}(m)$ (resp. $\hat \Bc^{1, \infty}_{q, n}(m)$) is the complete gr-dg $\Sb$-module obtained by the complete free gr-dg $\Kf[\Sb_m]$-module $\hat \Zc^{k, \infty}_{q, n} \otimes \Kf[\Sb_m]$ (resp. $\hat \Bc^{1, \infty}_{q, n} \otimes \Kf[\Sb_m]$) put in arity $m$. 
We denote by $\Wc$ the subcategory of weak equivalences. 
Notice that the domains of elements of $\texttt{I}$ or $\texttt{J}$ are small ($\aleph_1$-small) in the category $\SMod$.

The free curved operad is described in \cite[Section 2]{BMDC20} and is described as the left adjoint of the free-forgetful adjunction
\[
\xymatrix{\cfree : \SMod \ar@<.5ex>@^{->}[r] & \mathsf{Curved\ operads} : U, \ar@<.5ex>@^{->}[l]}
\]
where $\cfree(M, d_M) = \left(\Tc(M \oplus \vartheta I)/\left(\im({d_M}^2 - [\vartheta,\, -])\right),\, \bar d_M,\, \bar \vartheta\right)$ for $\vartheta$ is a formal parameter in homological degree $-2$ and in filtration degree $1$, the map $\bar d_M$ is the derivation induced by $d_M$ and $\left(\im({d_M}^2 - [\vartheta,\, -])\right)$ is the ideal generated by the image of the map ${d_M}^2 - [\vartheta,\, -] : M \to \Tc(M \oplus \vartheta I)$. 
By means of Theorem 3.3 of \cite{sC95}, we transfer along this adjunction the cofibrantly generated model category structure on $\SMod$ to the category of curved operads.

\begin{thm}[Theorem C.30 in \cite{BMDC20}]
\label{thm: new model structure on curved operads}
The category of curved operads is endowed with a cofibrantly generated model category structure where the generating (resp. acyclic) cofibrations are $\cfree \texttt{I}$ (resp. $\cfree \texttt{J}$). A map $f : \Pc \to \Qc$ is a \emph{weak equivalence} if and only if, in every arity, it is a graded quasi-isomorphism of gr-dg $\Sb$-modules.
Moreover $(\cfree,\, U)$ is a Quillen pair with respect to the cofibrantly model structures.
\end{thm}

Following Hinich \cite{vH97}, we similarly endow the category $\Palg$ of algebras over an $\Sb$-split curved operad $(\Pc,\, d,\, \theta)$ with a model category structure. We apply Theorem 3.3 of \cite{sC95} to the free-forgetful functor adjunction given in \cite[Proposition C.35]{BMDC20} between the categories of gr-dg $\Kf$-modules and $\Palg$
\[
\xymatrix{F_\Pc : \ModKgr \ar@<.5ex>@^{->}[r] & \Palg : \#, \ar@<.5ex>@^{->}[l]}
\]
where
\[ (V,\, d_V) \mapsto F_{\Pc}(V,\, d_V) \coloneqq \left(\Pc(V)/\left(\im\left(\eta \otimes ({d_{V}}^2) - \theta \otimes \id_V\right) \right),\, d_{\overline{\Pc(V)}}\right). \]

\begin{thm}[Theorem C.39 in \cite{BMDC20}]
\label{thm: model cat struct for curved algebras}
Let $(\Pc,\, d,\, \theta)$ be a curved operad which is $\Sb$-split. The category $\Palg$ of $(\Pc,\, d,\, \theta)$-algebras is a cofibrantly generated model category with generating cofibrations $F_\Pc (I_0^\infty)$ and generating acyclic cofibrations $F_\Pc(J_0^\infty)$. The weak equivalences are the maps that are graded quasi-isomorphisms. 
\end{thm}

By the same arguments, we prove that the category of $A$-module over $\Pc$, for $\Pc$ a curved operad and $A$ a $\Pc$-algebra, is endowed with a similar model structure.

\begin{thm}
\label{thm: model cat struct on modules}
The category $\ModAPz$ of $A$-modules over $\Pc$, where $\Pc$ is a curved operad and $A$ a $\Pc$-algebra, is endowed with a cofibrantly generated model category structure such that
\begin{itemize}
\item
the weak equivalences are the graded quasi-isomorphisms,
\item
the fibrations are the strict surjections.
\end{itemize}
\end{thm}

\begin{proof}
We transfer the model category structure from the category of gr-dg $\Kf$-modules to $A$-modules over $\Pc$ by means of Theorem 3.3 in \cite{sC95} using the free-forgetful adjunction given in Proposition \ref{prop: free AP-mod}. 
The category $\ModAP$ is complete and cocomplete by Lemma \ref{lem: U creates limits and colimits}. 
It follows that the forgetful functor $U$ preserves all colimits; in particular it preserves filtered $\aleph_1$-colimits. 
We finally have to show that $U$ maps relative $A\hotimes^\Pc(J_0^\infty)$-cell complexes to weak equivalences. 
We have $\Gr (A\hotimes^\Pc -) \cong \Gr A \otimes^{\Gr \Pc} \Gr (-) \cong U_{\Gr \Pc}(\Gr A) \otimes \Gr (-)$ which is exact since $\Kf$ is a field (so $U_{\Gr \Pc}(\Gr A)$ is flat). 
This concludes the proof.
\end{proof}

\bibliographystyle{alphaabbr}
\bibliography{../bib}

\end{document}